\documentclass[12pt]{amsart}
\usepackage{amssymb,latexsym,amscd,verbatim}
\usepackage[all]{xy}

\swapnumbers
\newtheorem{thm}{Theorem}[subsection]
\newtheorem{propose}[thm]{Proposition}
\newtheorem{lemma}[thm]{Lemma}
\newtheorem{cor}[thm]{Corollary}
\theoremstyle{definition}
\newtheorem{defn}[thm]{Definition}
\newtheorem{notation}[thm]{Notation}
\newtheorem{remark}[thm]{Remark}
\newtheorem{remarks}[thm]{Remarks}
\newtheorem{example}[thm]{Example}
\newtheorem{examples}[thm]{Examples}
\numberwithin{equation}{section}
\newcounter{spec}
\newenvironment{thlist}{\begin{list}{\rm{(\roman{spec})}}%
{\usecounter{spec}\labelwidth=20pt\itemindent=0pt\labelsep=10pt}}%
{\end{list}}

\renewcommand{\d}{\mbox{\LARGE $\cdot $}}
\newcommand{\Xs}{X_{\d}}            
\renewcommand{\hat}{\widehat}
\newcommand{\Spec}{\operatorname{Spec}} 
\newcommand{\Hom}{\operatorname{Hom}}      
\newcommand{\Map}{\operatorname{Map}}      
\newcommand{\Ext}{\operatorname{Ext}}      
\newcommand{\Biext} {\operatorname{Biext}}  
\newcommand{\DM}{\operatorname{DM}}          
\newcommand{\M}{\mathcal{M}_1}   
\newcommand{\MM}{\mathcal{MM}}  
\renewcommand{\1}{{}_{\leq 1}}
\newcommand{\ncM}{\mathcal{M}_{\rm nc}} 
\newcommand{\anM}{\mathcal{M}_{\rm anc}} 
\newcommand{\Shv}{\operatorname{Shv}}
\newcommand{\PST}{\operatorname{PST}}
\newcommand{\EST}{\operatorname{EST}}
\newcommand{\HI}{\operatorname{HI}}
\newcommand{\CH}{\operatorname{CH}}

\newcommand{\car}{\operatorname{char}}
\newcommand{\Div}{\operatorname{Div}}

\newcommand{\sext}{\text{${\mathcal E}xt\,$}}  
\newcommand{\shom}{\text{${\mathcal H}om\,$}}  
\newcommand{\ihom}{{\rm\underline{Hom}}}  
\newcommand{\oo}{\operatornamewithlimits{\otimes}\limits}

\renewcommand{\P}{\mathbb{P}}   
\newcommand{\Aff}{\mathbb{A}}   
\newcommand{\A}{{\rm\underline A}}      
\newcommand{\E}{{\rm\underline E}}      
\newcommand{\sH}{\mathcal{H}}
\newcommand{\sO}{\mathcal{O}}

\newcommand{\C}{\mathbb{C}}     
\newcommand{\Q}{\mathbb{Q}}     
\newcommand{\Z}{\mathbb{Z}}     
\newcommand{\N}{\mathbb{N}}
\newcommand{\G}{\mathbb{G}}     
\newcommand{\HH}{\mathbb{H}}    
\newcommand{\EExt}{{\rm \mathbb{E}xt}} 
\newcommand{\R}{\mathbb{R}}     
     
\newcommand{\uG}{\underline{G}}

\newcommand{\uR}{\underline{R}}
\newcommand{\uT}{\underline{T}}

\newcommand{\im}{\operatorname{Im}}        
\renewcommand{\ker}{\operatorname{Ker}}  
\newcommand{\coker}{\operatorname{Coker}} 
\newcommand{\gr}{\operatorname{gr}}        
\newcommand{\Pic}{\operatorname{Pic}}     
\newcommand{\RPic}{\operatorname{RPic}}     
\newcommand{\Alb}{\operatorname{Alb}}     
\newcommand{\LAlb}{\operatorname{LAlb}}     
\newcommand{\SAb}{\operatorname{SAb}}
\newcommand{\AbS}{\operatorname{AbS}}

\newcommand{\LA}[1]{\mbox{${\rm L}_{#1}{\rm Alb}$}}
\newcommand{\RA}[1]{\mbox{${\rm R}^{#1}{\rm Pic}$}}

\newcommand{\Tot}{\operatorname{Tot}}     
\newcommand{\tot}{\operatorname{tot}}     
\newcommand{\NS}  {\operatorname{NS}}      
\newcommand{\rank}{\operatorname{rank}}    
\newcommand{\tors}{{\operatorname{tors}}}
\newcommand{\qi}{{\rm q.i.}\,}      

\newcommand{\by}[1]{\stackrel{#1}{\rightarrow}}
\newcommand{\longby}[1]{\stackrel{#1}{\longrightarrow}}
\newcommand{\vlongby}[1]{\stackrel{#1}{\mbox{\large{$\longrightarrow$}}}}
\newcommand{\iso}{\stackrel{\sim}{\longrightarrow}}

\renewcommand{\tilde}{\widetilde}
\newcommand{\df}{\mbox{\,${:=}$}\,}
\newcommand{\ie}{{\it i.e. }}
\newcommand{\cf}{{\it cf. }}
\newcommand{\eg}{{\it e.g. }}
\newcommand{\resp}{{\it resp. }}
\newcommand{\loccit}{{\it loc. cit. }}
\newcommand{\et} {{\rm \acute{e}t}}
\newcommand{\eh} {{\rm \acute{e}h}}
\newcommand{\Zar}{{\rm Zar}}
\newcommand{\Nis}{{\rm Nis}}
\newcommand{\cdh}{{\rm cdh}}
\newcommand{\an}{{\rm an}}
\newcommand{\fr}{{\rm fr}}
\newcommand{\tor}{{\rm tor}}
\newcommand{\dis}{{\rm dis}}
\newcommand{\red}{{\rm red}}
\newcommand{\tf}{{\rm tf}}
\newcommand{\eff}{{\rm eff}}
\newcommand{\gm}{{\rm gm}}

\renewcommand{\bar}{\overline}
\newcommand{\into}{\hookrightarrow}
\renewcommand{\implies}{\mbox{$\Rightarrow$}}
\newcommand{\veq}{\mbox{\large $\parallel$}}  
\newcommand{\sQ}{\mbox{\scriptsize{$\Q$}}}   

\newcommand{\limdir}[1]{\mathop{\rm
lim}_{\buildrel\longrightarrow\over{#1}}}
\newcommand{\liminv}[1]{\mathop{\rm
lim}_{\buildrel\longleftarrow\over{#1}}}
\newcommand{\onto}{\mbox{$\to\!\!\!\!\to$}}

\newcommand{\boxtensor}{\def\boxtimesten{\Box\kern-7.59pt\raise1.2pt
\hbox{$\times$} }}                                  

\newcounter{elno}                   

\newcommand{\cA}{\mathcal{A}}
\newcommand{\cB}{\mathcal{B}}

\newcommand{\cE}{\mathcal{E}}
\newcommand{\cF}{\mathcal{F}}

\newcommand{\cH}{\mathcal{H}}

\newcommand{\cM}{\mathcal{M}}

\newcommand{\cO}{\mathcal{O}}
\newcommand{\cP}{\mathcal{P}}

\newcommand{\cS}{\mathcal{S}}
\newcommand{\cT}{\mathcal{T}}
\newcommand{\cU}{\mathcal{U}}

\renewcommand{\phi}{\varphi}
\renewcommand{\epsilon}{\varepsilon}

\makeindex

\begin{document}

\title{On the derived category of 1-motives, I}
\author{Luca Barbieri-Viale}
\address{Dipartimento di Matematica Pura e Applicata, Universit\`a degli Studi di
Padova\\ Via Trieste, 63\\35121 Padova\\ Italy}
\email{barbieri@math.unipd.it}
\author{Bruno Kahn}
\address{Institut de Math{\'e}matiques de Jussieu\\
175--179 rue du Chevaleret\\75013 Paris\\ France}
\email{kahn@math.jussieu.fr}
\date{June 8, 2007}

\begin{abstract}
We consider the category of Deligne 1-motives over a perfect field $k$ of
exponential characteristic $p$ and its derived category for a suitable exact
structure after inverting $p$. As a first result, we provide a fully faithful
embedding into an \'etale version of Voevodsky's triangulated category of
geometric motives. Our second main result is that this full embedding ``almost"
has a left adjoint, that we call $\LAlb$.  Applied to the motive of a variety we
thus get a bounded complex of 1-motives, that we compute fully for smooth
varieties and partly for singular varieties. As an application we give motivic
proofs of Ro\v\i tman type theorems (in characteristic $0$).
\end{abstract}
\maketitle
\tableofcontents

\section*{Introduction}

While Grothendieck's construction of pure motives associated to smooth
projective varieties over a field $k$ is now classical
\cite{MA,demazure,kleiman,scholl}, the construction of mixed motives
associated to arbitrary $k$-varieties is still largely work in progress.
In this direction, the first concrete step was taken by Deligne in
\cite{D} where he defined 1-motives, which should ultimately be
\emph{mixed motives of level or dimension} $\leq 1$. 
They form a category that we shall denote by $\M(k)$ or $\M$.

Deligne's definition was motivated by Hodge theory, and 
he asked if some Hodge-theoretic constructions could be described as realisations of \emph{a priori} constructed $1$-motives. In this direction, the first author and Srinivas associated  in  \cite{BSAP} homological and cohomological Albanese and Picard 1-motives $\Alb^-(X)$, $\Alb^+(X)$, $\Pic^-(X)$ and $\Pic^+(X)$ to an algebraic scheme
$X$ in characteristic zero, providing extensions of the classical Picard and Albanese varieties. This work was pursued in \cite{BRS}, where part of Deligne's conjecture was reformulated and proven rationally (see also \cite{ram2}).

A different step towards mixed motives was taken by
Voevodsky who defined in \cite{V} a \emph{triangulated category} of
motives $\DM^{\eff}_{\gm}(k)$\index{$\DM^{\eff}_{\gm}(k)$}. Taken with rational coefficients, this
category is conjectured to have a ``motivic" $t$-structure whose heart
should be the searched-for abelian category of mixed motives.

Since $\M(k)$ is expected to be contained in such a heart, it is only natural to try and
relate Deligne's and Voevodsky's ideas. This is what Voevodsky did in \cite[p. 218]{V} (see
also \cite[Pretheorem 0.0.18]{V0}). Denote by $\M(k)\otimes\Q$ the abelian category  of
1-motives up to isogeny over $k$. When $k$ is perfect, Voevodsky said that there exists a
\emph{fully faithful functor}
\[D^b(\M(k)\otimes\Q)\into \DM^{\eff}_{-,\et}(k)\otimes\Q\]
whose essential
image is the thick subcategory
$d\1\DM^{\eff}_{\gm}(k)\otimes\Q\subseteq\allowbreak
\DM^{\eff}_{-,\et}(k)\otimes\Q$ generated by motives of smooth curves.

In fact, a $1$-motive may be regarded as a length $1$ complex of homotopy
invariant \'etale sheaves with transfers,  so that it defines an object of Voevodsky's triangulated category
$\DM^{\eff}_{-,\et}(k)$ \index{$\DM^{\eff}_{-,\et}(k)$} of
\'etale motivic complexes \cite[\S 3.3]{V} to which $\DM_\gm^\eff(k)$ maps.
This defines a functor
\[\M(k)\to \DM^{\eff}_{-,\et}(k).\]
 
F. Orgogozo justified Voevodsky's assertion in  \cite{OR} by extending a rational version of the above functor to $D^b(\M(k)\otimes\Q)$.

In this paper, we develop the above results as follows. First of all, let us stress that one has to invert the exponential characteristic $p$ of the base field $k$ throughout. This is due to several reasons:

\begin{itemize}
\item Since the category $\DM_{-,\et}^\eff(k)$ is $\Z[1/p]$-linear by \cite[Prop. 3.3.3 2)]{V}, we cannot expect better comparison results.
\item To be in the spirit of Voevodsky, we want to use only the \'etale topology and not the fppf topology which would be more natural from the viewpoint of $1$-motives. Trying to prove anything meaningful without inverting $p$ in this context seems doomed to failure.
\end{itemize}

Anyway, the basic reason why $p$ is inverted in $\DM_{-,\et}^\eff(k)$ is homotopy invariance (the Artin-Schreier exact sequence). But if one wants to deal with non homotopy invariant phenomena, Deligne $1$-motives are not sufficient and one should enlarge them to include $\G_a$ factors as in Laumon's $1$-motives (\cf  \cite{lau}, \cite{BM}). See \cite{alessandra} for work in this direction.

\subsection{The derived category of $1$-motives, $p$-integrally}\label{0.1} While the $\Z[1/p]$-linear category $\M[1/p]$ is not an abelian category, it fully embeds into the abelian category ${}^t\M[1/p]$ of \emph{$1$-motives with torsion} introduced in \cite{BRS} (in characteristic $0$), which makes it an exact category in the sense of Quillen (see \S \ref{1.2d}). Its derived category $D^b(\M[1/p])$ with respect to this exact structure makes
sense, and moreover the functor $D^b(\M[1/p])\to D^b({}^t\M[1/p])$ turns out to be an equivalence (Theorem \ref{ptors}).

\subsection{$p$-integral equivalence} \label{0.2}Let
$\DM_{\gm,\et}^\eff=\DM_{\gm,\et}^\eff(k)$ be the thick subcategory of
$\DM_{-,\et}^\eff(k)$ generated by the image of $\DM^{\eff}_{\gm}(k)$ (see
Definition \ref{d2.1.1}) and $d\1\DM_{\gm,\et}^\eff$ the thick
subcategory of $\DM_{\gm,\et}^\eff$ generated by motives of smooth
curves.  In Theorem \ref{t1.2.1}, we refine the Voevodsky-Orgogozo
equivalence to an equivalence of categories
\begin{equation}\label{eqeq}
D^b(\M(k)[1/p])\iso d\1\DM_{\gm,\et}^\eff
\end{equation}

\subsection{Duality}\label{0.3} Deligne's extension of Cartier duality to 1-motives \cite{D}
provides the category of 1-motives with a natural involution $M
\mapsto M^*$ which extends to $D^b(\M (k)[1/p])$: see Proposition \ref{pcd}. This duality exchanges
the category ${}^t\M[1/p]$ of  \S \ref{0.1} with an abelian category ${}_t\M[1/p]$ of \emph{$1$-motives
with cotorsion} (see \S \ref{1.8}).

We show in Theorem \ref{teq} that, under $\Tot$, Deligne's Cartier duality is transformed into
the involution $M\mapsto  \ihom (M,\Z (1))$ on $d\1\DM^{\eff}_{\gm,\et}(k)$ given by
the internal (effective) Hom. Of course, this result involves biextensions. 

\subsection{Left adjoint} Composing \eqref{eqeq} with the
inclusion into $\DM_{\gm,\et}^\eff(k)$, we obtain a ``universal
realisation functor"
\[\Tot:D^b(\M(k)[1/p])\to \DM_{\gm,\et}^\eff(k).\] 

It was conjectured by Voevodsky (\cite{V1}; this is also implicit in \cite[Preth. 0.0.18]{V0})
that, rationally, $\Tot$ has a left adjoint. We prove this in Section \ref{qlalb}.

It is shown in Remark \ref{noleft} that $\Tot$ does not have a left adjoint integrally.
There is nevertheless an integral statement, which involves an interplay between the \'etale and
the
\emph{Nisnevich} topology. Let $\alpha^*:\DM_\gm^\eff(k)\to \DM_{\gm,\et}^\eff(k)$ be the
change of topology functor. We find a functor
\[\LAlb: \DM^{\eff}_{\gm}(k)\to D^b(\M(k)[1/p])\]
verifying the following universal property: if $(M,N)\in \DM_\gm^\eff(k)\times D^b(\M[1/p])$,
then there is a functorial isomorphism
\begin{equation}\label{eq0.1}
\Hom_{\DM_{\gm,\et}^\eff(k)}(\alpha^*M,\Tot(N))\simeq
\Hom_{D^b(\M[1/p])}(\LAlb(M),N).
\end{equation}

We give its construction in Sect. \ref{lalb}. 

The point is that, applying  $\LAlb$  to various motives, we get
interesting and intrinsically-defined $1$-motives. For example, applying it
to the motive
$M(X)$ of a smooth variety
$X$, we get the \emph{homological Albanese complex} $\LAlb
(X):=\LAlb(M(X))$ of $X$. Its homology $1$-motives
$\LA{i}(X):={}_tH_i(\LAlb(X))$ relative to the $t$-structure on $D^b(\M[1/p])$ with heart ${}_t\M[1/p]$
(see \S \ref{0.3}) are $1$-motives (with cotorsion) functorially attached to $X$.

\subsection{Smooth schemes} We then proceed to compute  $\LAlb (X)$
for a smooth scheme
$X$: in principle this determines $\LAlb$ on the whole of $\DM_\gm^\eff$, since this category
is generated by the $M(X)$. It is related with the ``Albanese scheme" $\cA_{X/k}$ of \cite{ram}
(extending the Serre Albanese variety of \cite{serrealb}) in the following way: $\LAlb(X)$ is a
``$3$-extension" of $\cA_{X/k}$ by the Cartier dual of the N\'eron-Severi group of $X$, that we
define as the \'etale sheaf represented by cycles of codimension $1$ on $X$ modulo algebraic
equivalence.  (See Theorem
\ref{trunc}.) We deduce that
$\LA{1}(X)$ is isomorphic to the $1$-motive $\Alb^- (X)$ of \cite{BSAP}.

\subsection{$\LAlb$ and $\RPic$} Composing
$\LAlb$ with duality, we obtain a
contravariant functor
$$\RPic: \DM^{\eff}_{\gm}(k)\to D^b(\M(k)[1/p])$$
such that 
\[\RA{i}(M)\df{}^tH^i(\RPic (M)) \simeq {}_tH_i(\LAlb (M))^*\]
for any $M\in \DM^{\eff}_{\gm}(k)$. Here, 
${}^t H^i$ is defined with respect to the $t$-structure with heart ${}^t\M[1/p]$. We call
$\RPic$ the {\it motivic Picard}\, functor. We define the \emph{cohomological Picard
complex} by  $\RPic (X):=\RPic (M(X))$.

\subsection{Singular schemes}\label{sing}
When $k$ is of characteristic $0$, the motive and motive with compact
support $M(X)$ and $M^c(X)$ are defined for any variety $X$ as objects
of $\DM_\gm^\eff(k)$, so that
$\LAlb(X)$ and the 
\emph{Borel-Moore Albanese complex} $\LAlb^c (X) \df \LAlb (M^c
(X))$ make sense. Still in characteristic $0$ we further
define, for an equidimensional scheme $X$ of dimension
$n$, the \emph{cohomological Albanese complex} $\LAlb^* (X)
\df\allowbreak \LAlb (M (X)^*(n)[2n])$. We define similarly 
$\RPic^c (X):= \RPic (M^c (X))$  and $\RPic^* (X):=\RPic
(M (X)^*(n)[2n])$. 
 We describe some properties
of these complexes in Sect. \ref{6}.

We then give some general qualitative estimates for $\LA{i}(X)$ in Proposition \ref{c3.1} (see also Proposition \ref{c3.1bis}) as well as $\LA{i}^c (X)\allowbreak\df {}_tH_i(\LAlb^c (X))$ in
Proposition \ref{c3.1c}. Sect. \ref{12} is devoted to a detailed study of $\LA{1}(X)$,
$\LA{1}^c(X)$ and $\LA{1}^*(X)$; the main results are summarised in the introduction of this section. In
particular, we prove that $\LA{1}(X)$ is canonically isomorphic to the $1$-motive $\Alb^-(X)$ of \cite{BSAP} if $X$ is normal or proper. Here, the interplay between $\LAlb$ and $\RPic$ (duality between Picard and Albanese) plays an essential r\^ole. We also prove that $\LA{1}^*(X)\simeq \Alb^+(X)$ for any $X$ (in fact, $\RA{1}^*(X)\simeq \Pic^-(X)$ in Theorem~\ref{*=-}).

It is quite striking that $\LA{i}(X), \LA{i}^c(X)$ and $\LA{i}^*(X)$ are actually Deligne $1$-motives for $i\le 1$, but not in general for $i\ge 2$ (already for $X$ smooth projective).

We also completely compute $\LA{i}(X)$ for any curve $X$, showing that $M (X)$
has a ``Chow-K\"unneth decomposition" in 
$\DM^{\eff}_{\gm,\et}(k)\otimes\Q$ and that 
$\LA{i} (X)$ coincide with Deligne-Licht\-en\-baum motivic homology
of the curve $X$ (see Theorem~\ref{lasc}, \cf \cite{LI} and \cite{BH}). Finally, we
completely compute $\LA{i}^c (X)$
of a smooth curve $X$ (see Theorem~\ref{bmlac}), showing that
$\LA{1}^c (X) = H^1_m(X)(1)$ is Deligne's motivic $H^1$ in
\cite{D}. Dually, we recover Deligne's 1-motivic $H^1$ of
any curve. With a little more effort, one should be able to identify our computations with
those of Lichtenbaum in \cite{LI} and \cite{LI1}.

\subsection{Ro\v\i tman's torsion theorem} The isomorphism \eqref{eq0.1} comes with
a functorial map (``motivic Albanese map")
\begin{equation}\label{eq0.2}
a_M:\alpha^*M \to \Tot \LAlb (M)
\end{equation} 
for any $M\in\DM_\gm^\eff$. If $X$ is smooth projective, this canonical map
applied to $M=M(X)$ gives back the
Albanese map from the $0$-th Chow group to the rational points of the
Albanese variety. This thus translates very classical
mathematics to the motivic setting.  When $X$ is only smooth, we
recover a generalised Albanese map from Suslin homology
\[a_X^{sing}:H_0^{sing}(X;\Z)[1/p]\to \cA_{X/k}(k)[1/p]\] 
which was first constructed by Ramachandran \cite{ra1} and
Spie\ss-Szamuely \cite{spsz}.\footnote{The
observation that Suslin homology is related to $1$-motives is initially
due to Lichtenbaum \protect\cite{LI}.} The map $a_X^{sing}$ is an isomorphism if
$\dim (X)\leq 1$ (see Proposition~\ref{sus}). 

We then get a very natural proof of the classical theorem of Ro\v\i
tman, and even of its generalisation to open smooth varieties by Spie\ss-Sza\-mu\-e\-ly \cite[Th.
1.1]{spsz} (removing their hypothesis on the existence of a smooth compactification): see Theorem
\ref{troitclas}. 

We also deal with singular schemes when $\car k=0$, see Proposition
\ref{p12.3} and its corollaries. Here there is an overlap with recent work of Geisser
\cite{geisser2}. The works may be compared as follows: Geisser works in arbitrary
characteristic and can handle $p$-torsion in characteristic $p$, but he works only with proper
schemes, while the use of $\DM$ forces us to work in characteristic $0$ for singular schemes,
but we do handle open schemes.

Still in characteristic $0$, we get a Borel-Moore version of
Ro\v\i tman's theorem  as well, see Proposition \ref{p13.4} and its corollary.

Notably, we obtain a ``cohomological" Ro\v\i tman theorem, involving torsion in a motivic cohomology group: see Corollary \ref{c14.4}. In a recent work, Mallick  \cite{mallick} proves a parallel cohomological theorem, involving torsion in the Levine-Weibel Chow group (\cf \cite[Th. 6.4.1]{BSAP}).   Mallick works with projective schemes, but in any characteristic. Hopefully, the two theorems are compatible, see Remark \ref{r14.5}.\bigskip

Moreover, the consideration of Voevodsky's
ca\-te\-go\-ries provides us with some nonobvious extra
structures on $D^b(\M[1/p])$:

\subsection{The homotopy $t$-structure} It turns out that the homotopy $t$-structure on
$\DM_{-,\et}^\eff$ and the equivalence of categories \eqref{eqeq} induce a \emph{third}
$t$-structure on $D^b(\M[1/p])$, that we also call the homotopy $t$-structure (Theorem
\ref{t3.2.3}; see also Corollary \ref{c3.3.2}). Its heart is formed of so-called
\emph{$1$-motivic sheaves}: their consideration is very useful for the computation of
$\LAlb(X)$ for smooth $X$.

\subsection{Tensor structure and internal Hom} Similarly, the functor $\LAlb$ turns out to
transport the tensor structure on $\DM_\gm^\eff\otimes\Q$ to a tensor structure on
$D^b(\M\otimes\Q)$. This tensor structure is exact (for the standard $t$-structure), respects
the weight filtration and may be computed explicitly. There is also an exact internal Hom. See
Sect. \ref{tens}.

\subsection{Realisations} For $X$ smooth over $k =\C$ the complex numbers, one can easily check
that the $1$-motive $\RA{i} (X)$ has a Hodge realisation abstractly isomorphic to
$H^i_{(1)}(X_{\an},\Z (1))$, the largest 1-motivic part of the mixed Hodge structure on 
$H^i(X_{\an},\Z)$ Tate twisted by 1.
We expect that for any scheme
$X$  over a field $k$ of characteristic zero the 1-motives $\RA{i}
(X)$ are isogenous to the 1-motives $M_i(X)$ constructed in
\cite{BRS}. For $X$ normal or proper, this also holds for $i=1$ by the comparison results indicated in \ref{sing} and the results of \cite{BSAP}.

In the second part of this work, we plan to deal with Hodge and $\ell$-adic realisations, thereby providing a canonical version of the isomorphism of the previous paragraph.  More
generally, this should give a more conceptual (and hopefully more integral) proof of the main results of \cite{BRS}, and more, \eg see \cite{BM} where the expected formulas are displayed (up to isogeny).

\subsection{Caveat} While one might hope that these results are a partial template for a
future theory of mixed motives, we should stress that some of them are definitely special to
level $\le 1$. Namely:

\begin{itemize}
\item It is succintly pointed out in \cite[\S 3.4 p. 215]{V} that the non finite generation of
the Griffiths group prevents higher-dimensional analogues of $\LAlb$ to exist. (This goes
against \cite[Conj. 0.0.19]{V0}.)
\item Contrary to Theorem \ref{t3.2.3}, the homotopy $t$-structure does not induce a
$t$-structure on $d_{\le n}\DM_{\gm,\et}^\eff$ for $n\ge 2$. This can already be seen on
$\Z(2)$, although here the homotopy sheaves are conjecturally ind-objects of $d_{\le
2}\DM_{\gm,\et}^\eff$ (see \cite[\S 6]{V0}). The situation seems to be similar for a surface; it
would be interesting to work out a conjectural picture in general.
\end{itemize}

These two issues seem related in a mysterious way!

\subsection{A small reading guide} Since this article is rather long, we would like to offer
some suggestions to the reader, hoping that they will be helpful.

One might start by quickly brushing through \S \ref{1.1} to review the definition of Deligne's
$1$-motives, look up \S \ref{1.2d} to read the definition of $D^b(\M[1/p])$ and then proceed
directly to Theorem \ref{t1.2.1} (full embedding), referring to Section \ref{sect1} ad libitum
to read the proof of this theorem. The lengths of Sections \ref{homotopy} and \ref{dual} are
necessary evils; they may very well be skipped at first reading with just a look at their main
results (Theorem \ref{t3.2.3}, the homotopy $t$-structure, and Theorem \ref{teq}, agreement of
the two Cartier dualities). 

One may then read Section \ref{lalb} on the construction of $\LAlb$
(which hopefully will be pleasant enough), glance through Section \ref{qlalb} (the rational
version of $\LAlb$) and have a look in passing at Section \ref{tens} for the tensor structure and internal Hom on
$D^b(\M\otimes\Q)$. After this, the reader might fly over the mostly formal sections
\ref{6} and \ref{6c}, jump to Theorem \ref{trunc} which computes
$\LAlb(X)$ for a smooth scheme $X$, read Sections \ref{comps} and \ref{12} on $\LAlb$ of
singular schemes where he or she will have a few surprises, read Section \ref{rovi} on Ro\v\i tman's theorem and its generalisations,  finally have a well-earned rest in recovering familiar
objects in Section
\ref{scurves} (the case of curves). And never look at the appendices.

The reader will also find an index of notations at the end.
 
\subsection*{Acknowledgements}
This work has been done during periods of stay of the first author
at the IH\'ES and the ``Institut de Math\'ematiques de Jussieu" University of Paris 7 and of the second author at the Universities of Padova and Roma ``La Sapienza''. We like to thank
these institutions for their hospitality and financial support. We also thank A. Bertapelle, M. Hindry, L. Illusie, P. Jossen, M. Ojanguren, M. Saito, T. Szamuely, C. Voisin for helpful comments or suggestions regarding this work. The first author is grateful to all his coauthors in the subject of $1$-motives; the second author would like to acknowledge the inspiration provided by the paper of Spie\ss-Szamuely \cite{spsz}. Finally, our intellectual debt to A. Grothendieck, P. Deligne and V. Voevodsky is evident.
\bigskip 

\emph{In all this paper, $k$ is a perfect field of exponential characteristic $p$. We write
$Sm(k)$ \index{$Sch(k)$, $Sm(k)$} for the category of smooth schemes of finite type and $Sch(k)$  for the category of all
separated schemes of finite type. Since we ignore characteristic $p$ phenomena in this paper,  we invert $p$ in the Hom groups of all categories constructed out of commutative group schemes and \'etale sheaves from Subsection \ref{1.2} onwards.}

\part{The universal realisation functor}

\section{The derived category of $1$-motives} \label{sect1}

The main reference for (integral, free) $1$-motives is \cite[\S
10]{D}, see also \cite[\S 1]{BSAP}. We also provide an Appendix~\ref{AppendixB} 
on 1-motives with torsion which were introduced in \cite[\S 1]{BRS}. For the derived category of
$1$-motives up to isogeny we refer to \cite[Sect. 3.4]{V} and \cite{OR}: here we are interested
in the integral version.

\subsection{Deligne 1-motives}\label{1.1} The
following terminology is handy:

\begin{defn}\label{d1.1.1} a) An abelian sheaf $L$ on $(Sm(k))_\et$ is
\emph{discrete} if it is locally constant $\Z$-constructible (\ie with
finitely generated geometric fibres). The full subcategory of discrete
abelian sheaves on $(Sm(k))_\et$ is denoted by ${}^t\cM_0(k)={}^t\cM_0$. \\   
b) A \emph{lattice} is a
$k$-group scheme locally constant for the {\'e}tale topology, with
geometric fibre(s) isomorphic to a finitely generated free abelian group,
\ie representing a torsion-free discrete sheaf. The full subcategory of lattices is denoted by
$\cM_0(k)=\cM_0$. \index{$\cM_0$, ${}^t\cM_0$}
\end{defn}

A \emph{Deligne $1$-motive} over $k$ is a complex
of group  schemes 
$$M= [ L \by{u} G]$$ 
where $L$ is a lattice and
$G$ is a  semi-abelian $k$-scheme. Thus $G$ can be represented
by an extension
$$0\to T \to G \to A \to 0$$ where $T$ is a $k$-torus and $A$ is an
abelian $k$-scheme.

As a complex, \emph{we shall place $L$ in degree $0$ and $G$ in degree $1$}.
Note that this convention is only partially shared by the existing
literature.

A map from $M = [ L \by{u} G]$ to $M'= [ L' \by{u'} G']$ is a
commutative square
\[ \begin{CD}
 L @>u>>  G\\
@V{f}VV  @V{g}VV  \\
L'@>u'>> G'
\end{CD}\]
in the category of group schemes. Denote by $(f, g):M\to M'$ such a
map. The natural composition of squares makes up the category
of Deligne's 1-motives. We shall denote this category by
$\M(k)$. We shall usually write $\M$ instead
of $\M(k)$, unless it is necessary to specify $k$. The following lemma is
immediate:

\begin{lemma}\label{lidco} $\M$ is an idempotent complete additive category.\qed
\end{lemma}

\begin{defn}\label{1mot} Let $R$ be a commutative ring. For any additive
category
$\cA$, we denote by $\cA\otimes R$ \index{$\cA\otimes R$} the $R$-linear category obtained from $\cA$ by tensoring morphisms by $R$, and by $\cA\boxtimes R$
\index{$\cA\boxtimes R$} the
pseudo-abelian hull (idempotent completion) of $\cA\otimes R$.
\end{defn}

This distinction is useful as $\cA\otimes R$ may not be idempotent
complete even if $\cA$ is.

We shall also use the following category, which is technically very
useful:

\begin{defn}\label{Meff} Let $\anM^{\eff}$ \index{$\anM^{\eff}$} denote the category given by complexes of group schemes $[L\to G]$ where $L$ is discrete and
$G$ is a commutative algebraic group whose connected component of the
identity $G^0$ is semi-abelian. It contains $\M$ as a full subcategory.
\end{defn}

This category is studied in more detail in \S \ref{noncon}.

\begin{propose}[\cf \protect{\cite[3.2.2]{OR}}] \label{isoab} The
inclusion $\M\to\anM^\eff$ induces an equivalence of categories 
\[e:\M\otimes\Q\iso\anM^{\eff}\otimes \Q.\]
In particular, the category $\M\otimes\Q$ is abelian, hence
$\M\otimes\Q=\M\boxtimes\Q$.\index{$\M$, $\M\otimes\Q$}
\end{propose}

\begin{proof} (See also Lemma \ref{lB.1.3}.)  It is enough to show that $e$ is essentially
surjective. But if
$[L\to G]\in\anM^{\eff}$, then we have a diagram
\[\begin{CD}
[L^0\to G^0]@>>> [L^0_{\fr}\to G^0/u(L^0_{\tor})]\\
@VVV\\
[L\by{u} G]
\end{CD}\]
where the vertical (\resp horizontal) map is a pull-back (\resp a
push-out) and $L^0_{\fr}\df L^0/L^0_{\tor}$ where $L^0_{\tor}$ is the
torsion subgroup of $L^0$. Both maps are isomorphisms in
$\anM^{\eff}\otimes\Q$. The last assertion follows from the fact that
$\anM^{\eff}\otimes\Q$ is abelian (Proposition \ref{nca}).
\end{proof}

\begin{sloppypar}
\begin{remarks} 1 (see also Def. \ref{dA.2} c)). An \emph{isogeny} between Deligne's 1-motives,
from $M = [ L \by{u}G]$ to $M'= [ L' \by{u'} G']$ in $\M (k)$,
is a diagram of group schemes with exact columns
$$
\begin{CD}
&& 0 \\
&& @V{}VV \\
0&& F\\
@V{}VV @V{}VV  \\
L @>{u}>>  G\\
@V{f}VV @V{g}VV  \\
L' @>{u'}>> G'\\
@V{}VV @V{}VV  \\
E&&0\\
@V{}VV&&\\
0&&
\end{CD}
$$
where $F$ and $E$ are finite. Isogenies become invertible in
$\M\otimes\Q$.\\

2. The category $\M$ of
Deligne's 1-motives has kernels and cokernels (see Proposition~\ref{lim})
but it is not abelian. This easily follows  from the diagram hereabove: an
isogeny has vanishing kernel and cokernel but it is not an isomorphism in
$\M$.
\end{remarks}
\end{sloppypar}

\subsection{Weights and cohomological dimension} Recall that $M =[L\to G]
\in \M$ has an increasing filtration by sub-$1$-motives as follows:\index{$W_i(M)$}
$$W_i(M) =\left \{\begin {array}{cl} M & i\geq 0\\ G
& i= -1\\ T & i= -2\\ 0 & i \leq -3 \end{array} \right. $$

We then have $\gr_{-2}^W(M) = T[-1]$, $\gr_{-1}^W(M) = A[-1]$ and
$\gr_{0}^W(M) = L$ (according to our convention of placing $L$ in
degree zero). We say that $M$ is \emph{pure of weight $i$} if
$\gr_j^WM=0$ for all $j\ne i$. Note that for two  pure 1-motives
$M, M'$, $\Hom (M, M')\neq 0$ only if they have the
same weight.

\begin{propose}[\protect{\cite[3.2.4]{OR}}] \label{iso1} The
category $\M\otimes\Q$ is of cohomological dimension $\leq 1$, \ie if
$\Ext^i (M,M') \neq 0$, for $M, M'\in \M\otimes\Q$, then $i =0$ or $1$.
\end{propose}

Recall a sketch of the proof in \cite{OR}: one first checks that
$\Ext^1(M,M')\allowbreak=0$ if $M,M'$ are pure of weights $i,i'$ and $i\le i'$. This
formally reduces the issue to checking that if $M,M',M''$ are pure
respectively of weights $0,-1,-2$, then the Yoneda product of two classes
$(e_1,e_2)\in \Ext^1(M,M')\times \Ext^1(M',M'')$ is $0$. Of course we may
assume $e_1$ and $e_2$ integral. By a transfer argument, one may further
reduce to $k$ algebraically closed. Then the point is that $e_1$ and
$e_2$ ``glue" into a $1$-motive, so are induced by a 3 step filtration on
a complex of length $1$; after that, it is formal to deduce that $e_2\cdot
e_1=0$ (\cf \cite[IX, Prop. 9.3.8 c)]{sga7}).

\begin{remark} We observe that Proposition \ref{iso1} can be regarded as
an algebraic version of a well-known property of $\M (\C)\otimes \Q$.
Namely, $\M (\C)\otimes \Q$ can be realised as a  thick abelian
sub-category of $\Q$-mixed Hodge structures, see \cite{D}. Since the
latter is of cohomological dimension $\le 1$, so is $\M(\C)\otimes\Q$
(use \cite[Ch. III, Th. 9.1]{mcl}).
\end{remark}

\subsection{Group schemes and sheaves with transfers}\label{ur}

\begin{defn}[\cf Def. \protect{\ref{dD.1}}]\label{hi} We denote by $\HI_\et=\HI_\et(k)$ \index{$\HI_\et$} the
category of homotopy invariant \'etale sheaves with transfers over $Sm(k)$: this is
the full subcategory of the category $\EST(k)=\Shv_\et(SmCor(k))$ \index{$\EST(k)$} from \cite[\S
3.3]{V} consisting of those \'etale sheaves with transfers that are
homotopy invariant.
\end{defn}

Let $G$ be a commutative $k$-group
scheme. We shall denote by $\uG$ the associated \'etale sheaf
of abelian groups. In fact, under a minor assumption, $\uG$ is an
\emph{\'etale sheaf with transfers}, as explained by Spie\ss-Szamuely
\cite[Proof of Lemma 3.2]{spsz}, \cf also Orgogozo \cite[3.1.2]{OR}.
Both references use symmetric powers, hence deal only with smooth
quasi-projective varieties. Here is a cheap way to extend their
construction to arbitrary smooth varieties: this avoids to have to
prove that $\DM_\gm^\eff(k)$ may be presented in terms of
smooth quasi-projective varieties, \cf \cite[beg. of \S 1]{OR}.

\begin{lemma} \label{l1.3} Suppose that the neutral component $G^0$ is
quasi-pro\-jec\-tive. Then the
\'etale sheaf
$\uG$ is provided with a canonical structure of presheaf with transfers.
Moreover, if $G^0$ is a semi-abelian variety, then $\uG$ is
homotopy invariant.
\end{lemma}

\begin{proof} For two smooth $k$-varieties $X,Y$, we have to provide a
pairing
\[c(X,Y)\otimes \uG(X)\to \uG(Y)\]
with the obvious compatibilities. As in \cite[Ex. 2.4]{VL}, it is enough
to construct a good transfer $f_*:\uG(W)\to \uG(X)$ for any finite
surjective map $f:W\to X$ with $X$ a normal $k$-variety. For $X$ and $W$
quasi-projective, this is done in \cite{spsz} or \cite{OR}\footnote{For the symmetric powers
of $G$ to exist as schemes, it suffices that $G^0$ be quasi-projective.}. In general, cover $X$
by affine opens $U_i$ and let $V_i=f^{-1}(U_i)$. Since $f$ is finite, $V_i$ is also affine,
hence transfers $\uG(V_i)\to \uG(U_i)$ and
$\uG(V_i\cap V_j)\to \uG(V_i\cap V_j)$ are defined; the commutative
diagram
\[\begin{CD}
0\to \uG(W)@>>> \prod \uG(V_i)@>>> \prod \uG(V_i\cap V_j)\\
&&@V{f_*}VV @V{f_*}VV\\
0\to \uG(X)@>>> \prod \uG(U_i)@>>> \prod \uG(U_i\cap U_j)
\end{CD}\]
uniquely defines the desired $f_*$.

The second statement of the lemma is well-known (\eg  \cite[3.3.1]{OR}).
\end{proof}
   
Actually, the proof of \cite[Lemma 3.2]{spsz}
defines a homomorphism in
$\HI_\et$
\[\sigma:L_\et (G)\to \uG\]
which is split by the obvious morphism of sheaves
\[\gamma:\uG\to L_\et (G)\]
given by the graph of a morphism. Therefore $\sigma$ is an epimorphism of
sheaves. (One should be careful, however, that $\gamma$ is not additive.)
When $\uG$ is homotopy invariant, one deduces from it  as in
\cite[Remark 3.3]{spsz} a morphism in $\DM_{-,\et}^\eff(k)$
\begin{equation}\label{eq1.4}
M_\et (G)=C_*(L_\et (G))\to \uG.
\end{equation}

Let $\HI_\et^{[0,1]}$ be the category
of complexes of length $1$ of objects of $\HI_\et$ (concentrated in
degrees $0$ and $1$): this is an abelian
category. Lemma \ref{l1.3} gives us a functor
\begin{align*}
\anM^\eff&\to \HI_\et^{[0,1]}\\
M&\mapsto\underline{M}
\end{align*}
hence, by composing with the embedding $\M\to \anM^\eff$ of
Proposition \ref{isoab}, another functor
\begin{equation}\label{eq2.2}
\M\to \HI_\et^{[0,1]}.
\end{equation}

\subsection{$1$-motives with torsion and an exact structure on $\M[1/p]$}\label{1.2}\

Recall that, from now on, we invert the exponential characteristic $p$ in the Hom groups of  all categories constructed out of commutative $k$-group schemes or \'etale $k$-sheaves. This does nothing in characteristic $0$.

The reader can check that most of the statements below become false if $p>1$ is not inverted. We hope that statements integral at $p$ may be recovered in the future by considering some kind of non-homotopy invariant motives and cohomology theories.

We start with:

\begin{propose}\label{pexact} Let $M^\cdot$ be a complex of
objects of
$\anM^\eff[1/p]$. The following conditions are equivalent:
\begin{thlist}
\item The total complex $\Tot(M^\cdot)$ in $C(\HI_\et)[1/p]$ (see Definition \ref{hi} and Lemma
\ref{l1.3}) is a\-cycl\-ic.
\item For any $q\in\Z$, $H^q(M^\cdot)$ is of the form $[F^q =F^q]$, where $F^q$ is finite.
\end{thlist}
\end{propose}

\begin{proof} (ii) $\Rightarrow$ (i) is obvious. For the converse, let
$M^q=[L^q\to G^q]$ for all $q$. Let $L^\cdot$ and $\uG^\cdot$ be the two
corresponding ``column" complexes of sheaves. Then we have a long exact
sequence in
$\HI_\et$:
\[\dots\to H^q(L^\cdot)\to H^q(\uG^\cdot)\to H^q(\Tot(M^\cdot))\to
H^{q+1}(L^\cdot)\to\dots\]

The assumption implies that $H^q(L^\cdot)\iso H^q(\uG^\cdot)$ for all $q$.
Since $H^q(L^\cdot)$ is discrete and $H^q(\uG^\cdot)$ is representable by
a commutative algebraic group, both must be finite.
\end{proof}

We now restrict to complexes of $\M[1/p]$.

\begin{defn}\label{dexact} A complex of $\M[1/p]$ is \emph{acyclic} if it
satisfies the equivalent conditions of Proposition \ref{pexact}. An
acyclic complex of the form $0\to N'\by{i} N\by{j} N''\to 0$ is called a
\emph{short exact sequence}. 
\end{defn}

Recall that in \cite{BRS} a category of $1$-motives with torsion was introduced. We shall
denote it here by ${}^t\M$ \index{${}^t\M$, ${}^t\M^\eff$} in order to distinguish it from $\M$. Denote by ${}^t\M^\eff$ the
effective $1$-motives with torsion: ${}^t\M^\eff$ is the full subcategory of
the category $\anM^\eff$ of Definition \ref{Meff} consisting of the objects
$[L\to G]$ where
$G$ is connected. Then ${}^t\M$ is the
localisation of ${}^t\M^\eff$ with respect to quasi-isomorphisms.

The main properties of ${}^t\M$ are recalled in Appendix \ref{AppendixB}. In particular, the category ${}^t\M[1/p]$ is abelian (Theorem \ref{1mtora}) and by Proposition \ref{free} we have a full embedding
\begin{equation}\label{fullem}
\M[1/p]\into{}^t\M[1/p]
\end{equation}
which makes $\M[1/p]$ an exact subcategory of ${}^t\M[1/p]$. The following lemma is clear:

\begin{lemma}\label{lexact} A complex $0\to N'\by{i} N\by{j} N''\to 0$ in $\M[1/p]$ is a short exact sequence in the sense of Definition \ref{dexact} if and only if it is a short exact sequence for the exact structure given by Proposition \ref{free}.
\end{lemma}

\begin{remarks}\label{rexact} 1) There is another, much stronger, exact
structure on $\M[1/p]$, induced by its full embedding in $\anM^\eff[1/p]$: it
amounts to require  a complex $[L^\cdot\to G^\cdot]$ to be exact if and
only if both complexes $L^\cdot$ and $\uG^\cdot$ are acyclic. We shall
not use this exact structure in the sequel. (See also Remark \ref{r1.8.6}.)

2) Clearly, the complexes of Definition \ref{dexact} do not provide $\anM^\eff[1/p]$ with an exact
structure. It is conceivable, however, that they define an exact structure on the
localisation of $\anM^\eff[1/p]/\text{homotopies}$ with respect to morphisms with acyclic kernel and
cokernel.
\end{remarks}

\subsection{The derived category of $1$-motives}\label{1.2d}

\begin{lemma}\label{l1.5.1} A complex in $C(\M[1/p])$ is acyclic in the sense of Definition
\ref{dexact} if and only if it is acyclic with respect to the exact structure of $\M[1/p]$ provided
by Lemma
\ref{lexact} in the sense of \cite[1.1.4]{BBD} or \cite[\S 1]{neeman}.
\end{lemma}

\begin{proof} Let $X^\cdot\in C(\M[1/p])$. Viewing $X^\cdot$ as a complex of objects of $\anM^\eff[1/p]$,
we define $D^n=\im (d^n:X^n\to X^{n+1})$. Note that the $D^n$ are Deligne $1$-motives. Let
$e_n:X^n\to D^n$ be the projection and $m_n:D^n\to X^{n+1}$ be the inclusion. We have
half-exact sequences
\begin{equation}\label{eq1.5}
0\to D^{n-1}\by{m_{n-1}} X^n\by{e_n} D^n\to 0
\end{equation}
with middle cohomology equal to $H^n(X^\cdot)$. Thus, if $X^\cdot$ is acyclic in the sense of
Definition \ref{dexact}, the sequences \eqref{eq1.5} are short exact which means that $X^\cdot$
is acyclic with respect to the exact structure of $\M[1/p]$. Conversely, suppose that $X^\cdot$ is
acyclic in the latter sense. Then, by definition, we may find
${D'}^n$, $e'_n$, $m'_n$ such that $d^n=m'_ne'_n$ and that the sequences analogous to
\eqref{eq1.5} are short exact. Since $\anM^\eff[1/p]$ is abelian (Proposition \ref{nca}), ${D'}^n=D^n$ and we are done.
\end{proof}

From now on, we shall only say ``acyclic" without further precision. 

Let $K(\M[1/p])$ be the homotopy category of $C(\M[1/p])$. By \cite[Lemmas 1.1 and 1.2]{neeman}, the full
subcategory of $K(\M[1/p])$ consisting of acyclic complexes is triangulated and thick (the latter
uses the fact that $\M[1/p]$ is idempotent-complete, \cf Lemma \ref{lidco}). Thus one may define the 
derived category of $\M[1/p]$ in the usual way:

\begin{defn}\label{ddermot} \index{$D^b(\M[1/p])$}
a) The \emph{derived category of $1$-motives} is
the localisation $D(\M[1/p])$ of the homotopy category $K(\M[1/p])$ with
respect to the thick subcategory $A(\M[1/p])$ consisting of acyclic complexes. Similarly for
$D^\pm(\M[1/p])$ and $D^b(\M[1/p])$.\\
b) A morphism in $C(\M[1/p])$ is a
\emph{quasi-isomorphism} if its cone is acyclic.
\end{defn}

\subsection{Torsion objects in the derived category of $1$-motives}

Let $\cM_0$ be the category of lattices (see Definition \ref{d1.1.1}): the inclusion functor
$\cM_0[1/p]\by{A}
\M[1/p]$ provides it with the structure of an exact subcategory of $\M[1/p]$.
Moreover, the embedding
\[\cM_0[1/p]\by{B} {}^t\cM_0[1/p]\]
is clearly exact, where ${}^t\cM_0$ is the abelian category of discrete
\'etale sheaves (see Definition \ref{d1.1.1} again). In fact, we also have an exact functor
\begin{align*}
{}^t\cM_0[1/p]&\by{C}{}^t\M[1/p]\\
L&\mapsto [L\to 0].
\end{align*}

 Hence an induced diagram of
triangulated categories:
\[\begin{CD}
D^b(\cM_0[1/p])@>B>> D^b({}^t\cM_0[1/p])\\
@V{A}VV @V{C}VV\\
D^b(\M[1/p])@>D>> D^b({}^t\M[1/p]).
\end{CD}\]

\begin{thm}\label{ptors} In the
above diagram\\ 
a) $B$ and $D$ are equivalence of categories.\\ 
b) $A$ and $C$ are fully faithful; restricted to torsion objects they are equivalences of
categories.
\end{thm}

(For the notion of torsion objects, see Proposition \ref{p1.1}.)

\begin{proof} a) For $B$, this follows from Proposition
\ref{derxact} provided we check that any object $M$ in ${}^t\cM_0[1/p]$ has a
finite left resolution by objects in $\cM_0[1/p]$. In fact $M$ has a length
$1$ resolution: let $E/k$ be a finite Galois extension of group $\Gamma$
such that the Galois action on $M$ factors through $\Gamma$. Since $M$ is
finitely generated, it is a quotient of some power of $\Z[\Gamma]$, and
the kernel is a lattice. Exactly the same argument works for $D$.

b) By a) it is sufficient to prove that $C$ is fully faithful. It suffices to verify that the
criterion of Proposition \ref{pschapira} is verified by the full embedding ${}^t\cM_0[1/p]\to {}^t\M[1/p]$.

Let $[L\to 0]\into [L'\to G']$ be a monomorphism in ${}^t\M[1/p]$. We may assume that it is given
by an effective map. The assumption implies that $L\to L'$ is mono: it then suffices to compose
with the projection $[L'\to G']\to [L'\to 0]$.

It remains to show that $A$ is essentially
surjective on torsion objects. Let $X=[C^\cdot\to G^\cdot]\in
D^b(\M[1/p])$, and let $n>0$ be such that $n 1_X=0$. Arguing as in the proof of Proposition \ref{pexact}, this implies that the cohomology sheaves of both $C^\cdot$
and $G^\cdot$ are killed by some possibly larger integer
$m$. We have an exact triangle
\[[0\to G^\cdot]\to X\to [C^\cdot\to 0]\by{+1}
\]
which leaves us to show that $[0\to G^\cdot]$ is in the essential image
of $C$. Let $q$ be the smallest integer such that $G^q\ne 0$: we have an
exact triangle
\[\{G^q\to \im d^q\}\to  G^\cdot \to \{0\to G^{q+1}/\im
d^q\to\dots\}\by{+1}\]
(here we use curly braces in order to avoid confusion with the square
braces used for $1$-motives). By descending induction on $q$, the right
term is in the essential image, hence we are reduced to the case where
$G^\cdot$ is of length $1$. Then $d^q:G^q\to G^{q+1}$ is epi and
$\mu:=\ker d^q$ is finite and locally constant\footnote{Note that this is
true even if $m$ is divisible by the characteristic of $k$, since we
only consider sheaves over \emph{smooth} $k$-schemes.}. Consider the
diagram in
$K^b(\anM^\eff[1/p])$
\[
[\begin{smallmatrix}0\\\downarrow\\ 0\end{smallmatrix}\to
\begin{smallmatrix}G^q\\\downarrow\\ G^{q+1}\end{smallmatrix}] \leftarrow
[\begin{smallmatrix}0\\\downarrow\\ \mu\end{smallmatrix}\to
\begin{smallmatrix}G^q\\||\\ G^{q}\end{smallmatrix}]
\leftarrow [\begin{smallmatrix}L_1\\\downarrow\\ L_0\end{smallmatrix}\to
\begin{smallmatrix}G^q\\||\\ G^{q}\end{smallmatrix}]\to
[\begin{smallmatrix}L_1\\\downarrow\\ L_0\end{smallmatrix}\to
\begin{smallmatrix}0\\\downarrow\\ 0\end{smallmatrix}]
\]
where $L_1\to L_0$ is a resolution of $\mu$ by lattices (see proof of a)).
Clearly all three maps are quasi-isomorphisms, which implies that the
left object is quasi-isomorphic to the right one on
$D^b(\M[1/p])$.
\end{proof}

\begin{cor} \label{nobox} Let $A$ be a subring of $\Q$ containing $1/p$. Then the natural functor
\[D^b(\M[1/p])\otimes A\to D^b(\M\otimes A)\]
is an equivalence of categories. These categories are idempotent-com\-ple\-te for any $A$.
\end{cor}

\begin{proof} By Proposition \ref{pB.4.1}, this is true by replacing the category $\M[1/p]$ by ${}^t\M[1/p]$. On the
other hand, the same argument as above shows that the functor $D^b(\M\otimes A)\to
D^b({}^t\M\otimes A)$ is an equivalence. This shows the first statement; the second one follows
from the fact that $D^b$ of an abelian category is idempotent-complete.
\end{proof}

\subsection{Discrete sheaves and permutation modules} The following proposition will be used in \S \ref{s2.4.1}.

\begin{propose}\label{pperm} Let $G$ be a profinite group. Denote by $D^b_c(G)$ the
derived category of finitely generated (topological discrete)
$G$-modules. Then $D^b_c(G)$ is thickly generated by $\Z$-free permutation modules.
\end{propose}

\begin{proof} The statement says that the smallest thick subcategory $\cT$ of $D^b_c(G)$ which contains permutation modules is equal to $D^b_c(G)$. Let $M$ be a finitely generated $G$-module: to prove that
$M\in \cT$, we immediately reduce to the
case where $G$ is finite. Let $\bar M = M/M_\tors$. Realise $\bar
M\otimes\Q$ as a direct summand of $\Q[G]^n$ for $n$ large enough. Up to
scaling, we may assume that the image of $\bar M$ in $\Q[G]^n$ is
contained in $\Z[G]^n$ and that there exists a submodule $N$ of $\Z[G]^n$
such that $\bar M\cap N=0$ and $\bar M\oplus N$ is of finite index in
$\Z[G]^n$. This reduces us to the case where $M$ is \emph{finite}.
Moreover, we may assume that $M$ is $\ell$-primary for some prime $\ell$.

Let $S$ be a Sylow $\ell$-subgroup of $G$. Recall that there exist two
inverse isomorphisms
\begin{gather*}
\phi:\Z[G]\otimes_{\Z[S]} M\iso \Hom_{\Z[S]}(\Z[G],M)\\
\phi(g\otimes m)(\gamma) =
\begin{cases} 
\gamma g m&\text{if $\gamma g\in S$}\\
0&\text{if $\gamma g\notin S$.}
\end{cases}\\
\psi:\Hom_{\Z[S]}(\Z[G],M)\iso \Z[G]\otimes_{\Z[S]} M\\
\psi(f)=\sum_{g\in S\backslash G}g^{-1}\otimes f(g).
\end{gather*}

On the other hand, we have the obvious unit and counit
homomorphisms
\begin{gather*}
\eta:M\to \Hom_{\Z[S]}(\Z[G],M)\\
\eta(m)(g)=gm\\
\epsilon: \Z[G]\otimes_{\Z[S]} M\to M\\
\epsilon(g\otimes m)= gm.
\end{gather*}

It is immediate that 
\[\epsilon\circ \psi\circ \eta = (G:S).\]

Since $(G:S)$ is prime to $\ell$, this shows that $M$ is a direct summand
of the induced module $\Z[G]\otimes_{\Z[S]} M\simeq
\Hom_{\Z[S]}(\Z[G],M)$. But it is well-known (see \eg \cite[\S 8.3, cor. to Prop. 26]{serrerep})
that $M$, as an $S$-module, is a successive extension of trivial
$S$-modules. Any trivial torsion $S$-module has a length $1$ resolution
by trivial torsion-free $S$-modules. Since the ``induced module" functor
is exact, this concludes the proof.
\end{proof}

\subsection{ Cartier duality and $1$-motives with cotorsion}\label{1.8}  We now introduce a new
category
${}_t\M$:

\begin{defn}\label{1cot}
We denote by ${}_t\M^\eff$ \index{${}_t\M$, ${}_t\M^\eff$} the full subcategory of $\anM^\eff$ consisting of those $[L\to G]$ such that $L$ is a lattice and $G$ is an
extension of an abelian variety by a group of multiplicative type, and
by ${}_t\M$ its localisation with respect to quasi-isomorphisms. An object of ${}_t\M$ is
called a \emph{$1$-motive with cotorsion}.
\end{defn}

Recall that Deligne \cite[\S 10.2.11-13]{D} (\cf \cite[1.5]{BSAP})
defined a self-duality on the category $\M$, that he called
\emph{Cartier duality}. The following facts elucidate the introduction
of the category ${}_t\M$.

\begin{lemma}\label{brst} Let $\Gamma$ be a group of multiplicative type,
$L$ its Cartier dual and $A$ an abelian
variety (over $k =\bar k$). We have an isomorphism
$$\tau:\Ext (A, \Gamma)\longby{\simeq} \Hom (L, \Pic^0(A))$$
given by the canonical ``pushout'' mapping.
\end{lemma}

\begin{proof} Displaying $L$ as an extension of $L_{\fr}$ by $L_{\tor}$
denote the corresponding torus by $T \df \Hom (L_{\fr},\G_m)$ and let $F\df \Hom
(L_{\tor},\G_m)$ be the dual finite group. We obtain a map of short exact sequences
\[\begin{CD}
0\to \Ext (A, T)&\to& \Ext (A, \Gamma)&\to& \Ext (A, F)\to 0\\
@V{\tau_{\fr}}VV @V{\tau}VV  @V{\tau_{\tor}}VV\\
0\to \Hom (L_{\fr}, \Pic^0(A))&\to& \Hom (L, \Pic^0(A))&\to&\Hom
(L_{\tor}, \Pic^0(A))\to 0.
\end{CD}\]

Now $\tau_{\fr}$ is an isomorphism by the classical Weil-Barsotti
formula, \ie $\Ext (A, \G_m)\cong \Pic^0(A)$, and $\tau_{\tor}$ is an
isomorphism since the N\'eron-Severi group of $A$ is free: $\Hom
(L_{\tor}, \Pic^0(A)) = \Hom (L_{\tor}, \Pic (A)) = H^1(A, F) = \Ext (A,
F)$ (\cf \cite[4.20]{MI}).
\end{proof}

\begin{lemma}\label{dualt}  Cartier duality on $\M$
extends to a contravariant additive functor 
\[
(\ \ )^*:{}^t\M^\eff \to {}_t\M^\eff
\]
which sends a \qi to a \qi
\end{lemma}

\begin{proof} The key point is that
$\Ext (- , \G_m)$ vanishes on discrete sheaves (\cf \cite[4.17]{MI}), hence Cartier duality
extends to an exact duality between discrete sheaves and groups of multiplicative type. 

To define the functor, we proceed  as usual (see
\cite[1.5]{BSAP}): starting with $M=[L\by{u}
A]\in {}^t\M^\eff$, le $G^u$ be the extension of the dual abelian
variety $A^*$ by the Cartier dual $L^*$ of $L$ given by Lemma
\ref{brst} (note that $G^u$ may be described as
the group scheme which represents the functor associated to $\Ext
(M,\G_m)$). We define $M^{*} = [0\to G^u]\in {}_t\M^\eff$. For a general
$M=[L\by{u} G]\in {}^t\M^\eff$, with $G$ an extension of $A$ by $T$, the
extension $M$ of $[L\by{\bar u} A]$ by the toric part $[ 0\to T]$ provides
the corresponding extension
$G^{\bar u}$ of $A'$ by $L^*$ and a boundary map
$$u^{*}:\Hom (T, \G_m)\to \Ext ([L\by{\bar u} A],\G_m)= G^{\bar u}(k)$$ 
which defines $M^*\in {}_t\M^\eff$.

For a quasi-isomorphism $M\onto M'$ with kernel $[F\by{=}F]$ for a finite
group $F$, \cf \eqref{qi1mot}, the quotient $[L\by{\bar u} A]\onto
[L'\by{\bar u'} A']$ has kernel
$[F\onto F_A]$ where $F_A\df \ker (A\onto A')$ and the following is a
pushout
\[\begin{CD}
0\to \Hom (T', \G_m)&\to& \Hom (T, \G_m)&\to&\Hom (F_T,\G_m)\to 0\\
@V{(u')^*}VV @V{u^*}VV  @V{\veq}VV\\
0\to \Ext ([L'\by{\bar u'} A'],\G_m)&\to& \Ext ([L\by{\bar u} A],\G_m)
&\to&\Ext ([F\onto F_A],\G_m)\to 0
\end{CD}\]
where $F_T\df \ker (T\onto T')$. 
\end{proof}

\begin{propose}\label{pcd} a) The functor of Lemma \ref{dualt} induces an
anti-equivalence of categories
\[(\ \ )^*:{}^t\M[1/p] \iso {}_t\M[1/p].\]
b) The category ${}_t\M[1/p]$ is abelian and the two functors of a) are exact.\\
c) Cartier duality on $\M[1/p]$ is an exact functor, hence induces a
triangulated self-duality on $D^b(\M[1/p])$.  
\end{propose}

\begin{proof} a)  The said functor exists by Lemma \ref{dualt}, and it is clearly additive.
Let us prove that it is i) essentially surjective, i) faithful and iii) full.

i) We proceed exactly as in the proof of Lemma \ref{dualt}, taking an $[L'\to G']\in {}_t\M[1/p]$,
and writing $G'$ explicitly as an extension of an abelian variety by a group of multiplicative
type. 

ii) We reduce to show that the functor of Lemma \ref{dualt} is faithful by using that Lemma
\ref{B1.1} is also true in $ {}_t\M^\eff[1/p]$ (dual proof). By additivity, we need to prove that if
$f:M_0\to M_1$ is mapped to $0$, then $f=0$. But, by construction, $f^*$ sends the
mutiplicative type part of $M_1^*$ to that of $M_0^*$.

iii) Let $M_0=[L_0\to G_0]$, $M_1=[L_1\to G_1]$ in $ {}^t\M^\eff[1/p]$, and let $f:M_1^*\to M_0^*$
be (for a start) an effective map. We have a diagram
\[\begin{CD}
0@>>> \Gamma_1@>>> G'_1@>>> A'_1@>>> 0\\
&& && @V{f_G}VV\\
0@>>> \Gamma_0@>>> G'_0@>>> A'_0@>>> 0
\end{CD}\]
where $M_i^*=[L'_i\to G'_i]$, $A'_i$ is the dual of the abelian part of $M_i$ and $\Gamma_i$
is the dual of $L_i$. If $f_G$ maps $\Gamma_1$ to $\Gamma_2$, there is no difficulty to get an
(effective) map $g:M_0\to M_1$ such that $g^*=f$. In general we reduce to this case: let $\mu$
be the image of $f_G(\Gamma_1)$ in $A'_0$: this is a finite group. Let now $A'_2=A'_0/\mu$, so
that we have a commutative diagram
\[\begin{CD}
0@>>> \Gamma_0@>>> G'_0@>>> A'_0@>>> 0\\
&& @VVV @V{||}VV @VVV\\
0@>>> \Gamma_2@>>> G'_0@>>> A'_2@>>> 0
\end{CD}\]
where $\mu=\ker(A'_0\to A'_2)=\coker(\Gamma_0\to \Gamma_2)$. By construction, $f_G$ induces
maps $f_\Gamma:\Gamma_1\to \Gamma_2$ and $f_A:A'_1\to A'_2$.

Consider the object $M_2=[L_2\to G_2]\in {}^t\M^\eff[1/p]$ obtained from  $(L'_0,\Gamma_2,A'_2)$
and the other data by the same procedure as in the proof of Lemma \ref{dualt}. We then have a
\qi $s:M_2\to M_0$ with kernel $[\mu = \mu]$ and a map $g:M_2\to M_1$ induced by
$(f_L,f_\Gamma,f_A)$, and $(gs^{-1})^*=f$. 

If $f$ is a \qi, clearly $g$ is a \qi; this concludes the proof of fullness.

b) Since ${}^t\M[1/p]$ is abelian, ${}_t\M[1/p]$ is abelian by a). Equivalences of abelian categories
are automatically exact.

c) One checks as for ${}^t\M[1/p]$ that the inclusion of $\M[1/p]$ into ${}_t\M[1/p]$ induces the exact
structure of $\M[1/p]$.  Then, thanks to b), Cartier duality preserves exact sequences of $\M[1/p]$,
which means that it is exact on $\M[1/p]$. 
\end{proof}

\begin{remarks}\label{r1.8.6} 1) Cartier duality does not preserve the strong exact structure of
Remark \ref{rexact} 1). For example, let $A$ be an abelian variety, $a\in
A(k)$ a point of order $m>1$ and $B=A/\langle a\rangle$. Then the sequence
\[0\to [\Z\to 0]\by{m} [\Z\by{f} A]\to [0\to B]\to 0,\]
with $f(1) = a$, is exact in the sense of Definition \ref{dexact} but not in the sense of Remark
\ref{rexact}. However, its dual
\[0\to [0\to B^*]\to [0\to G]\to [0\to \G_m]\to 0\]
is exact in the strong sense. Taking the Cartier dual of the latter sequence, we come back
to the former.

2) One way to better understand what happens in Lemma \ref{dualt} and Proposition \ref{pcd}
would be to introduce a category ${}_t\tilde\M^\eff$, whose objects are quintuples
$(L,u,G,A,\Gamma)$ with $L$ a lattice, $\Gamma$ a group of multiplicative type, $A$ an abelian
variety, $G$ an extension of $A$ by $\Gamma$ and $u$ a morphism from $L$ to $G$. Morphisms in
${}_t\tilde\M^\eff$ are additive and respect all these structures. There is an obvious functor
$(L,u,G,A,\Gamma)\mapsto [L\by{u} G]$ from ${}_t\tilde\M^\eff$ to ${}_t\M^\eff$, the functor
of Lemma \ref{dualt} lifts to an anti-isomorphism of categories ${}^t\M^\eff\iso
{}_t\tilde\M^\eff$ and the localisation of ${}_t\tilde\M^\eff[1/p]$ with respect to the images of
\qi of ${}^t\M^\eff[1/p]$ is equivalent to ${}_t\M[1/p]$. We leave details to the interested reader.
\end{remarks}

Dually to Theorem \ref{ptors}, we now have:

\begin{thm}\label{tstr}  The natural functor $\M[1/p]\to
{}_t\M[1/p]$ is fully faithful and induces an equivalence of categories
\[
D^b(\M[1/p])\iso D^b({}_t\M[1/p]).
\]
Moreover, Cartier duality exchanges ${}^t\M[1/p]$ and
${}_t\M[1/p]$ inside the derived category $D^b(\M[1/p])$.
\end{thm}

\begin{proof} This follows from Theorem \ref{ptors} and Proposition~\ref{pcd}.
\end{proof}

\begin{notation}\label{not} For $C\in D^b(\M[1/p])$, we write ${}^tH^n(C)$ (\resp ${}_tH^n(C)$  \index{${}^tH^n$, ${}_tH^n$} for
its cohomology objects relative to the 
$t$-structure with heart ${}^t\M[1/p]$ (\resp ${}_t\M[1/p]$). We also write ${}^tH_n$ for
${}^tH^{-n}$ and ${}_tH_n$ for ${}_tH^{-n}$.
\end{notation}

Thus we have \emph{two} $t$-structures on $D^b(\M[1/p])$ which are exchanged by Cartier duality;
naturally, these two $t$-structures coincide after tensoring with $\Q$. In Section
\ref{homotopy}, we shall introduce a third $t$-structure, of a
completely different kind: see Corollary \ref{c3.3.2}.

We shall also come back to Cartier duality in Section \ref{dual}.

\subsection{How not to invert $p$}\label{alessandra1} This has been done by Alessandra Ber\-ta\-pel\-le \cite{alessandra}. She defines a larger variant of $^t\M$ by allowing finite connected $k$-group schemes in the component of degree $0$. Computing in the fppf topology, she checks that the arguments provided in Appendix \ref{AppendixB} carry over in this context and yield in particular an integral analogue to Theorem \ref{1mtora}. Also, the analogue of \eqref{fullem} is fully faithful integrally, hence an exact structure on $\M$; she also checks that the analogue of Theorem \ref{ptors} holds integrally. 

In particular, her work provides an exact structure on $\M$, hence an integral definition of $D^b(\M)$. One could check that this exact structure can be described \emph{a priori} using Proposition \ref{pexact} and Lemma \ref{lexact}, and working with the fppf topology.

It is likely that the duality results of \S \ref{1.8} also extend to Bertapelle's context.

\section{Universal realisation}

\subsection{Statement of the theorem}
The derived category of $1$-motives up to isogeny can be realised in
Voevodsky's triangulated category of motives.  With rational coefficients, this is part of
Voevodsky's Pretheorem 0.0.18 in \cite{V0} and claimed in \cite[Sect.
3.4, on page 218]{V}. Details of this fact appear in Orgogozo \cite{OR}. In this section we
shall give a $p$-integral version of this theorem, where $p$ is the exponential characteristic
of $k$, using the \'etale version of Voevodsky's category.

By Lemma \ref{l1.3}, any $1$-motive $M =[L\to G]$ may be regarded as a
complex of homotopy invariant \'etale sheaves with transfers. By Lemma \ref{lD.1.3},
$M[1/p]\df M\otimes_\Z \Z[1/p]$ is a complex of strictly homotopy invariant \'etale sheaves
with transfers; this defines a functor
\begin{align}
\M(k)&\to \DM_{-,\et}^{\eff}(k)\label{eq8}\\
M&\mapsto M[1/p].\notag
\end{align}
(see \cite[Sect. 3]{V} for motivic
complexes).

From now on, we will usually drop the mention of $k$ from the notation for the various
categories of motives encountered.

\begin{defn}\label{d2.1.1} We denote by $\DM_{\gm,\et}^\eff$ \index{$\DM_{\gm,\et}^\eff$} the thick subcategory of
$\DM_{-,\et}^{\eff}$ generated by the image of $\DM_\gm^\eff$ under the ``change of
topology" functor
\[\alpha^*:\DM_-^\eff\to \DM_{-,\et}^\eff\] 
of \cite[\S 3.3]{V}. We set $M_\et(X)\df \alpha^* M(X)$.\index{$M_\et(X)$}
\end{defn}

\begin{thm}\label{t1.2.1} Let $p$ be the exponential characteristic of $k$. The functor
\eqref{eq8} extends to a fully faithful triangulated functor
\[T:D^b(\M[1/p])\to \DM^\eff_{-,\et}\]
where the left hand side was defined in \S \ref{1.2}. Its essential image is the thick
subcategory $d_{\le 1}\DM_{\gm,\et}^\eff$
\index{$d_{\le 1}\DM_{\gm,\et}^\eff$} of $\DM_{\gm,\et}^\eff$ 
generated by motives of smooth curves.
\end{thm}

The proof is in several steps.

\subsection{Construction of $T$} We follow Orgogozo.
Clearly, the embedding \eqref{eq2.2} extends to a functor
\[C^b(\M)\to C^b(\HI_\et^{[0,1]}).\]

By Lemma \ref{ltot}, we have a canonical functor $C^b(\HI_\et^{[0,1]})\by{\Tot}D^b(\HI_\et)$,
and there is a canonical composite functor 
\[\begin{CD}
D^b(\HI_\et)@>\otimes_\Z \Z[1/p]>>D^b(\HI_\et^s)\to \DM_{-,\et}^\eff
\end{CD}\] 
where $\HI_\et^s$ is the category of strictly homotopy invariant \'etale sheaves with transfers
(see Def. \ref{dD.1} and Proposition \ref{pD.1.4}). To get $T$,
we are therefore left to prove

\begin{lemma}\label{l2.2.1} The composite functor
\[C^b(\M)\to C^b(\HI_\et^{[0,1]})\by{\Tot}D^b(\HI_\et)\]
factors through $D^b(\M[1/p])$.
\end{lemma}

\begin{proof} It is a general fact that a homotopy in $C^b(\M)$ is mapped to a homotopy in
$C^b(\HI_\et^{[0,1]})$, and therefore goes to $0$ in $D^b(\HI_\et)$, so that the functor
already factors through $K^b(\M)$. The lemma now follows from Lemma \ref{l1.5.1}.
\end{proof}

\subsection{Full faithfulness}\label{2.3} It is sufficient by Proposition \ref{p1.2} to
show that $T\otimes\Q$ and $T_\tors$ are fully faithful. 

For the first fact, we reduce to \cite[3.3.3 ff]{OR}. We have to be a little careful since
Orgogozo's functor is not quite the same as our functor: Orgogozo sends $C$ to $\Tot(C)$ while
we send it to $\Tot(C)[1/p]$, but the map $\Tot(C)\to \Tot(C)[1/p]$ is an isomorphism in
$\DM_{-,\et}^\eff\otimes \Q$ by Proposition \ref{ptensq} (see also Remark \ref{r2.4} 2)). 

For the reader's convenience we sketch the
proof of \cite[3.3.3 ff]{OR}: it first uses the equivalence of categories
\[\DM_-^\eff\otimes\Q\iso \DM_{-,\et}^\eff\otimes \Q\]
of \cite[Prop. 3.3.2]{V} (\cf Proposition \ref{ptensq}).
One then reduces to show
that the morphisms
\[\Ext^i(M,M')\to \Hom(\Tot(M),\Tot(M')[i])\]
are isomorphisms for any pure $1$-motives $M,M'$ and any $i\in\Z$.
This is done by a case-by-case inspection, using the fact
\cite[3.1.9 and 3.1.12]{V} that in $\DM_-^{\eff}\otimes\Q$
\[\Hom(M(X),C)\otimes\Q=\HH^0_{\Zar}(X,C)\otimes\Q\]
for any smooth variety $X$. The key points are that 1) for such $X$ we have
$H^i_{\Zar}(X,\G_m)=0$ for
$i>1$ and for an abelian variety $A$, $H^i_{\Zar}(X,A)=0$ for $i>0$
because the sheaf $A$ is flasque, and 2) that any abelian variety is up
to isogeny a direct summand of the Jacobian of a curve. This point will  also be used for the
essential surjectivity below.

\begin{sloppypar}
For the second fact, the argument in the proof of \cite[Prop. 3.3.3 1]{V} shows that the
functor $\DM_{-,\et}^\eff\to D^-(\Shv((\Spec k)_\et))$ which takes a complex of sheaves on
$Sm(k)_\et$ to its restriction to $(\Spec k)_\et$ is an equivalence of categories on the
full subcategories of objects of prime-to-$p$ torsion. The conclusion
then follows from Proposition \ref{ptors}.
\end{sloppypar}

\subsection{Gersten's principle} We want to formalise here an important computational method
which goes back to Gersten's conjecture but was put in a wider perspective and systematic
use by Voevodsky. For the \'etale topology it replaces advantageously (but not completely) the
recourse to proper base change.

\begin{propose}\label{pgersten} a) Let $C$ be a complex of presheaves with transfers on $Sm(k)$
with homotopy invariant cohomology presheaves. Suppose that $C(K)\df \varinjlim_{k(U)=K} C(U)$
is acyclic for any function field $K/k$. Then the associated complex of Zariski sheaves $C_\Zar$
is acyclic.\\ 
b) Let $f:C\to D$ be a morphism of complex of presheaves with transfers on $Sm(k)$
with homotopy invariant cohomology presheaves. Suppose that for any function field $K/k$,
$f(K):C(K)\to D(K)$ is a quasi-isomorphism. Then $f_\Zar:C_\Zar\to D_\Zar$ is a
quasi-isomorphism.  \\
c) The conclusions of a) and b) hold for the \'etale topology if their hypotheses are weakened
by replacing $K$ by $K_s$, a separable closure of $K$.
\end{propose}

\begin{proof} a) Let $F = H^q(C)$ for some $q\in\Z$, and let $X$ be a smooth $k$-variety with
function field $K$. By \cite[Cor. 4.18]{V2}, $F(\cO_{X,x})\into F(K)$ for any $x\in X$, hence
$F_\Zar=0$. b) follows from a) by considering the cone of $f$. c) is seen similarly.
\end{proof}

\subsection{An important computation}\label{2.5} Recall that the category $\DM_{-,\et}^{\eff}$ is provided with a partial internal Hom denoted by $\ihom_\et$, \index{$\ihom_\et$} defined on pairs  $(M,M')$ with $M\in\DM_{\gm,\et}^{\eff}$: it is defined analogously to the one of \cite[Prop. 3.2.8]{V} for the Nisnevich topology. We need:

\begin{defn}\label{dpi0} Let $X\in Sch(k)$. We denote by $\pi_0(X)$ the largest \'etale $k$-scheme such that the structural map $X\to \Spec k$ factors through $\pi_0(X)$.
\index{$\pi_0(X)$}
\end{defn}

(The existence of $\pi_0(X)$ is obvious, for example by Galois descent.)

\begin{propose}\label{lcurve} Let $f:C\to \Spec k$ be a smooth projective
$k$-curve. Then, in $\DM_{-,\et}^\eff$: \\
a) There is a canonical isomorphism
\[\ihom_\et(M_\et(C),\Z(1)[2])\simeq R_\et f_*\G_m[1/p][1].\]
b) we have
\[R^q_\et f_*\G_m[1/p]=
\begin{cases}
R_{\pi_0(C)/k}\G_m[1/p] &\text{for $q=0$}\\
\underline{\Pic}_{C/k}[1/p] &\text{for $q=1$}\\
0&\text{else.}
\end{cases}
\]
Here, $R_{\pi_0(C)/k}$ denotes the Weil restriction of scalars from $\pi_0(C)$ to $k$.\\
c) The morphism
\[M_\et(C)\to \ihom_\et(M_\et(C),\Z(1)[2])\]
induced by the class $\Delta_C\in \Hom(M_\et(C)\otimes M_\et(C),\Z(1)[2])$ of the diagonal is an
isomorphism.
\end{propose}

\begin{proof} This is \cite[Cor. 3.1.6]{OR} with three differences: 1) the fppf topology should
be replaced by the \'etale topology; $p$ must be inverted (\cf Corollary \ref{cD.1}); 3) the
truncation is not necessary since $C$ is a curve.

a) is the \'etale analogue of \cite[Prop. 3.2.8]{V}
since $\Z_\et(1)=\G_m[1/p][-1]$ (see Corollary \ref{cD.1}) and $f^* (\G_{m, k}) =\G_{m, C}$ for
the big \'etale sites. In b), the isomorphisms for $q=0,1$ are clear; for $q>2$, we reduce by
Gersten's principle (Prop. \ref{pgersten}) to stalks at separably closed fields, and
then the result is classical
\cite[IX (4.5)]{sga4}.

It remains to prove c). Recall that its Nisnevich analogue is true in $\DM_\gm^\eff$
(\cite[Th. 4.3.2 and Cor. 4.3.6]{V}, but see \cite[App. B]{hk} to avoid resolution of
singularities). Let $\alpha^*:\DM_-^\eff\to \DM_{-,\et}^\eff$ be the change of topology functor
(\cf \cite[Remark 14.3]{VL}). By b), the natural morphism
\begin{equation}\label{curves}
\alpha^*\ihom_\Nis(M(C), \Z (1))\to \ihom_\et(\alpha^*M(C), \Z (1))
\end{equation}
is an isomorphism. Hence the result.
\end{proof}

\subsection{Essential image} 
We proceed in two steps:

\subsubsection{The essential image of $T$ is contained in 
$\cT:=d_{\le 1}\DM_{\gm,\et}^\eff$}\label{s2.4.1} It is
sufficient to prove that
$T(N)\in \cT$ for
$N$ a $1$-motive of type $[L\to 0]$, $[0\to G]$ ($G$ a torus) or $[0\to
A]$ ($A$ an abelian variety). For the first type, this follows from
Proposition \ref{pperm}. For the second type, Proposition \ref{pperm}
applied to the character group of $G$ shows that $T([0\to G])$ is
contained in the thick subcategory generated by permutation tori, which
is clearly contained in $\cT$.

It remains to deal with the third type. If $A=J(C)$ for a smooth
projective curve $C$ having a rational $k$-point $c$, then $T([0\to
A])=A[-1]$ is the direct summand of $M(C)[-1]$ (determined by $c$)
corresponding to the pure motive $h^1(C)$, so belongs to $\cT$. If $A\to
A'$ is an isogeny, then Proposition \ref{pperm} implies that
$A[-1]\in\cT$ $\iff$ $A'[-1]\in \cT$. In general we may write $A$ as the
quotient of a jacobian $J(C)$. Let $B$ be the connected part of the
kernel: by complete reducibility there exists a third abelian variety
$B'\subseteq J(C)$ such that $B+B'=J(C)$ and $B\cap B'$ is finite. Hence
$B\oplus B'\in \cT$, $B'\in\cT$ and finally $A\in \cT$ since it is
isogenous to $B'$.

\subsubsection{The essential image of $T$ contains $\cT$} It suffices to
show that $M(X)$ is in the essential image of $T$ if $X$ is smooth
projective irreducible of dimension $0$ or $1$. Let $E$ be the field of
constants of $X$. If $X=\Spec E$, $M(X)$ is the image of $[R_{E/k}\Z\to
0]$. If $X$ is a curve, we apply Proposition \ref{lcurve}: by c) it suffices to show that the
sheaves of b) are in the essential image of $T$.
We have already observed that $R_{E/k}\G_m[1/p]$ is in the essential image of $T$. We then have
a short exact sequence
\[0\to R_{E/k} J(X)[1/p]\to \underline{\Pic}_{X/k}[1/p]\to R_{E/k}\Z[1/p]\to 0.\]

Both the kernel and the cokernel in this extension belong to the image of $T$, and the proof
is complete.
\qed

\subsection{The universal realisation functor} 

\begin{defn}\label{tot} Define the \emph{universal realisation functor}
$$\Tot : D^b(\M[1/p]) \to \DM_{\gm,\et}^{\eff}$$\index{$\Tot$}
to be the composition of the equivalence of categories of Theorem \ref{t1.2.1} and the embedding
$d_{\le 1}\DM_{\gm,\et}^\eff\to \DM_{\gm,\et}^\eff$.
\end{defn}

\begin{remarks} \label{r2.4} 1) In view of Theorem \ref{tstr}, the equivalence of Theorem
\ref{t1.2.1}, yields \emph{two} ``motivic" $t$-structures on $d_{\leq
1}\DM_{\gm,\et}^{\eff}$: one with heart ${}^t\M[1/p]$ and the other with
heart ${}_t\M[1/p]$. We shall describe a third one, the homotopy $t$-structure, in Theorem
\ref{t3.2.3}.

2) In what follows we shall frequently commit an abuse of notation in writing $\uG$ rather
that $\uG[1/p]$, etc. for the image of (say) a semi-abelian variety in $\DM_{\gm,\et}^\eff$
by the functor $\Tot$. This is to keep notation light. A more mathematical justification is
that, according to Proposition \ref{pD.2}, the functor $T$ is naturally isomorphic to the
composition
\begin{multline*}
D^b(\M[1/p])\to D^b(\HI_\et^{[0,1]}[1/p])\to D^b(\HI_\et[1/p])\\
\to
D^-(\Shv_\et(Sm(k))[1/p])\longby{C_*}\DM_{-,\et}^\eff
\end{multline*}
which (apparently) does not invert $p$ on objects.
\end{remarks}

\section{$1$-motivic sheaves and the homotopy $t$-structure}\label{homotopy}

We recall the blanket assumption that $p$ is inverted in all Hom groups.

\subsection{Some useful lemmas} Except for Proposition \ref{ptorsion}, this subsection is in
the spirit of \cite[Ch. VII]{gacl}.

Let $G$ be a commutative $k$-group scheme, and let us write $\uG$\index{$\uG$} for
the associated sheaf of abelian groups for a so far unspecified
Grothendieck topology. Let also $\cF$ be another sheaf of abelian groups.
We then have:

\begin{itemize}
\item $\Ext^1(\uG,\cF)$ (an Ext of sheaves);
\item $H^1(G,\cF)$ (cohomology of the scheme $G$);
\item $\bar H^2(G,\cF)$: this is the homology
of the complex
\[\cF(G)\by{d^1} \cF(G\times G)\by{d^2} \cF(G\times G\times G)\]
where the differentials are the usual ones.
\end{itemize}

\begin{propose}\label{p1} There is an exact sequence (defining $A$)
\[0\to A\to \Ext^1(\uG,\cF)\by{b} H^1(G,\cF)\by{c} H^1(G\times G,\cF)\]
and an injection
\[0\to A\by{a} \bar H^2(G,\cF).\]
\end{propose}

\begin{proof} Let us first define the maps $a,b,c$:

\begin{itemize}
\item $c$ is given by $p_1^*+p_2^*-\mu^*$, where $\mu$ is the group law of
$G$.
\item For $b$: let $\cE$ be an extension of $\uG$ by $\cF$. We have an
exact sequence
\[\cE(G)\to \uG(G)\to H^1(G,\cF).\]
Then $b([\cE])$ is the image of $1_G$ by the connecting homomorphism. Alternatively, we may
think of $\cE$ as an $\cF$-torsor over $G$ by forgetting its group structure.
\item For $a$: we have $b([\cE])=0$ if and only if $1_G$ has an antecedent
$s\in \cE(G)$. By Yoneda, this $s$ determines a section $s:\uG\to \cE$ of
the projection. The defect of $s$ to be a homomorphism gives a
well-defined element of  $\bar H^2(G,\cF)$ by the usual cocycle
computation: this is $a([\cE])$.
\end{itemize}

Exactness is checked by inspection.
\end{proof}

\begin{remark} It is not clear whether $a$ is surjective.
\end{remark}

\begin{propose}\label{p2} Suppose that the map
\[\cF(G)\oplus\cF(G)\vlongby{(p_1^*,p_2^*)}\cF(G\times G)\]
is surjective. Then $ \bar H^2(G,\cF)=0$.
\end{propose}

\begin{proof} Let $\gamma\in \cF(G\times G)$ be a $2$-cocycle. We may
write $\gamma=p_1^*\alpha+p_2^*\beta$. The cocycle condition implies that
$\alpha$ and $\beta$ are constant. Hence $\gamma$ is constant, and it is
therefore a $2$-coboundary (of itself).
\end{proof}

\begin{example}\label{ex3.1.4}
$\cF$ locally constant, $G$ smooth, the topology = the \'etale topology. Then
the condition of Proposition \ref{p2} is verified. We thus get an isomorphism
\[\Ext^1(\uG,\cF)\iso H^1_\et(G,\cF)_{mult}\]
with the group of multiplicative classes in $H^1_\et(G,\cF)$.
\end{example}

\begin{lemma}\label{c3.2} Let $G$ be a semi-abelian $k$-variety and $L$ a
locally constant $\Z$-constructible \'etale sheaf with torsion-free
geometric fibres. Then $\Ext^1(\uG,L)=0$.
\end{lemma}

\begin{proof} By the $\Ext$ spectral sequence, it suffices to show that
$\shom(\uG,L)\allowbreak=\sext(\uG,L)=0$. This reduces us to the case $L=\Z$. Then the first
vanishing is obvious and the second follows from Example \ref{ex3.1.4}.
\end{proof}

\begin{lemma}\label{trans} Let $\cE\in \Ext^1(\uG,\G_m)$ and let $g\in G(k)$. Denote by
$\tau_g$ the left translation by $g$. Then $\tau_g^*b(\cE)=b(\cE)$. Here $b$ is the map of
Proposition \ref{p1}.
\end{lemma}

\begin{proof} By Hilbert's theorem 90, $g$ lifts to an $e\in \cE(k)$. Then $\tau_e$ induces a
morphism from the $\G_m$-torsor $b(\cE)$ to the $\G_m$-torsor $\tau_g^*b(\cE)$: this morphism
must be an isomorphism.
\end{proof}

For the proof of Theorem \ref{text} below we shall need the case $i=2$ of the following
proposition, which unfortunately cannot be proven with the above elementary methods.

\begin{propose}\label{ptorsion} Let $G$ be a smooth commutative algebraic $k$-group and $L$ a
discrete $k$-group scheme. Let $\cA=\Shv_\et(Sm(k))$ be the category of abelian \'etale sheaves
on the category of smooth $k$-varieties. Then, for any $i\ge 2$, the group $\Ext^i_\cA(\uG,L)$
is torsion.
\end{propose}

\begin{proof} Considering the connected part $G^0$ of $G$, we reduce to the case where $G$ is connected, hence geometrically connected. We now turn to the techniques of \cite{breen}\footnote{We thank L. Illusie for pointing out this reference.}: using essentially the Eilenberg-Mac Lane spectrum associated to $\uG$, Breen gets two spectral sequences $'E_r^{p,q}$ and $''E_r^{p,q}$ converging to the same abutment, with
\begin{itemize}
\item $''E_2^{p,1}=\Ext^p_\cA(\uG,L)$;
\item $''E_2^{p,q}$ is torsion for $q\ne 1$;
\item $'E_2^{p,q}$ is the $p$-th cohomology group of a complex involving terms of the form $H^q_\et(G^a,L)$.
\end{itemize}

(In \cite{breen}, Breen works with the fppf topology but his methods carry over here without
any change: see remark in  \loccit top of p. 34.) It follows from \cite[(2.1)]{deninger} that
$H^q_\et(G^a,L)$ is torsion for any $q>0$: to see this easily, reduce to the case where $L$ is
constant by a transfer argument involving a finite extension of $k$. Hence $'E_2^{p,q}$ is
torsion for $q>0$. On the other hand, since $G$ is geometrically connected, so are its powers
$G^a$, which implies that $H^0(G^a,L)=H^0(k,L)$ for any $a$. Since the complex giving
$'E_2^{*,0}$ is just the bar complex, we get that $'E_2^{0,0} = L(k)$ and $'E_2^{p;0}=0$ for
$p>0$. Thus all degree $>0$ terms of the abutment are torsion, and the conclusion follows.
\end{proof}

\subsection{$1$-motivic sheaves} 

\begin{defn}\label{d3.1} An \'etale sheaf $\cF$ on $Sm(k)$ is
\emph{$1$-motivic} if  there is a morphism of sheaves
\begin{equation}\label{2}
\uG\by{b} \cF
\end{equation}
where $G$ is a semi-abelian variety and $\ker b, \coker b$ are discrete
(see Definition \ref{d1.1.1}).\\ We denote by $\Shv_0$ the full
subcategory of $\Shv_\et(Sm(k))[1/p]$ consisting of discrete sheaves and by
$\Shv_1$ the full subcategory of $\Shv_\et(Sm(k))[1/p]$ consisting of
$1$-motivic sheaves.\index{$\Shv_1$, $\Shv_0$}
\end{defn}

\begin{remark} The category $\Shv_0$ is equivalent to the category ${}^t\cM_0[1/p]$ of
Definition \ref{d1.1.1}.
\end{remark}

\begin{propose}\label{p3.1} a) In Definition \ref{d3.1} we may choose $b$ such that $\ker
b$ is torsion-free: we then say that $b$ is \emph{normalised}.\\ 
b)  Given two $1$-motivic sheaves $\cF_1,\cF_2$, normalised morphisms
$b_i:\uG_i\to \cF_i$ and a map $\phi:\cF_1\to \cF_2$,
there exists a unique homomorphism of group schemes $\phi_G:G_1\to G_2$ such that the diagram
\[\begin{CD}
\uG_1@>{b_1}>> \cF_1\\
@V{\phi_G}VV@V{\phi}VV\\
\uG_2@>{b_2}>> \cF_2
\end{CD}\]
commutes.\\
c) Given a $1$-motivic sheaf $\cF$, a pair $(G,b)$ with $b$ normalised
is uniquely determined by $\cF$.\\
d) The categories $\Shv_0$ and $\Shv_1$ are exact abelian subcategories of $\Shv_\et(Sm(k))$.
\end{propose}
\begin{proof} 
a) If $\ker b$ is not torsion-free, simply divide $G$ by the image of its torsion. 

b) We want to construct a commutative diagram
\begin{equation}\label{eq4.2}
\begin{CD}
0@>>> L_1@>a_1>> \uG_1@>{b_1}>> \cF_1@>{c_1}>> E_1@>>> 0\\
&&@V{\phi_L}VV@V{\phi_G}VV@V{\phi}VV@V{\phi_E}VV\\
0@>>> L_2@>a_2>> \uG_2@>{b_2}>> \cF_2@>{c_2}>> E_2@>>> 0
\end{CD}
\end{equation}
where $L_i=\ker b_i$ and $E_i=\coker b_i$. It is clear that
$c_2\phi b_1=0$: this proves the existence of $\phi_E$. We also get a
homomorphism of sheaves $\uG_1\to \uG_2/L_2$, which lifts to
$\phi_G:\uG_1\to \uG_2$ by Lemma \ref{c3.2}, hence $\phi_L$.

From the construction, it is clear that $\phi_E$ is uniquely determined
by $\phi$ and that $\phi_L$ is uniquely determined by $\phi_G$. It
remains to see that $\phi_G$ is unique. Let $\phi'_G$ be another choice.
Then $b_2(\phi_G-\phi'_G)=0$, hence $(\phi_G-\phi'_G)(\uG_1)\subseteq
L_2$, which implies that $\phi_G=\phi'_G$.

c) Follows from b).

d) The case of $\Shv_0$ is obvious. For $\Shv_1$, given a map $\phi$ as in
b), we want to show that
$\cF_3=\ker
\phi$ and $\cF_4=\coker\phi$ are $1$-motivic. Let $G_3=(\ker\phi_G)^0$ and
$G_4=\coker\phi_G$: we get induced maps $b_i:\uG_i\to \cF_i$ for $i=3,4$.
An easy diagram chase shows that $\ker b_i$ and $\coker b_i$ are both discrete. 
\end{proof}

Here is an extension of Proposition \ref{p3.1} which elucidates the
structure of $\Shv_1$ somewhat:

\begin{thm}\label{t3.2.4} a) Let $\SAb$\index{$\SAb$} be the category of
semi-abelian $k$-varie\-ties. Then the inclusion functor
\begin{align*}
\SAb&\to \Shv_1\\
G&\mapsto \underline{G}
\end{align*}
has a right adjoint/left inverse $\gamma$; the counit of this adjunction
is given by \eqref{2} (with $b$ normalised). The functor $\gamma$ is faithful and ``exact up to
isogenies". For a morphism
$\phi\in
\Shv_1$, $\gamma(\phi)=\phi_G$ is an isogeny if and only if $\ker\phi$ and
$\coker\phi\in \Shv_0$. In particular, $\gamma$ induces an
equivalence of categories 
\[\Shv_1/\Shv_0\iso\SAb\otimes\Q\] 
where $\SAb\otimes\Q$ is the category of semi-abelian varieties up to
isogenies.\\ b) The inclusion functor $\Shv_0\to \Shv_1$ has a left
adjoint/left inverse $\pi_0$; the unit of this adjunction is given by
$\coker b$ in \eqref{2}. The right exact functor
\[(\pi_0)_\Q:\Shv_1\to \Shv_0\otimes\Q\]
has one left derived functor $(\pi_1)_\Q$ given by $\ker b$ in \eqref{2}.
\end{thm}

\begin{proof} a) The only delicate thing is the exactness of $\gamma$ up to isogenies. This
means that, given a short exact sequence $0\to \cF'\to \cF\to \cF''\to 0$ of $1$-motivic
sheaves, the sequence
\[0\to \gamma(\cF')\to \gamma(\cF)\to\gamma(\cF'')\to 0\]
is half exact and the middle homology is finite. This follows from a chase in the diagram
\[\begin{CD}
0@>>> L'@>a'>> \uG'@>{b'}>> \cF'@>{c'}>> E'@>>> 0\\
&&@VVV@VVV@VVV@VVV\\
0@>>> L@>a>> \uG@>{b}>> \cF@>{c'}>> E@>>> 0\\
&&@VVV@VVV@VVV@VVV\\
0@>>> L''@>a''>> \uG''@>{b''}>> \cF''@>{c''}>> E''@>>> 0
\end{CD}\]
of which we summarize the main points: (1) $G'\to G$ is injective because its kernel is the
same as $\ker(L'\to L)$. (2) $G\to G''$ is surjective because (i) $\Hom(\uG''\to \coker(E'\to
E))=0$ and (ii) if $L''\to \coker(\uG\to \uG'')$ is onto, then this cokernel is $0$. (3) The
middle homology is finite because the image of $\ker(\uG'\to \uG)\to E'$ must be finite.

In b), the existence and characterisation of $(\pi_1)_\Q$ follows from the exactness of $\gamma$
in a).
\end{proof}

\begin{remark} One easily sees that $\pi_1$ does not exist integrally. Rather, it exists as a
functor to the category of pro-objects of $\Shv_0$. (Actually to a finer subcategory: compare
\cite{proalg}.)
\end{remark}

\subsection{Extensions of $1$-motivic sheaves} The aim of this subsection is to prove:

\begin{thm}\label{text}
The categories $\Shv_0$ and $\Shv_1$ are thick in the abel\-ian category $\Shv_\et(Sm(k))[1/p]$.
\end{thm}

\begin{proof} For simplicity, let us write $\cA\df \Shv_\et(Sm(k))[1/p]$ as in  Proposition
\ref{ptorsion}. The statement is obvious for $\Shv_0$. Let us now show that $\Shv_1$ is closed
under extensions in $\cA$. Let
$\cF_1,\cF_2$ be as in
\eqref{eq4.2} (no map given between them). We have to show that the injection
\begin{equation}\label{eq3.1}
\Ext^1_{\Shv_1}(\cF_2,\cF_1)\into \Ext^1_\cA(\cF_2,\cF_1)
\end{equation}
is surjective. This is certainly so in the following special cases:
\begin{enumerate}
\item $\cF_1$ and $\cF_2$ are semi-abelian varieties;
\item $\cF_2$ is semi-abelian and $\cF_1$ is discrete (see Example \ref{ex3.1.4}).
\end{enumerate}

For $m>1$, consider 
\[\cF^m=\coker(L_1\vlongby{(a_1,m)}\uG_1\oplus L_1)\]
so that we have two exact sequences
\[\begin{CD}
0@>>> \uG_1@>(1_{G_1},0)>> \cF^m@>>> L_1/m@>>> 0\\
0@>>> L_1@>(a_1,0)>> \cF^m@>>> \uG_1/L_1\oplus L_1/m@>>> 0.
\end{CD}\] 

The first one shows that \eqref{eq3.1} is surjective for $(\cF_2,\cF_1)=(\uG_2,\cF^m)$. Let us now consider the commutative diagram with exact rows associated to the second one, for an unspecified $m$:
\begin{equation}
\begin{CD}
\Ext^1_{\Shv_1}(\uG_2,\cF^m)&\to& \Ext^1_{\Shv_1}(\uG_2,\uG_1/L_1\oplus L_1/m)&\longby{}& \Ext^2_{\Shv_1}(\uG_2,L_1)\\
@V{\wr}VV @VVV @VVV\label{eq3.2}\\
\Ext^1_\cA(\uG_2,\cF^m)&\to& \Ext^1_\cA(\uG_2,\uG_1/L_1\oplus L_1/m)&\longby{\delta^m}& \Ext^2_\cA(\uG_2,L_1).
\end{CD}
\end{equation}

Note that the composition 
\[\Ext^1_\cA(\uG_2,\uG_1/L_1)\to \Ext^1_\cA(\uG_2,\uG_1/L_1\oplus L_1/m)\longby{\delta^m} \Ext^2_\cA(\uG_2,L_1)\]
coincides with the boundary map $\delta$ associated to the exact sequence
\[0\to L_1\to \uG_1\to \uG_1/L_1\to 0.\]

Let $e\in \Ext^1_\cA(\uG_2,\uG_1/L_1)$. By Proposition \ref{ptorsion}, $f=\delta(e)$ is torsion. Choose now $m$ such that $mf=0$. Then there exists $e'\in \Ext^1_\cA(\uG_2,L_1/m)$ which bounds to $f$ via the Ext exact sequence associated to the exact sequence of sheaves
\[0\to L_1\by{m} L_1\to L_1/m\to 0.\]

Since $\delta^m(e,-e')=0$, \eqref{eq3.2} shows that $(e,-e')$ comes from the left, which shows that \eqref{eq3.1} is surjective for $(\cF_2,\cF_1)=(\uG_2,\uG_1/L_1)$.

By Lemma \ref{c3.2}, in the commutative diagram
\[\begin{CD}
\Ext^1_{\Shv_1}(\uG_2,\uG_1/L_1)@>>> \Ext^1_{\Shv_1}(\uG_2,\cF_1)\\
@V{\wr}VV @VVV\\
\Ext^1_\cA(\uG_2,\uG_1/L_1)@>>> \Ext^1_\cA(\uG_2,\cF_1)
\end{CD}\]
the horizontal maps are isomorphisms. Hence  \eqref{eq3.1} is surjective for $\cF_2=\uG_2$ and any $\cF_1$. 

To conclude, let $\cF$ be an extension of $\cF_2$ by $\cF_1$ in $\cA$. By the above, $\cF'\df b_2^*\cF$ is $1$-motivic as an extension of $\uG_2$ by $\cF_1$, and we have an exact sequence
\[0\to L_2\to \cF'\to\cF\to E_2\to 0.\]

Let $b':\uG\to \cF'$ be a normalised map (in the sense of Proposition \ref{p3.1}) from a
semi-abelian variety to $\cF'$ and let $b:\uG\to \cF$ be its composite with the above map. It
is then an easy exercise to check that $\ker b$ and $\coker b$ are both discrete. Hence $\cF$
is $1$-motivic.
\end{proof}

\begin{remark}\label{fppf} We may similarly define $1$-motivic sheaves for the fppf
topology over $\Spec k$; as one easily checks, all the above results hold equally
well in this context. This is also the case for \S \ref{3.7} below. 

In fact, let $\Shv_1^{\rm fppf}$ be the $\Z[1/p]$-linear category of fppf
$1$-motivic sheaves and $\pi:(\Spec k)_{\rm fppf}\to Sm(k)_\et$ be the projection functor. Then the
functors $\pi^*$ and $\pi_*$ induce \emph{quasi-inverse equivalences of categories} between
$\Shv_1$ and $\Shv_1^{\rm fppf}$. Indeed it suffices to check that $\pi_*\pi^*$ is naturally
isomorphic to the identity on $\Shv_1$: if $\cF\in \Shv_1$ and we consider its normalised
representation, then in the commutative diagram
\[\begin{CD}
0@>>> L@>>> \uG@>>> \cF@>>> E@>>> 0\\
&& @VVV @VVV @VVV @VVV\\
0@>>> \pi_*\pi^*L@>>> \pi_*\pi^*\uG@>>> \pi_*\pi^*\cF@>>> \pi_*\pi^*E@>>> 0
\end{CD}\]
the first, second and fourth vertical maps are isomorphisms and the lower
sequence is still exact: both facts follow from \cite[p. 14, Th. III.3.9]{MI}.

In particular the restriction of $\pi_*$ to $\Shv_1^{\rm fppf}$ is exact. Actually,
$(R^q\pi_*)_{|\Shv_1^{\rm fppf}}=0$ for $q>0$ (use same reference).  
\end{remark}

\subsection{A basic example} 

\begin{propose}\label{p3.3.1} Let $X\in Sm(k)$. Then the sheaf $\Pic_{X/k}$ is
$1$-motivic.\index{$\Pic_{X/k}$, $\NS_{X/k}$}
\end{propose}

\begin{proof} Suppose first that $X$ is smooth projective. Then $\Pic_{X/k}$ is an extension
of the discrete sheaf $\NS_{X/k}$ (N\'eron-Severi) by the abelian variety $\Pic^0_{X/k}$ (Picard variety). 

In general, we apply de Jong's theorem \cite[Th. 4.1]{DJ}: there exists a
diagram
\[\begin{CD}
\tilde U @>>> \bar X\\
@V{p}VV\\
   U  @>>>   X
\end{CD}\]
where the horizontal maps are open immersions, $\bar X$ is smooth
projective and the vertical map is finite \'etale. Then we get a
corresponding diagram of Pics
\[\begin{CD}
\Pic_{\tilde U/k} @<<< \Pic_{\bar X/k}\\
@A{p^*}AA\\
\Pic_{U/k}  @<<<   \Pic_{X/k}.
\end{CD}\]

The horizontal morphisms are epimorphisms and their kernels are
lattices. This already shows by Proposition \ref{p3.1} d) that $\Pic_{\tilde U/k}\in \Shv_1$. 

Consider the \v Cech spectral sequence associated to the \'etale cover $p$. It yields an exact sequence
\begin{multline*}
0\to \check H^1(p,\underline{H}^0_\et(\tilde U,\G_m))\to \Pic_{U/k}\to \check H^0(p,\Pic_{\tilde U/k})\\
\to \check H^2(p,\underline{H}^0_\et(\tilde U,\G_m)).
\end{multline*}

All the $\check H^i$  are cohomology sheaves of complexes of objects of the abelian category
$\Shv_1$, hence belong to $\Shv_1$; it then follows from Theorem \ref{text} that  $\Pic_{U/k}\in
\Shv_1$, as well as $\Pic_{X/k}$.
\end{proof}

\subsection{Application: the N\'eron-Severi group of a smooth scheme}

\begin{defn}\label{dNS} Let $X\in Sm(k)$.\\
a) Suppose that $k$ is algebraically closed. Then we write $\NS(X)$ for the group of cycles of
codimension $1$ on $X$ modulo algebraic equivalence.\\ b) In general, we define $\NS_{X/k}$ as
the \'etale sheaf on $Sm(k)$ given by
\[\NS_{X/k}(U)=\NS(X\times_k \overline{k(U)})^G\]
where $U\in Sm(k)$ is irreducible, $\overline{k(U)}$ is a separable closure of $k(U)$ and
$G=Gal(\overline{k(U)}/k(U))$.
\end{defn}

\begin{sloppypar}
\begin{propose} \label{belong} The natural map
$e:\underline{\Pic}_{X/k}\to \NS_{X/k}$ identifies $\NS_{X/k}$ with
$\pi_0(\underline{\Pic}_{X/k})$ (\cf Theorem \ref{t3.2.4} b)). In particular, $\NS_{X/k}\in
\Shv_0$. 
\end{propose}
\end{sloppypar}

\begin{proof} It is well-known that cycles modulo algebraic equivalence are
invariant by extension of algebraically closed base field. By Proposition \ref{pi0} b), this
implies that $e$ induces a map $\bar e:\pi_0(\underline{\Pic}_{X/k})\to \NS_{X/k}$, which is
evidently epi. But let $\underline{\Pic}^0_{X/k}=\ker e$: by \cite[Lemma 7.10]{bo},
$\Pic^0(X_{\bar k})=\underline{\Pic}^0_{X/k}(\bar k)$ is divisible, which forces $\bar e$ to be
an isomorphism.\index{$\Pic_{X/k}$, $\NS_{X/k}$}
\end{proof}

\begin{remark}\label{r6.1} In particular, $\NS(X)$ is finitely generated if $k$ is
algebraically closed: this was proven in \cite[Th. 3]{pic} in a quite different way.
\end{remark}

\subsection{Technical results on $1$-motivic sheaves}

\begin{propose}\label{p4.2} The functor
\begin{align*}
ev: \Shv_1&\to Ab\\
\cF&\mapsto \cF(\bar k)
\end{align*}
to the category $Ab$ of abelian groups is faithful, hence (\cf \cite[Ch. 1, p. 44, prop. 1]{bbki})
``faithfully exact": a se\-quen\-ce $\cE$ is exact if and only if $ev(\cE)$ is exact.
\end{propose}

\begin{proof} The exactness of $ev$ is clear. For faithfulness, let
$\phi:\cF_1\to \cF_2$ be such that $ev(\phi)=0$. In $ev$\eqref{eq4.2}, we
have $\phi_G(\uG_1(\bar k))\subseteq L_2(\bar k)$; since the former group
is divisible and the latter is finitely generated, $ev(\phi_G)=0$. Hence
$\phi_G=0$. On the other hand, $ev(\phi_E)=0$, hence $\phi_E=0$. This
implies that $\phi$ is of the form $\psi c_1$ for $\psi:E_1\to \cF_2$.
But $ev(\psi)=0$, which implies that $\psi=0$.
\end{proof}

The following strengthens Theorem \ref{t3.2.4} b):

\begin{propose} \label{pi0} a) Let $G$ be a commutative algebraic
$k$-group and let $E$ be a $Gal(\bar k/k)$-module, viewed as an \'etale
sheaf over $Sm(k)$ ($E$ is not supposed to be constructible). Then $\Hom(\uG,E)=0$.\\ b) Let
$\cF\in
\Shv_1$ and
$E$ as in a). Then any morphism $\cF\to E$ factors canonically through
$\pi_0(\cF)$.
\end{propose}

\begin{proof} a) Thanks to Proposition \ref{p4.2} we may assume $k$
algebraically closed. By Yoneda, $\Hom(\uG,E)$ is a subgroup of $E(G)$
(it turns out to be the subgroup of multiplicative sections but we don't
need this). Since $E(k)\iso E(G)$, any homomorphism from $\uG$ to $E$ is
constant, hence $0$.

b) follows immediately from a) and Proposition \ref{p3.1}.
\end{proof}

\begin{lemma}\label{linj} Let $\cF\in \Shv_1$, $K$ a separably closed extension of $k$ and $M/K$
an algebraic extension. Then the map $\cF(K)\to \cF(M)$ is injective.
\end{lemma}

\begin{proof} Consider a normalised representation of $\cF$:
\begin{equation}\label{eqnorm}
0\to L\to \uG\by{b} \cF\to E\to 0.
\end{equation}

The lemma then follows from an elementary chase in the diagram:
\[\begin{CD}
0@>>> L(K)@>>> G(K)@>>> \cF(K)@>>> E(K)@>>> 0\\
&& @V{\wr}VV @V{\text{mono}}VV @VVV @V{\wr}VV\\
0@>>> L(M)@>>> G(M)@>>> \cF(M)@>>> E(M)@>>> 0.
\end{CD}\]
\end{proof}

\begin{defn}\label{dsabt} We denote by ${}^t\AbS(k)={}^t\AbS$ 
\index{${}^t\AbS$, $\AbS$} the $\Z[1/p]$-linear category of commutative
$k$-group schemes $G$ such that $G^0$ is semi-abelian and $\pi_0(G)$ is discrete.  An object of ${}^t\AbS$ is called a \emph{semi-abelian scheme with torsion}.
\end{defn}

\begin{propose}\label{pladj} The functor
\begin{align*}
{}^t\AbS&\to \Shv_1\\
G&\mapsto \uG
\end{align*}
has a left adjoint/left inverse $\Omega$.
\end{propose}

\begin{proof} Let $\cF\in \Shv_1$ with normalised representation \eqref{eqnorm}. As the set of
closed subgroups of $H\subseteq G$ is Artinian, there is a minimal $H$ such that the composition
\[L\to \uG\to \uG/\underline{H}\]
is trivial. Then $\cF/b(\underline{H})$ represents an object $\Omega(\cF)$ of ${}^t\AbS$ and
it follows from Proposition \ref{p3.1} b) that the universal property is satisfied.
(In other words,  $\Omega(\cF)$ is the quotient of $\cF$ by the Zariski closure of $L$ in $G$.)
\end{proof}

\begin{propose}\label{p3.6} Let $f:\cF_1\to \cF_2$ be a morphism in $\Shv_1$.
Assume that for any $n>1$ prime to $p$, $f$ is an isomorphism on $n$-torsion and
injective on
$n$-cotorsion. Then $f$ is injective with lattice cokernel. If $f$ is even
bijective on $n$-cotorsion, it is an isomorphism.
\end{propose}

\begin{proof} a) We first treat the special case where $\cF_1=0$. Consider multiplication by $n$ on the normalised presentation of $\cF_2$:
\[\begin{CD}
0@>>> L@>>> G@>>> \cF_2@>>> E@>>> 0\\
&& @V{n_L}VV @V{n_G}VV @V{n}VV @V{n_E}VV\\
0@>>> L@>>> G@>>> \cF_2@>>> E@>>> 0.
\end{CD}\]

Since $L$ is torsion-free, $n_G$ is injective for all $n$, hence $G=0$ and $\cF_2=E$. If moreover multiplication by $n$ is surjective for any $n$, we have $\cF_2=0$ since $E$ is finitely generated. 

b) The general case. Split $f$ into two short exact sequences:
\begin{gather*}
0\to K\to \cF_1\to I\to 0\\
0\to I\to \cF_2\to C\to 0.
\end{gather*}

We get torsion/cotorsion exact sequences
\begin{gather*}
0\to {}_nK\to {}_n\cF_1\to {}_nI\to K/n\to \cF_1/n\to I/n\to 0\\
0\to {}_nI\to {}_n\cF_2\to {}_nC\to I/n\to \cF_2/n\to C/n\to 0.
\end{gather*}

A standard diagram chase successively yields ${}_n K=0$, ${}_n\cF_1\iso {}_nI\iso {}_n\cF_2$, $\cF_1/n\iso I/n$, $K/n=0$ and ${}_n C=0$. By a), we find $K=0$ and $C$ a lattice, which is what we wanted.
\end{proof}

\subsection{Presenting $1$-motivic sheaves by group schemes}\label{3.7} In this subsection, we
give another description of the category $\Shv_1$; it will be used in the next subsection.

\begin{defn}\label{s1a} Let $\AbS$\index{${}^t\AbS$, $\AbS$} be the $\Z[1/p]$-linear category of commutative $k$-group schemes $G$ such that $\pi_0(G)$ is a lattice and $G^0$ is a semi-abelian variety (it is a full subcategory of ${}^t\AbS$). We denote by $S_1^\eff$ \index{$S_1$, $S_1^\eff$} the full subcategory of ${}^t\AbS^{[-1,0]}$ consisting of those complexes $F_\cdot=[F_1\to F_0]$ such that
\begin{thlist}
\item $F_1$ is discrete (\ie in ${}^t\cM_0$);
\item $F_0$ is of the form $L_0\oplus G$, with $L_0\in {}^t\cM_0$ and $G\in \SAb$;
\item $F_1\to F_0$ is a monomorphism;
\item $\ker(F_1\to L_0$) is free. 
\end{thlist}
We call $S_1^\eff$ the \emph{category of presentations}.
\end{defn}

We shall view $S_1^\eff$ as a full subcategory of $\Shv_1^{[-1,0]}$ via the functor $G\mapsto
\underline{G}$ which sends a group scheme to the associated representable sheaf. In this light,
$F_\cdot$ may be viewed as a
\emph{presentation} of
$\cF:=H_0(\underline{F_\cdot})$. In the next definition, quasi-isomorphisms are also understood
from this viewpoint.

\begin{defn}\label{s1b} We denote by $\Sigma$ the collection of quasi-iso\-morph\-isms
of $S_1^\eff$, by ${\bar S}_1^\eff$ the homotopy
category of $S_1^\eff$ (Hom groups quotiented by homotopies) and by
$S_1=\Sigma^{-1}{\bar S}_1^\eff$ the localisation of ${\bar S}_1^\eff$ with respect to (the
image of) $\Sigma$.
\end{defn}

The functor $F_\cdot\mapsto H_0(F_\cdot)$ induces a functor 
\begin{equation}\label{ftot}
h_0:S_1\to \Shv_1.
\end{equation}

Let $F_\cdot=(F_1,L_0,G)$ be a presentation of $\cF\in \Shv_1$. Let $L=\ker(F_1\to
L_0)$ and $E=\coker(F_1\to L_0)$. Then we clearly have an exact sequence
\begin{equation}\label{eq2.3}
0\to L\to \uG\to\cF\to E\to 0.
\end{equation}

\begin{lemma}\label{l2.3.1} Let $F_\cdot=(F_1,L_0,G)\in S_1^\eff$. Then, for any finite
Galois extension
$\ell/k$ such that $L_0$ is constant over $\ell$, there exists a \qi $\tilde F_\cdot\to
F_\cdot$, with $\tilde F_\cdot=[\tilde F_1\by{u_0} \tilde L_0\oplus G]$ such that $u_0$ is
diagonal and $\tilde L_0$ is a free $Gal(\ell/k)$-module.
\end{lemma}

\begin{proof} Just take for $\tilde L_0$ a free module projecting onto $L_0$ and for
$\tilde F_1\to \tilde L_0$ the pull-back of $F_1\to L_0$.
\end{proof}

\begin{lemma}\label{l2.3.2} The set $\Sigma$ admits a calculus of right
fractions within $\bar S_1^\eff$ in the sense of (the dual of) \cite[Ch. I, \S 2.3]{GZ}.
\end{lemma}

\begin{proof} The statement is true by Lemma \ref{lloc} if we replace $S_1^\eff$ by
$\Shv_1^{[-1,0]}$; but one easily checks that the constructions in the proof of Lemma
\ref{lloc} preserve $S_1^\eff$.
\end{proof}

\begin{propose}\label{p3.4.6} The functor $h_0$ of \eqref{ftot} is an equivalence of categories. In
particular, $S_1$ is abelian.
\end{propose}

\begin{proof} {\it Step 1.} $h_0$ is essentially surjective. Let $\cF\in \Shv_1$ and let
\eqref{eq2.3} be the exact sequence attached to it by Proposition \ref{p3.1} b). We shall
construct a presentation of $\cF_\cdot$ from \eqref{eq2.3}. Choose elements $f_1,\dots,f_r\in
\cF(\bar k)$ whose images generate
$E(\bar k)$. Let
$\ell/k$ be a finite Galois extension such that all $f_i$ belong to $\cF(\ell)$, and let
$\Gamma=Gal(\ell/k)$. Let $\tilde L_0=\Z[\Gamma]^r$ and define a morphism of sheaves $\tilde
L_0\to \cF$ by mapping the $i$-th basis element to $f_i$. Then $\ker(\tilde L_0\to E)$ maps to
$\uG/L$. Let $M_0$ be the kernel of this morphism, and let $L_0= \tilde L_0/M_0$. Then $\tilde
L_0\onto E$ factors into a morphism $L_0\onto E$, whose kernel $K$ injects into $\uG/L$.

Pick now elements $g_1,\dots,g_s\in G(\bar k)$ whose image in $G(\bar k)/L(\bar k)$ generate
the image of $K(\bar k)$, and $g_{s+1},\dots,g_t\in G(\bar k)$ be generators of the image of
$L(\bar k)$. Let $\ell'/k$ be a finite Galois extension such that all the $g_i$ belong to
$G(\ell')$, and let $\Gamma'=Gal(\ell'/k)$. Let $\tilde F_1=\Z[\Gamma']^t$, and define a
map
$f:\tilde F_1 \to G$ by mapping the $i$-th basis element to $g_i$.  By construction,
$f^{-1}(L)=\ker(\tilde F_1\onto K)$ and $f':f^{-1}(L)\to L$ is onto. Let $M_1$
be the kernel of $f'$ and $F_1=\tilde F_1/M_1$: then $\tilde F_1\to K$ factors through
$F_1$ and
$\ker(F_1\onto K)=\ker(F_1\to L_0)\iso L$. In particular, condition (iii) of Definition
\ref{s1a} is verified.

{\it Step 2.} $h_0$ is faithful. Let $f:F_\cdot\to F'_\cdot$ be a map in $S_1$ such that
$h_0(f)=0$. By Lemma \ref{l2.3.2}, we may assume that $f$ is an effective map (\ie comes from
$S_1^\eff$). We have
$f(L_0\oplus G)\subseteq
\im(L'_1\to  L'_0\oplus G')$, hence $f_{|G} = 0$ and 
$f(L_0)$ is contained in $\im(L'_1\to L'_0\oplus G')$. Pick a finite Galois extension $\ell/k$
such that
$L_0$ and $L'_1$ are constant over $\ell$. By Lemma \ref{l2.3.1}, take a \qi $u:
[\tilde F_1\to \tilde L_0]\to [F_1\to L_0]$ such that $\tilde L_0$ is $Gal(\ell/k)$-free.
Then the composition
$\tilde L_0\to L_0\to \im(L'_1\to L'_0\oplus G')$ lifts to a map $s:\tilde L_0\to L'_1$,
which defines a homotopy between $0$ and $fu$.

{\it Step 3.} $h_0$ is full. Let $F_\cdot,F'_\cdot\in S_1$ and let $\phi:\cF\to \cF'$, where
$\cF=h_0(F_\cdot)$ and $\cF'=h_0(F'_\cdot)$. In particular, we get a map $\phi_G:\uG\to
\uG'$ and a map $\psi:L_0\to L'_0\oplus \uG'/F'_1$. Let
$\ell/k$ be a finite Galois extension such that $F'_1$ is constant over $\ell$. Pick a \qi
$u:\tilde F_\cdot\to F_\cdot$ as in Lemma
\ref{l2.3.1} such that $\tilde L_0$ is
$Gal(\ell/k)$-free. Then $\psi\circ u$ lifts to a map $\tilde\psi:\tilde L_0\to L'_0\oplus
\uG'$. The map
\[f=(\tilde\psi,\phi_G):\tilde L_0\oplus G\to L'_0\oplus G'\]
sends $\tilde F_1$ into $F'_1$ by construction, hence yields a map $f:\tilde F_\cdot\to
F'_\cdot$ such that $h_0(f u^{-1})=\phi$.
\end{proof}

\begin{cor} The obvious functor
\[S_1\to D^b(\Shv_1)\]
is fully faithful.
\end{cor}

\begin{proof} The composition of this functor with $H_0$ is the equivalence $h_0$ of
Proposition \ref{p3.4.6}. Therefore it suffices to show that the restriction of $H_0$ to the
image of $S_1$ is faithful. This is obvious, since the objects of this image are homologically
concentrated in degree $0$.
\end{proof}

\subsection{The transfer structure on $1$-motivic sheaves} Recall the category $\AbS$ from
Definition \ref{s1a}. Lemma 
\ref{l1.3} provides a functor
\[\rho:\AbS\to \HI_\et[1/p].\]

The aim of this subsection is to prove:

\begin{propose}\label{ptransf} This functor extends to a full embedding
\[\rho:\Shv_1\into\HI_\et[1/p]\]
where $\HI_\et$ is the category of Definitions \ref{hi} and
\ref{dD.1}. This functor is exact with thick image (\ie stable under extensions).
\end{propose}

\begin{proof} By Proposition \ref{p3.4.6}, it suffices to construct a functor $\rho:S_1\to
\HI_\et[1/p]$. First define a functor $\tilde \rho:S_1^\eff\to \HI_\et[1/p]$ by
\[\tilde\rho([F_1\to F_0])= \coker(\rho(F_1)\to \rho(F_0)).\]

Note that the forgetful functor $f:\HI_\et\to \Shv_\et(Sm(k))$ is faithful and exact, hence
conservative. This first gives that
$\tilde\rho$ factors into the desired $\rho$.

Proposition \ref{p3.1} d) says that $f\rho$ is (fully faithful and) exact.
Since
$f$ is faithful, $\rho$ is fully faithful and exact. 

It remains to show that $\rho$ is thick. Recall that $\Shv_1$ is thick in $\Shv_\et(Sm(k))[1/p]$ by
Theorem \ref{text}. Since $f$ is exact, we are then left to show:

\begin{lemma} The transfer structure on a sheaf $\cF\in \Shv_1$ is unique.
\end{lemma}

\begin{proof} Let $\mu$ be the transfer structure on $\cF$ given by the beginning of the proof
of Proposition \ref{ptransf}, and let $\mu'$ be another transfer structure. Thus, for $X,Y\in
Sm(k)$, we have two homomorphisms
\[\mu,\mu':\cF(X)\otimes c(Y,X)\to \cF(Y)\]
and we want to show that they are equal. We may clearly assume that $Y$ is irreducible.

Let $F=k(Y)$ be the function field of $Y$. Since $\cF$ is a homotopy invariant Zariski sheaf
with transfers, the map $\cF(Y)\to \cF(F)$ is injective by \cite[Cor. 4.19]{V2}. Thus we may
replace $Y$ by $F$.

Moreover, it follows from the fact that $\cF$ is an \'etale sheaf and from Lemma \ref{linj}
that $\cF(F)\into \cF(\bar F)$, where $\bar F$ is an algebraic closure of $F$. Thus, we may even
replace $Y$ by $\bar F$.

Then the group $c(Y,X)$ is replaced by $c(\bar F,X) = Z_0(X_{\bar F})$. Since $\bar F$ is
algebraically closed, all closed points of $X_{\bar F}$ are rational, hence all finite
correspondences from $\Spec \bar F$ to $X$ are linear combinations of morphisms. Therefore $\mu$
and
$\mu'$ coincide on them.
\end{proof}

This concludes the proof of Proposition \ref{ptransf}.
\end{proof}

\subsection{$1$-motivic sheaves and $\DM$} Recall from Definition \ref{dD.1} the subcategory
$\HI_\et^s\subset \HI_\et$ \index{$\HI_\et^s$} of strictly homotopy invariant \'etale sheaves with transfers:
this is a full subcategory of $\DM_{-,\et}^\eff$. By Proposition \ref{pD.1.4}, we have
\[\HI_\et^s= \{\cF\otimes_\Z\Z[1/p]\mid \cF\in\HI_\et\}.\] 

The introduction of $\Shv_1$ is now made clear by the following 

\begin{sloppypar}
\begin{thm} \label{t3.2.3} Let $\Shv_1^s\subseteq \HI_\et^s$ be the full
subcategory image of $\Shv_1$ by the functor $\cF\mapsto \cF[1/p]$ of Lemma
\ref{lD.1.3}. Then $\Shv_1\to \Shv_1^s$ is an equivalence of categories. Moreover,
let
$M\in d\1 \DM_{\gm,\et}^{\eff}$. Then for all $i\in\Z$, $\sH_i(M)\in \Shv_1^s$. In particular, 
there is a $t$-structure on $d\1 \DM_{\gm,\et}^{\eff}$, with heart $\Shv_1^s$; it is induced by
the homotopy $t$-structure of Corollary \ref{cD.2} on
$\DM_{\fr,\et}^{\eff}$ (see Definition \ref{dfr} and Theorem \ref{tfr}).
\end{thm}\index{$\Shv_1^s$}
\end{sloppypar}

\begin{proof} The first assertion is clear since $\HI_\et^s\to \HI_\et\to
\HI_\et[1/p]$ is an equivalence of categories by Proposition \ref{pD.1.4}. By
Proposition
\ref{p3.1} d), we reduce to the case
$M=M(C)$, $C\by{p}\Spec k$ a smooth projective curve. By Proposition \ref{lcurve}, the
cohomology sheaves of $M(C)$ belong to $\Shv_1^s$: for $\cH^1$ this is clear and for $\cH^2$ it
is a (trivial) special case of Proposition \ref{p3.3.1}.
\end{proof}

Note that the functor $\M\to \HI_\et^{[0,1]}$ refines to an functor $\M\to \Shv_1^{[0,1]}$,
hence, using Lemma \ref{ltot} again, we get a composed triangulated functor
\begin{equation}\label{hts}
\tot:D^b(\M[1/p])\to D^b(\Shv_1^{[0,1]})\to D^b(\Shv_1)
\end{equation}\index{$\tot$}
refining the one from Lemma \ref{l2.2.1} (same proof). We then get:

\begin{cor} \label{c3.3.2} The two functors
\[\begin{CD}
D^b(\M[1/p])@>\tot>>D^b(\Shv_1)\to d\1\DM_{\gm,\et}^\eff
\end{CD}\]
are equivalences of categories. \end{cor}

\begin{proof} For the composition, this is Theorem \ref{t1.2.1}. This implies that the second functor is full and essentially surjective, and to conclude, it suffices by Lemma \ref{lA.2} to see that it is conservative. But this follows immediately from Theorem \ref{t3.2.3}. Note that by Proposition~\ref{pD.1.4} we have that $\Shv_1^s\simeq \Shv_1$ (see Theorem \ref{t3.2.3} for the definition of $\Shv_1^s$).  
\end{proof}

\begin{defn}\label{dhts} We call the $t$-structure defined on $D^b(\M[1/p])$ or on
$d\1\DM_{\gm,\et}^\eff$ by Corollary \ref{c3.3.2} the \emph{homotopy $t$-structure}.
\end{defn}

\begin{remark}[\cf \S \protect{\ref{alessandra1}}]\label{alessandra2}  In
\cite{alessandra}, A. Bertapelle defines an \emph{integral} version of
the category $\Shv_1^{\rm fppf}$ of Remark \ref{fppf} and constructs an
equivalence of categories 
\[D^b(\M)\simeq D^b(\Shv_1^{\rm fppf})\]
without inverting $p$ (not going via $\DM$). Hence the homotopy $t$-structure of Definition \ref{dhts} exists integrally.
\end{remark}

\subsection{Comparing $t$-structures}\label{3.6} In this subsection, we want
to compare the homotopy $t$-structure of Definition \ref{dhts} with the motivic $t$-structure of Theorem \ref{ptors} a).

Let $C\in D^b(\M[1/p])$. Recall from \ref{not} the notation ${}^tH_n(C)\in {}^t\M[1/p]$ for its homology relative to the torsion $1$-motivic
$t$-structure from Theorem \ref{ptors}. We also write $\sH^n(C)\allowbreak\in
\Shv_1$ for its cohomology objects of  relative to the homotopy
$t$-structure. \index{$\sH^n$, $\sH_n$}

Consider the functor $\tot$ of \eqref{hts}. Let $\cF$ be a $1$-motivic sheaf and
$(G,b)$ its associated normalised pair (see Proposition \ref{p3.1} a)). Let $L=\ker
b$ and $E=\coker b$. In $D^b(\M[1/p])$, we have an exact triangle
\[[L\to G][1]\to \tot^{-1}(\cF)\to [E\to 0]\by{+1}\]
(see Corollary \ref{c3.3.2}). This shows:

\begin{lemma} \label{l4.3.1} We have
\begin{align*}
{}^tH_0(\tot^{-1}(\cF))&=[E\to 0]\\
{}^tH_1(\tot^{-1}(\cF))&=[L\to G]\\
{}^tH_q(\tot^{-1}(\cF))&=0 \text{ for } q\ne 0,1.\qed
\end{align*}
\end{lemma}

On the other hand, given a 1-motive (with torsion or cotorsion) $M=[L\by{f} G]$, we clearly have
\begin{align}
\sH^0(M)&=\ker f\notag\\
\sH^{1}(M)&=\coker f\label{eq3.9}\\
\sH^q(M)&=0 \text{ for } q\ne 0,1.\notag
\end{align}
by considering it as a complex of length $1$ of $1$-motivic sheaves.

In particular, ${}^t\M[1/p]\cap \Shv_1= \Shv_0$,
${}^t\M[1/p]\cap \Shv_1[-1]$ consists of quotients of semi-abelian
varieties by discrete subsheaves and ${}^t\M[1/p]\cap \Shv_1[q]\allowbreak=0$
for $q\ne 0,-1$.

Here is a more useful result relating $\sH^i$ with the two motivic $t$-structures:

\begin{propose}\label{p3.10} Let $C\in D^b(\M[1/p])$; write $[L_i\by{u_i} G_i]$ for ${}_tH_i(C)$ and $[L^i\by{u^i} G^i]$ for ${}^tH^i(C)$\footnote{Note that $(L_i,G_i)$ and $(L^i,G^i)$ are determined only up to the relevant \qi's.}. Then we have exact sequences in $\Shv_1$:
\begin{gather*}
\dots\to L_{i+1}\by{u_{i+1}} G_{i+1}\to \sH_i(C)\to L_{i}\by{u_{i}} G_{i}\to\dots\\
\dots\to L^{i-1}\by{u^{i-1}} G^{i-1}\to \sH^i(C)\to L^{i}\by{u^{i}} G^{i}\to\dots
\end{gather*}
\end{propose}

\begin{proof} For the first one, argue by induction on the length of $C$ with respect to the motivic $t$-structure with heart ${}_t\M[1/p]$ (the case of length $0$ is \eqref{eq3.9}). For the second one, same argument with the other motivic $t$-structure.
\end{proof}

Note finally that the homotopy $t$-structure is far from being invariant under Cartier duality:
this can easily be seen by using Proposition \ref{p3.4.6}.

\section{Comparing two dualities}\label{dual}

In this section, we show that the classical Cartier duality for $1$-motives is compatible with a
``motivic Cartier duality" on triangulated motives, described in Definition \ref{cdef} below. 

\subsection{Biextensions of $1$-motives}\label{sbiext}This material is presumably well-known
to  experts, and the only reason why we write it up is that we could not find
it in the literature. Exceptionally, we put $1$-motives in degrees $-1$ and $0$ in this subsection and in the next one, for compatibility with Deligne's conventions in \cite{D}.

Recall (see \cite[\S 10.2]{D}) that for $M_1= [L_1\by{u_1} G_1]$ and 
$M_2= [L_2\by{u_2} G_2]$ two complexes of abelian sheaves over some site $\cS$, concentrated in
degrees
$-1$ and
$0$, a \emph{biextension} of
$M_1$ and
$M_2$ by an abelian sheaf $H$ is given by a (Grothendieck) biextension $P$
of $G_1$ and $G_2$ by $H$ and a pair of compatible trivializations of the
biextensions of
$L_1\times G_2$ and $G_1\times L_2$ obtained by pullbacks.
Let $\Biext (M_1,M_2;H)$ denote the group of isomorphism  classes
of biextensions. We have the following fundamental formula (see \cite[\S
10.2.1]{D}):\index{$\Biext$}
\begin{equation}\label{biextform}
\Biext (M_1,M_2;H) = \EExt^1_\cS(M_1\oo^L M_2,H).
\end{equation}

\begin{sloppypar}
Suppose now that $M_1$ and $M_2$ are two Deligne $1$-motives. Since $G_1$ and $G_2$ are smooth,
we may compute biextensions by using the \'etale topology. Hence, we shall take here
\[\cS=Sm(k)_\et.\]

Let $M_2^*$ denote the Cartier dual of $M_2$ as constructed by Deligne
(see  \cite[\S 10.2.11]{D} and \cite[\S 0]{BSAP}) along with the
Poincar\'e biextension $P_{M_2}\in \Biext (M_2,M_2^*;\G_m)$. We also
have  the transpose ${}^tP_{M_2}=P_{M_2^*}\in \Biext (M_2^*,M_2;\G_m)$.
Pulling back ${}^tP_{M_2}$ yields a map
\begin{align}
\gamma_{M_1,M_2} : \Hom (M_1, M_2^*)& \to \Biext (M_1,M_2;\G_m)\label{eqbiext}\\
\phi&\mapsto(\phi\times 1_{M_1})^*({}^tP_{M_2})\notag
\end{align}
which is clearly additive and natural in $M_1$.
\end{sloppypar}

\begin{propose} \label{biext} The map $\gamma_{M_1,M_2}$ yields a
natural equivalence of functors from $1$-motives to abelian groups, \ie
the functor 
$$M_1\mapsto \Biext (M_1,M_2;\G_m)$$
on $1$-motives is representable by the Cartier dual $M_2^*$. Moreover,
$\gamma_{M_1,M_2}$ is also natural in $M_2$. 
\end{propose}

\begin{proof}  We start with a few lemmas:

\begin{lemma}\label{ld0} For $q\le 0$, we have
\[\Hom_{\bar k}(M_1\oo^L M_2,\G_m[q])=0.\]
\end{lemma}

\begin{proof} For $q<0$ this is trivial and for $q=0$ this is \cite[Lemma 10.2.2.1]{D}.
\end{proof}

\begin{lemma} \label{kalg} Let $\bar k$ be an algebraic closure of $k$
and $G=Gal(\bar k/k)$. Then
\begin{align*}
\Hom_k (M_1, M_2^*)&\iso \Hom_{\bar k} (M_1, M_2^*)^G\\
\Biext_k (M_1,M_2;\G_m)&\iso \Biext_{\bar k} (M_1,M_2;\G_m)^G.
\end{align*}
\end{lemma}

\begin{sloppypar}
\begin{proof} The first isomorphism is obvious. For the second, thanks
to \eqref{biextform} we may use the spectral sequence
\[H^p(G,\Hom_{\bar k}(M_1\oo^L M_2,\G_m[q]))\Rightarrow \Hom_k(M_1\oo^L
M_2,\G_m[p+q]).\]

(This is the only place in the proof of Proposition \ref{biext} where we
shall use \eqref{biextform}.) The assertion then follows from Lemma \ref{ld0}. \end{proof}
\end{sloppypar}

Lemma \ref{kalg}, reduces the proof of Proposition \ref{biext} to the
case where \emph{$k$ is algebraically closed}, which we now assume. The following is a special
case of this proposition:

\begin{lemma} \label{biexta} The map $\gamma_{M_1,M_2}$ is an
isomorphism when $M_1$ and $M_2$ are abelian varieties $A_1$ and $A_2$,
and is natural in $A_2$.
\end{lemma}

Again this is certainly well-known and mentioned explicitly as such in
\cite[VII, p. 176, (2.9.6.2)]{sga7}. Unfortunately we have not been able to
find a proof in the literature, so we provide one for the reader's
convenience.

\begin{proof}  We shall use the universal property of the Poincar\'e
bundle \cite[Th. p. 125]{mumford}. Let $P\in \Biext(A_1,A_2)$. Then 
\begin{enumerate}
\item $P_{|A_1\times\{0\}}$ is trivial;
\item $P_{|\{a\}\times A_2}\in \Pic^0(A_2)$ for all $a\in A_1(k)$.
\end{enumerate}

Indeed, (1) follows from the multiplicativity of $P$ on the $A_2$-side.
For (2) we offer two proofs (note that they use multiplicativity on
different sides): 

\begin{sloppypar}
\begin{itemize}
\item By multiplicativity on the $A_1$-side, $a\mapsto P_{|\{a\}\times
A_2}$ gives a homomorphism $A_1(k)\to \Pic(A_2)$. Composing with the
projection to $\NS(A_2)$ gives a homomorphism from a divisible group to a
finitely generated group, which must be trivial.
\item (More direct but more confusing): we have to prove that
$T_b^*P_{|\{a\}\times A_2}=P_{|\{a\}\times A_2}$ for all $b\in A_2(k)$.
Using simply $a$ to denote the section $\Spec k\to A_1$ defined by $a$,
we have a commutative diagram
\[\begin{CD}
A_2@>a\times 1_{A_2}>> A_1\times A_2\\
@V{T_b}VV @VV{1_{A_1}\times T_b}V \\
A_2@>a\times 1_{A_2}>> A_1\times A_2.
\end{CD}\]

Let $\pi_1:A_1\to\Spec k$ and $\pi_2:A_2\to\Spec k$ be the two
structural maps. Then by multiplicativity on the $A_2$-side, an easy
computation gives
\[(1_{A_1}\times T_b)^*P=P\otimes \left(1_{A_1}\times (\pi_2\circ
b)\right)^*P.\]

Applying $(a\times 1_{A_2})^*$ to this gives the result since $ (a\times
1_{A_2})^*\circ\left(1_{A_1}\times (\pi_2\circ b)\right)^*P=\pi_{A_2}^*
P_{a,b}$ is trivial.
\end{itemize}
\end{sloppypar}

By the universal property of the Poincar\'e bundle, there exists a
unique morphism\footnote{For convenience we denote here by $A'$ the dual of an
abelian variety $A$ and by $f'$ the dual of a homomorphism $f$ of abelian
varieties.} $f:A_1\to A_2'$ such that $P\simeq (f\times
1_{A_2})^*({}^tP_{A_2})$. It remains to see that $f$ is a homomorphism:
for this it suffices to show that $f(0)=0$. But
\begin{multline*}
\sO_{A_2}\simeq P_{|\{0\}\times A_2}=(0\times 1_{A_2})^*\circ (f\times
1_{A_2})^*({}^t P_{A_2})\\ =(f(0)\times 1_{A_2})^*({}^t
P_{A_2})=(P_{A_2})_{|A_2\times\{f(0)\}}=f(0)
\end{multline*}
where the first isomorphism holds by multiplicativity of $P$ on the
$A_1$-side.

Finally, the naturality in $A_2$ reduces to the fact that, if $f:A_1\to A_2'$, then $(f\times
1_{A_2})^*({}^tP_{A_2})\simeq (1_{A_1}\times f')^* (P_{A_1})$. This follows from the
description of
$f'$ on $k$-points as the pull-back by $f$ of line bundles.
\end{proof}

We also have the following easier

 \begin{lemma}\label{kalg1} Let $L$ be a lattice and $A$ an abelian
variety. Then the natural map
\begin{align*}
\Hom(L,A')&\to \Biext(L[0],A[0];\G_m)\\
f&\mapsto (1\times f)^*({}^tP_A)
\end{align*}
is bjiective.
\end{lemma}

\begin{proof} Reduce to $L=\Z$; then the right hand side can be identified with $\Ext(A,\G_m)$
and the claim comes from the Weil-Barsotti formula.
\end{proof}

Let us now come back to our two $1$-motives $M_1,M_2$. We denote by $L_i,
T_i$ and
$A_i$ the discrete, toric and abelian parts of $M_i$ for $i= 1,2$. Let us
further denote by 
$u'_i:L_i'\to A_i'$ the map corresponding to $G_i$ under the isomorphism
$\Ext (A_i,T_i) \simeq\Hom (L_i',A_i')$ where $L_i' =\Hom (T_i,\G_m)$ and
$A_i' =\Pic^0 (A_i)$. 

We shall use the symmetric avatar  $(L_i,A_i,L_i',A_i',\psi_i)$ of $M_i$
(see \cite[10.2.12]{D} or \cite[p. 17]{BSAP}): recall that
$\psi_i$ denotes a certain section of the Poincar\'e biextension
$P_{A_i}\in \Biext (A_i,A_i';\G_m)$ over $L_i\times L_i'$. The symmetric
avatar of the Cartier dual is 
$(L_i',A_i',L_i,A_i,\psi^t_i)$. By \loccit a map of 1-motives 
$\phi: M_1 \to M_2^*$ is equivalent to a homomorphism $f: A_1\to
A_2'$ of abelian varieties and, if $f'$ is the dual of $f$, liftings $g$
and $g'$ of $fu_1$ and $f'u_2$ respectively, \ie to the following
commutative squares
\begin{equation}\label{squares}
\begin{CD}
L_1@>g>> L_2'\\
@V{u_1}VV  @V{u'_2}VV  \\
 A_1 @>f>>  A_2'
\end{CD}\qquad \text{and} \qquad
\begin{CD}
 L_2@>g'>> L_1'\\
@V{u_2}VV  @V{u_1'}VV  \\
 A_2 @>f'>>  A_1'\\[4pt]
\end{CD}
\end{equation}
under the condition that 
\begin{equation}\label{cond}
(1_{L_1}\times g')^*\psi_1= (g\times 1_{L_2})^*{}^t\psi_2
\text{ on } L_1\times L_2.
\end{equation}

Now let $(P,\tau, \sigma)$ be a biextension of $M_1$ and
$M_2$ by $\G_m$, \ie  a biextension $P\in\Biext (G_1,G_2;\G_m)$, a
section $\tau$ on $L_1\times G_2$ and a section $\sigma$ on $G_1\times
L_2$ such that 
\begin{equation}\label{eq7}
\tau\mid_{L_1\times L_2}= \sigma\mid_{L_1\times L_2}.
\end{equation} 

We have to show that $(P,\tau,\sigma)=(\phi\times 1)^*({}^tP_{A_2},\tau_2,\sigma_2)$
for a unique $\phi:M_1\to M_2^*$, where $\tau_2$ and $\sigma_2$ are the
universal trivializations.

Recall that $\Biext (G_1,G_2;\G_m)= \Biext (A_1,A_2;\G_m)$ (\cf
\cite[10.2.3.9]{D}) so that,  by Lemma \ref{biexta}, $P$ is the pull-back
to $G_1\times G_2$ of $(f\times 1_{A_2})^*({}^tP_{A_2})$ for a unique
homomorphism $f:A_1\to A_2'$. We thus have obtained the map $f$ and its
dual $f'$ in \eqref{squares}, and we now want to show that the extra data
$(\tau,\sigma)$ come from a pair $(g,g')$ in a unique way.

We may view $E=(fu_1\times 1_{A_2})^*({}^tP_{A_2})$ as an extension of
$L_1\otimes A_2$ by $\G_m$. Consider the commutative diagram of exact
sequences
\begin{equation}\label{diag}
\begin{CD}
&&0&& 0&& 0\\
&&@VVV @VVV @VVV \\
0@>>> 0@>>> L_1\otimes T_2@= L_1\otimes T_2@>>> 0\\
&&@VVV @VVV @V{1_{L_1}\otimes i_2}VV\\
0@>>> \G_m@>i>> Q@>\pi'>> L_1\otimes G_2@>>> 0\\
&&||&&@V{\pi}VV @V{1_{L_1}\otimes p_2}VV\\
0@>>> \G_m@>>> E@>>> L_1\otimes A_2@>>> 0\\
&&@VVV @VVV @VVV \\
&&0&& 0&& 0
\end{CD}
\end{equation}
where $i_2$  (\resp $p_2$) is the inclusion $T_2\into G_2$ (\resp the
projection $G_2\to A_2$). The section $\tau$ yields a retraction
$\tilde\tau:Q\to
\G_m$ whose restriction to $L_1\otimes T_2$ yields a homomorphism
\[\tilde g:L_1\otimes T_2\to \G_m\]
which in turn defines a homomorphism as in \eqref{squares}. We denote
the negative of this morphism by $g$.

\begin{lemma}\label{standard} With this choice of $g$, the left square of
\eqref{squares} commutes and $\tau=(g\times 1_{G_2})^*\tau_2$.
\end{lemma}

\begin{proof} To see the first assertion, we may apply $\Ext^*(-,\G_m)$
to \eqref{diag} and then apply \cite[Lemma 2.8]{bs} to the corresponding
diagram. Here is a concrete description of this argument: via the map of
Lemma \ref{kalg1}, $u'_2g$ goes to the following pushout
\[\begin{CD}
0@>>> L_1\otimes T_2@>{1\otimes i_2}>> L_1\otimes G_2@>{1\otimes p_2}>>
L_1\otimes A_2@>>> 0\\ 
&&@V{-\tilde g}VV @V{\pi\circ\tau}VV@V{||}VV\\
0@>>>\G_m@>>> E @>>> L_1\otimes A_2@>>> 0. 
\end{CD}\] 
because, due to the relation $i\tilde\tau + \tau \pi'=1$,  the left square
in this diagram commutes.

In particular, we have
\begin{multline*}
Q=(1\otimes p_2)^*(fu_1\otimes
1)^*{}^tP_{A_2}=(fu_1\otimes p_2)^*{}^tP_{A_2}\\
=(u'_2g\otimes
p_2)^*{}^tP_{A_2}=(g\otimes 1)^*(u'_2\otimes p_2^*){}^tP_{A_2}.
\end{multline*}

For the second assertion, since $\Hom(L_1\otimes A_2,\G_m)=0$ it suffices
to check the equality after restricting to $L_1\otimes T_2$. This is clear
because under the isomorphism $\Hom (L'_2\otimes T_2,\G_m)= \Hom
(L'_2,L'_2)$, the canonical trivialization ${}^t\psi_2$
corresponds to the identity.
\end{proof}

Note that if we further pullback we obtain that 
\begin{equation}\label{eq4}
\tau\mid_{L_1\times L_2}=\psi_2^t\mid_{L_1\times L_2}.
\end{equation} 

The same computation with $\sigma$ yields a map
\[g':L_2\to L'_1\]
and the same argument as in Lemma \ref{standard} shows that with this
choice of $g'$ the right square of \eqref{squares} commutes. We now use
that $P=(1_{A_1}\times f')^*(P_{A_1})$, which follows from the
naturality statement in Lemma \ref{biexta}. As in the proof of Lemma
\ref{standard}, this implies that its  trivialization
$\sigma$ on $G_1\times L_2$ is the pullback of the canonical
trivialization $\psi_1$ on $G_1\times L_1'$ along 
$1_{G_1}\times g': G_1\times L_2\to G_1\times L_1'$. In particular: 
\begin{equation}\label{eq5}
\sigma\mid_{L_1\times L_2}=\psi_1\mid_{L_1\times L_2}.
\end{equation}

Put together, \eqref{eq7}, \eqref{eq4} and \eqref{eq5} show that
Condition \eqref{cond} is verified: thus we get a morphism $\phi:M_1\to
M_2^*$. Let $h:G_1\to G'_2$ be its group component. It remains to check
that $\sigma= (h\times 1_{L_2})^*\sigma_2$. As in the proof of Lemma
\ref{standard} we only need to check this after restriction to
$T_1\otimes L_2$. But the restriction of $h$ to the toric parts is
the  Cartier dual of $g'$, so we conclude by the same argument.

Finally, let us show that $\gamma_{M_1,M_2}$ is natural in $M_2$. This
amounts to comparing two biextensions. For the bitorsors this follows from Lemma \ref{biexta} and for the sections we may argue again as in the proof of Lemma \ref{standard}.
\end{proof}

\subsection{Biextensions of complexes of $1$-motives}

Let $\cA$ be a category of abelian sheaves, and consider two bounded complexes $C_1,C_2$ of objects of $\cA^{[-1,0]}$. Let $H\in \cA$. We have a double complex
\[\underline{\Biext}(C_1,C_2;H)^{p,q}\df\Biext(C_1^p,C_2^q;H).\]

\begin{defn} A \emph{biextension of $C_1$ and $C_2$ by $H$} is an element of the group of cycles 
\[\Biext(C_1,C_2;H)\df Z^0(\mathbf{Tot}\underline{\Biext}(C_1,C_2;H)).\]
Here $\mathbf{Tot}$ denotes the total complex associated to a double complex.
\end{defn}
\index{$\mathbf{Tot}$}
\index{$\underline{\Biext}$}
Concretely: such a biextension $P$ is given by a collection of biextensions $P_p\in \Biext(C_1^p,C_2^{-p};H)$ such that, for any $p$,
\[(d_1^p\otimes 1)^* P_{p+1}= (1\otimes d_2^{-p-1})^*P_p
\]
where $d_1^\cdot$ (\resp $d_2^\cdot$) are the differentials of $C_1$ (\resp of $C_2$).

Now suppose that $\cA$ is the category of fppf sheaves, that $H=\G_m$ and that all the $C_i^j$ are Deligne $1$-motives. By Lemma \ref{ld0}, 
we have
\[\EExt^i(C_1^p,C_2^q;\G_m)=0\text{ for } i\le 0.\]

Therefore, a spectral sequence argument yields an edge homomorphism
\begin{equation}\label{edge}
\Biext(C_1,C_2;\G_m)\to \EExt^1(C_1\oo^L C_2,\G_m).
\end{equation}

Recall that Deligne's Cartier duality \cite{D} provides an exact functor
$$M\mapsto M^*: \M[1/p]\to \M[1/p]$$
yielding by Proposition \ref{pcd} a triangulated functor
\begin{equation}\label{dcartier}
(\ \ )^*: D^b(\M[1/p])\to D^b(\M[1/p]).
\end{equation}

Note that for a complex of $1$-motives
\[C=(\cdots \to M^{i}\to M^{i+1}\to \cdots) \]
we can compute $C^*$ by
means of the complex
\[C^*=(\cdots \to (M^{i+1})^*\to (M^{i})^*\to \cdots) \]
of Cartier duals here placed in degrees ..., $-i-1$, $-i$, etc.

Let us now take in \eqref{edge} $C_1=C$, $C_2 = C^*$. For each $p\in\Z$, we have the
Poincar\'e biextension $P_p\in \Biext(C^p,(C^p)^*;\G_m)$. By Proposition \ref{biext}, the
$\{P_p\}$ define a class in $\Biext(C,C^*;\G_m)$.

\begin{defn}\label{dpc} This class $P_C$ is the \emph{Poincar\'e biextension of the
complex $C$}.
\end{defn}

Let $C_1,C_2\in C^b(\M)$. As in Subsection \ref{sbiext}, pulling back ${}^tP_{C_1}=P_{C_1^*}\in \Biext(C_1,C_1^*;\G_m)$ yields a map generalising \eqref{eqbiext}:
\begin{align}
\gamma_{C_1,C_2} : \Hom (C_1, C_2^*)& \to \Biext (C_1,C_2;\G_m)\label{eqbiextc}\\
\phi&\mapsto(\phi\times 1_{C_1})^*({}^tP_{C_2}).\notag
\end{align}
which is clearly additive and natural in $C_1$. We then have the following trivial extension of the functoriality in Proposition \ref{biext}: 

\begin{propose}\label{nat} $\gamma_{C_1,C_2}$ is also natural in $C_2$.\qed
\end{propose}

\subsection{Comparing two Ext groups} The aim of this subsection is to prove:

\begin{propose}\label{ext=ext} Let $C_1,C_2\in D^b(\M[1/p])$. Then the forgetful triangulated
functors
\[\DM_{-,\et}^{\eff}(k)\to D^{-}(\Shv_{\et}(SmCor(k)))\to
D^{-}(\Shv_{\et}(Sm(k)))\]
induce an isomorphism
\[\Hom_{\DM_{-,\et}^{\eff}}(\Tot C_1\otimes \Tot C_2,\Z(1)[q])\iso
\Hom_{Sm(k)_\et}(C_1\oo^L C_2,\G_m[q+1])\]
for any $q\in\Z$.
\end{propose}

\begin{proof}
Each of the two functors in Proposition \ref{ext=ext} has a left adjoint. For the first
one see \cite[Prop. 3.2.3]{V}; we shall denote it by $RC$ as in \loccit
The second may be constructed using \cite[Remark 1 p. 202]{V}: we shall
denote it by
$\Phi$. In both cases, the construction is done for the Nisnevich
topology but carries over for the \'etale topology as well (see also
\cite{VL}). The tensor product ${\oo^L}_{tr}$ in
$D^{-}(\EST)$ is defined from the formula
\[L(X)\otimes L(Y)=L(X\times Y)\]
see \cite[p. 206]{V}. For $X\in Sm(k)$, let $\Z(X)$ be the $\Z$-free \'etale sheaf on the representable sheaf $Y\mapsto \Map_k(Y,X)$. It is clear that
\begin{gather*}
\Phi\Z(X)=L(X)\\
\Z(X)\otimes \Z(Y)=\Z(X\times Y).
\end{gather*}

From this it follows that one has natural isomorphisms
\[\Phi(A\oo^L B)=\Phi(A){\oo^L}_{tr} \Phi(B).\]

On the other hand, the tensor product in $\DM_{-,\et}^{\eff}$ is
defined by descent of ${\oo^L}_{tr}$ via $RC$ \cite[p. 210]{V}. Hence we get an adjunction
isomorphism
\begin{multline*}
\Hom_{Sm(k)_\et}(C_1\oo^L C_2,\G_m[q+1])\simeq \Hom_{\DM_{-,\et}^{\eff}}(RC\circ
\Phi(C_1\oo^L C_2),\Z(1)[q])\\
\simeq \Hom_{\DM_{-,\et}^{\eff}}(RC\circ
\Phi(C_1)\otimes RC\circ
\Phi(C_2)),\Z(1)[q]).
\end{multline*}

Now, since the components of $C_1$ and $C_2$ belong to $\HI_\et$, the counit maps $RC\circ
\Phi(C_1)\to \Tot(C_1)$ and $RC\circ\Phi(C_2)\to \Tot(C_2)$ are isomorphisms. This concludes the
proof.
\end{proof}

\subsection{Two Cartier dualities} Recall the internal Hom $\ihom_\et$ from \S \ref{2.5}. We
define
\begin{equation}\label{D1}
D\1^\et (M) \df\ihom_\et (M, \Z (1))
\end{equation}
for any object $M\in \DM_{\gm,\et}^{\eff}$.\index{$D\1^\et$, $D\1^\Nis$}

We now want to compare the duality \eqref{dcartier} with the following duality on triangulated
$1$-motives:

\begin{sloppypar}
\begin{propose}\label{cd} 
The functor $D\1^\et$
restricts to a self-duality $(\ )^\vee$ (anti-equivalence of categories) on
$d\1\DM_{\gm,\et}^\eff$.
\end{propose}
\end{sloppypar}

\begin{proof} It suffices to compute on motives of smooth
projective curves $M (C)$. Then it is obvious in view of Proposition \ref{lcurve} c).
\end{proof}

\begin{defn}\label{cdef} For $M\in d\1\DM_{\gm,\et}^{\eff}$, we say that $M^\vee$ is the
\emph{motivic Cartier dual} of $M$.
\end{defn}

Note that motivic Cartier duality exchanges Artin motives and Tate motives, \eg $\Z^{\vee} = \Z (1)$. We are going to compare
it with the Cartier duality on $D^b(\M[1/p])$ (see Proposition \ref{pcd}) via Theorem \ref{t1.2.1}.

For two complexes of $1$-motives
$C_1$ and 
$C_2$, by composing \eqref{eqbiextc} and \eqref{edge} and applying Proposition \ref{nat}, we get
a bifunctorial morphism
\begin{equation}\label{bimap}
\Hom(C_1,C_2^*) \to \Biext(C_1,C_2;\G_m)\to\Hom(C_1\oo^L C_2,\G_m[-1])
\end{equation}
where the right hand side is computed in the derived category of \'etale
sheaves. This natural transformation trivially factors through $D^b(\M[1/p])$.

From Proposition \ref{ext=ext}, it follows that the map \eqref{bimap} may be reinterpreted as a
natural transformation
\[\Hom_{D^b(\M[1/p])}(C_1,C_2^*)\to \Hom_{d\1\DM_{\gm,\et}^{\eff}}(\Tot(C_1),\Tot(C_2)^\vee).\]

Now we argue \`a la Yoneda: taking $C_1=C$ and $C_2=C^*$, the image of the identity
yields a canonical morphism of functors:
\[\eta_C:\Tot(C^*)\to \Tot(C)^\vee.\]

\begin{thm}\label{teq} The natural transformation $\eta$ is an isomorphism of functors.
\end{thm}

\begin{proof} It suffices to check this on $1$-motives, since they are dense in the
triangulated category $D^b(\M[1/p])$. Using Yoneda again and the previous discussion, it then follows from
Theorem \ref{t1.2.1} and the isomorphisms \eqref{biextform} and \eqref{eqbiext} (the latter
being proven in Proposition \ref{biext}). The following diagram explains this:
\[\begin{CD}
\Hom(N,M^*)@>\text{Th. \ref{t1.2.1}}>\sim> \Hom(\Tot(N),\Tot(M^*))\\
@V{\wr}V{\eqref{biextform}+\eqref{eqbiext}}V @V{\eta_*}VV\\
\EExt^1_{Sm(k)_\et}(N\oo^L M,\G_m)@<\text{Prop. \ref{ext=ext}}<\sim<
\Hom(\Tot(N),\Tot(M)^\vee) .
\end{CD}\]
\end{proof}

\part{The functors $\LAlb$ and $\RPic$}

\section{A left adjoint to the universal realisation functor}\label{lalb}

The aim of this section is to construct the closest approximation to a left adjoint of the
full embedding $\Tot$ of Definition \ref{tot}: see Theorem \ref{ladj} (and Remark \ref{noleft}
for a caveat). In order to work it out, we first recollect some ideas from
\cite{V0}.

In Section \ref{qlalb}, we shall show that the functor $\LAlb$ of Theorem \ref{ladj} does
provide a left adjoint to $\Tot$ after we tensor Hom groups with $\Q$: this will provide a proof
of Pretheorems announced in \cite[Preth. 0.0.18]{V0} and \cite{V1}.

\subsection{Motivic Cartier duality} 

Recall the functor $D\1^\et:\DM_{\gm,\et}^{\eff}\to \DM_{-,\et}^\eff$ of \eqref{D1}. On the
other hand, by Corollary \ref{cD.2} and Theorem \ref{tfr}, we may consider truncation on
$\DM_{\gm,\et}^\eff$ with respect to the homotopy $t$-structure. We have:

\begin{lemma}\label{l2.0} Let $p:X\to \Spec k$ be a smooth variety. Then the truncated complex
$\tau_{\le 2}D\1^\et(M_\et(X))$ belongs to $d\1\DM_{\gm,\et}^\eff$ (here we set $M_\et(X)\df \alpha^* M(X)$).
\end{lemma}

\begin{proof} The same computation as in the proof of Proposition \ref{lcurve} yields that the
nonvanishing cohomology sheaves are $\sH^1=R_{\pi_0(X)/k}\G_m[1/p]$ and
$\sH^2=\Pic_{X/k}[1/p]$. Both belong to $\Shv_1$ (the latter by Proposition
\ref{p3.3.1}), hence the claim follows from Theorem \ref{t3.2.3}.
\end{proof}

Unfortunately, $\sH^i(D\1^\et(M_\et(X)))$ does not belong to $d\1\DM_{\gm,\et}^\eff$ for $i>2$
in general: indeed, it is well-known that this is a torsion sheaf of cofinite type, with
nonzero  divisible part in general (for $i\ge 3$ and in characteristic $0$, its corank is equal
to the
$i$-th Betti number of $X$). It might be considered as an ind-object of
$d_{\le 0}\DM_{\gm,\et}^\eff$, but this would take us too far. To get around this problem, we
shall restrict to the standard category of geometric triangulated motives of \cite{V},
$\DM_\gm^\eff$.

Let us denote by $D\1^\Nis$ \index{$D\1^\et$, $D\1^\Nis$} the same functor as $D\1^\et$ in the category $\DM_-^\eff$, defined
with the Nisnevich topology. Let as before
$\alpha^*:\DM_-^\eff\to \DM_{-,\et}^\eff$ denote the ``change of topology" functor.

\begin{lemma}\label{l2.1} a) For any smooth $X$ with motive $M(X)\in \DM_\gm^\eff$, we have 
\[\alpha^* D_{\le 1}^\Nis M(X) \iso \tau_{\le 2} 
D_{\le 1}^\et \alpha^* M(X).\]
b) The functor $\alpha^* D_{\le 1}^\Nis$ induces a triangulated functor 
\[\alpha^* D_{\le 1}^\Nis:\DM_\gm^\eff\to d\1\DM_{\gm,\et}^\eff.\]
\end{lemma}

\begin{proof} a) This is the weight $1$ case of the Beilinson-Lichtenbaum conjecture (here 
equivalent to Hilbert's theorem 90.) b) follows from a) and Lemma \ref{l2.0}.
\end{proof}

\begin{defn} We denote by $d\1:\DM_\gm^\eff\to d\1\DM_{\gm,\et}^\eff$ the composite functor $D\1^\et\circ \alpha^*\circ D\1^\Nis$.\index{$d\1$}
\end{defn}

Thus, for $M\in\DM_{\gm}^{\eff}$, we have
\begin{equation}\label{D2}
d\1(M)=\ihom_\et (\alpha^*\ihom_\Nis (M, \Z (1)), \Z (1)).
\end{equation}

The evaluation map $M \otimes\ihom_\Nis (M, \Z (1))\to \Z (1)$ then yields a canonical map
\begin{equation}\label{amap}
a_M :\alpha^*M \to d\1 (M)
\end{equation}
for any object $M\in \DM_{\gm}^{\eff}$.

\begin{propose}\label{cd2} 
The restriction of \eqref{amap} to $d\1\DM_{\gm}^{\eff}$ is an
isomorphism of functors. In particular, we have an equality
\[\alpha^*D\1(\DM_{\gm}^{\eff}) = \alpha^*d\1\DM_{\gm}^{\eff}.\]
\end{propose}

\begin{proof} For the first claim, we reduce to the case $M=M(C)$ where $C$ is a smooth proper
curve. The argument is then exactly the same as in Proposition \ref{cd}, using \eqref{curves}.
The other claim is then clear.
\end{proof}

\subsection{Motivic Albanese} 

\begin{thm}\label{ladj} Let $M\in \DM_{\gm}^{\eff}$. Then
 $a_M$ induces an isomorphism
\[\Hom (d\1 M, M')\iso \Hom (\alpha^*M, M')\]
for any $M'\in d\1\DM_{\gm,\et}^{\eff}$,
\end{thm}

\begin{proof}  By
Proposition \ref{cd}, $M'$ can be
written as $N^\vee=D\1^\et (N)$ for some
$N\in d\1\DM_{\gm,\et}^{\eff}$. We have the following commutative diagram
\[\begin{CD}
\Hom (\alpha^*M, D\1^\et (N)) &=& \Hom (\alpha^*M\otimes N, \Z (1))  &=&\Hom (N, D\1^\et (\alpha^*M))\\
@A{a_M^*}AA @A{(a_M\otimes 1_N)^*}AA  @A{D\1^\et(a_M)_*}AA\\
 \Hom (d\1M, D\1^\et (N)) &=& \Hom (d\1M\otimes N, \Z (1)) &=&\Hom (N,
D\1^\et(d\1 M)).
\end{CD}\]

But $D\1^\et(a_M)\circ a_{D\1^\Nis M}=1_{D\1^\et\alpha^*M}$ and $a_{D\1^\Nis M}$ is an isomorphism by Proposition
\ref{cd2}, which proves the claim.\end{proof}

\begin{defn} \label{LAlb} The \emph{motivic Albanese functor}
\[\LAlb : \DM_\gm^\eff\to D^b(\M[1/p])\]
is the composition of $d\1$ with a quasi-inverse to the equivalence of categories of  Theorem \ref{t1.2.1}. \end{defn}\index{$\LAlb$}

By Theorems \ref{t1.2.1} and \ref{ladj}, the relationship between $\LAlb$ and the functor
$\Tot$  of Definition \ref{tot} is as follows: for $M\in \DM_\gm^\eff$ and $N\in D^b(\M[1/p])$,
the map $a_M$ of \eqref{amap}  induces an isomorphism
\begin{equation}\label{lalbuniv}
\Hom (\LAlb M, N)\iso \Hom (\alpha^*M, \Tot(N)).
\end{equation}

We call $a_M$ the \emph{motivic Albanese map} associated to $M$ for reasons that will appear later.

\begin{remark}\label{noleft} The above suggests that $\LAlb$  might actually extend to a left
adjoint of $\Tot:D^b(\M[1/p])\to \DM_{\gm,\et}^\eff$.
Unfortunately this is not the case, and in fact \emph{$\Tot$ does not have a left adjoint}. 

Indeed, suppose that such a left adjoint exists, and let us denote it by $\LAlb^\et$. For
simplicity, suppose $k$ algebraically closed. Let $n\ge 2$.  For any $m>0$, the exact triangle
in
$\DM_{\gm,\et}^\eff$
\[\Z(n)\by{m}\Z(n)\to \Z/m(n)\by{+1}\]
must yield an exact triangle
\[\LAlb^\et \Z(n)\by{m}\LAlb^\et\Z(n)\to \LAlb^\et\Z/m(n)\by{+1.}\]

Since $\Tot$ is an equivalence on torsion objects, so must be $\LAlb^\et$.  Since $k$ is
algebraically closed, $\Z/m(n)\simeq \mu_m^{\otimes n}$ is constant, hence we must have $\LAlb^\et\Z/m(n)\allowbreak\simeq [\Z/m\to 0]$.
Hence, multiplication by $m$ must be bijective on the $1$-motives $H^q(\LAlb^\et(\Z(n)))$ for
all $q\ne 0,1$, which forces these $1$-motives to vanish. For $q=0,1$ we must have exact
sequences
\begin{multline*}
0\to H^0(\LAlb^\et(\Z(n)))\by{m}H^0(\LAlb^\et(\Z(n)))\to [\Z/m\to 0]\\
\to H^1(\LAlb^\et(\Z(n)))\by{m}H^1(\LAlb^\et(\Z(n)))\to 0
\end{multline*}
which force either $H^0=[\Z\to 0]$, $H^1=0$ or $H^0=0$, $H^1=[0\to\G_m]$. But both cases are
impossible as one easily sees by computing
\begin{multline*}
\Hom(M(\P^n),\Tot([\Z\to 0])[2n+1])= H^{2n+1}_\et(\P^n,\Z)[1/p]\\
\simeq H^{2n}_\et(\P^n,(\Q/\Z)')\simeq (\Q/\Z)'
\end{multline*}
via the trace map, where $(\Q/\Z)'=\bigoplus_{l\ne p}\Q_\ell/\Z_\ell$.

Presumably, $\LAlb^\et$ does exist with values in a suitable pro-category containing $D^b(\M[1/p])$, and
sends $\Z(n)$ to the complete Tate module of $\Z(n)$ for $n\ge 2$. Note that, by
\ref{tatetwist} below, $\LAlb(\Z(n))=0$ for $n\ge 2$, so that the natural transformation
$\LAlb^\et(\alpha^* M)\to \LAlb(M)$ will not be an isomorphism of functors in general.

We shall show in Section \ref{qlalb} that
$\LAlb$ does yield a left adjoint of $\Tot$ after Hom groups have been tensored  with $\Q$. 
\end{remark}

\subsection{Motivic Pic}

\begin{defn} \label{RPic}
The \emph{motivic Picard functor} (a contravariant functor) is the functor
\[\RPic : \DM_{\gm}^{\eff}\to D^b(\M[1/p])\]
given by $\Tot^{-1}\alpha^*D\1^\Nis$ (\cf Definition \ref{LAlb}).
\end{defn}\index{$\RPic$}

For $M\in \DM_{\gm}^{\eff}$ we then have the following
tautology
\[(\Tot \RPic (M))^{\vee} = \Tot \LAlb (M).\]

Actually, from Theorem \ref{teq} we deduce:

\begin{cor}\label{rpdla}
For $M\in \DM_{\gm}^{\eff}$ we have
\[\RPic (M)^* = \LAlb (M).\qed\]
\end{cor}

Therefore we get ${}^tH^i(\RPic (M)) = ({}_tH_i(\LAlb (M)))^*$.

\section{$\LAlb$ and $\RPic$ with rational coefficients}\label{qlalb}

Throughout this section, we use the notations $\otimes\Q$ and $\boxtimes\Q$ from Definition
\ref{1mot}.
\subsection{Rational coefficients revisited}
To give a rational version of Theorem \ref{ladj}, we have to be a little careful. Let
$\DM_-^\eff(k;\Q)$ and $\DM_{-,\et}^\eff(k;\Q)$ denote the full subcategories of
$\DM_-^\eff(k)$ and $\DM_{-,\et}^\eff(k)$ formed of those complexes whose cohomology sheaves
are uniquely divisible. Recall that by \cite[Prop. 3.3.2]{V}, the change of topology functor
\[\alpha^*:\DM_-^\eff(k)\to \DM_{-,\et}^\eff(k)\]
induces an equivalence of categories
\[\DM_-^\eff(k;\Q)\iso \DM_{-,\et}^\eff(k;\Q).\]

Beware that in \loccit, these two categories are respectively denoted by
$\DM_-^\eff(k)\otimes\Q$ and $\DM_{-,\et}^\eff(k)\otimes\Q$, while this notation is used here
according to Definition \ref{1mot}. The composite functors (with our notation)
\begin{gather*}
\DM_-^\eff(k;\Q)\to \DM_-^\eff(k)\to \DM_-^\eff(k)\otimes\Q\\
\DM_{-,\et}^\eff(k;\Q)\to \DM_{-,\et}^\eff(k)\to \DM_{-,\et}^\eff(k)\otimes\Q
\end{gather*}
are fully faithful but not essentially surjective. The functor $\alpha^*\otimes\Q$ is not
essentially surjective, nor (a priori) fully faithful. Nevertheless, these two composite
functors have a left adjoint/left inverse $C\mapsto C\otimes\Q$, and

\begin{propose}\label{ptensq} a) The compositions
\begin{gather*}
\DM_\gm^\eff(k)\otimes\Q\to \DM_-^\eff(k)\otimes\Q\longby{\otimes\Q} \DM_-^\eff(k;\Q)\\
\DM_{\gm,\et}^\eff(k)\otimes\Q\to \DM_{-,\et}^\eff(k)\otimes\Q\longby{\otimes\Q}
\DM_{-,\et}^\eff(k;\Q)
\end{gather*}
are fully faithful.\\
b) Via these full embeddings, the functor $\alpha^*$ induces equivalences of categories
\begin{gather*}
\DM_\gm^\eff(k)\boxtimes\Q\iso \DM_{\gm,\et}^\eff(k)\boxtimes\Q\\
d\1\DM_\gm^\eff(k)\boxtimes\Q\iso d\1\DM_{\gm,\et}^\eff(k)\boxtimes\Q.
\end{gather*}
Here $d\1\DM_\gm^\eff(k)$ is the thick subcategory of $\DM_\gm^\eff(k)$ generated by motives of smooth curves.
\end{propose}

\begin{proof} a) We shall give it for the first composition (for the second one it is
similar). Let $M,N\in\DM_\gm^\eff(k)$: we have to prove that the obvious map
\[\Hom(M,N)\otimes\Q\to \Hom(M\otimes \Q,N\otimes\Q)\]
is an isomorphism. We shall actually prove this isomorphism for any $M\in \DM_\gm^\eff(k)$ and
any $N\in \DM_-^\eff(k)$. By adjunction, the right hand side coincides with $\Hom(M,N\otimes
\Q)$ computed in $\DM_-^\eff(k)$. We may reduce to $M=M(X)$ for $X$ smooth. By \cite[Prop.
3.2.8]{V}, we are left to see that the map
\[H^q_\Nis(X,N)\otimes \Q\to H^q_\Nis(X,N\otimes\Q)\]
is an isomorphism for any $q\in\Z$. By hypercohomology spectral sequences, we reduce to the
case where $N$ is a sheaf concentrated in degree $0$; then the assertion follows from the fact
that Nisnevich cohomology commutes with filtering direct limits of sheaves.

b) It is clear that the two compositions commute with $\alpha^*$, which sends
$\DM_\gm^\eff(k)\otimes\Q$ into $\DM_{\gm,\et}^\eff(k)\otimes\Q$. By a) and \cite[Prop.
3.3.2]{V}, this functor is fully faithful, and the induced functor on the $\boxtimes$
categories remains so and is essentially surjective by definition of the two categories. Similarly for the $d\1$ categories.
\end{proof}

\begin{remarks} 1) In fact, 
$d\1\DM_{\gm,\et}^\eff(k)\otimes\Q=d\1\DM_{\gm,\et}^\eff(k)\boxtimes\Q$ thanks to Corollary
\ref{nobox} and Theorem \ref{t1.2.1}. We don't know whether the same is true for the other
categories.\\
2) See \cite[A.2.2]{riou} for a different, more general approach to Proposition \ref{ptensq}. 
\end{remarks}

\subsection{The functor $\LAlb^\Q$} We now get a rational version of Theorem \ref{ladj} by taking  \eqref{D2} with rational coefficients according with the results collected in \ref{nobox} and \ref{ptensq}.

\begin{cor}\label{cet} The functor $d\1$ of \eqref{D2} induces a left adjoint to the embedding
$d\1\DM_\gm^\eff(k)\boxtimes\Q\into\DM_\gm^\eff(k)\boxtimes\Q$. The Voevodsky-Orgogozo full
embedding
$\Tot:D^b(\M\otimes\Q)\to \DM_\gm^\eff\boxtimes\Q$ has a left adjoint/left inverse
$\LAlb^\Q$.\qed
\end{cor}

\begin{remark}\label{formq} In this case, the formula $d\1=D\1^\et\circ \alpha^*\circ D\1^\Nis$
collapses into the simpler formula
\[d\1=D\1^2\]
with $D\1=D\1^\Nis=D\1^\et$.
\end{remark}

\section{A tensor structure and an internal Hom on
$D^b(\M\otimes\Q$)}\label{tens}

In this section, coefficients are tensored with $\Q$ and we use the functor $\LAlb^\Q$ of
Corollary \ref{cet}.

\subsection{Tensor structure}

\begin{lemma}\label{biextrigid} Let $G_1,G_2$ be two semi-abelian varieties. Then, we have in
$\DM_{\gm,\et}^\eff\boxtimes \Q$:
\[\cH^q(D\1(\uG_1[-1]\otimes \uG_2[-1]))=
\begin{cases}
\Biext(G_1,G_2;\G_m)&\text{if $q=0$}\\
0&\text{else.}
\end{cases}
\]
\end{lemma}

\begin{proof} By Gersten's
principle (Proposition \ref{pgersten}), it is enough to show that the isomorphisms are valid
over function fields
$K$ of smooth $k$-varieties and that $\sH^0$ comes from the small \'etale site of $\Spec k$.
Since we work up to torsion, we may even replace $K$ by its perfect closure. Thus, without loss
of generality, we may assume $K=k$ and we have to show the lemma for sections over $k$.

For $q\le  0$, we use Proposition \ref{ext=ext}: for $q<0$ this follows from Lemma \ref{ld0}, while for $q=0$ it follows from the
isomorphisms \eqref{biextform} and \eqref{eqbiext} (see Proposition \ref{biext}), which show that $\Biext(G_1,G_2;\G_m)$ is rigid. 

For $q>0$, we use the formula
\[D\1(\uG_1[-1]\otimes\uG_2[-1])\simeq \ihom(\uG_1[-1],\Tot([0\to G_2]^*))\]
coming from Theorem \ref{teq}. Writing $[0\to G_2]^*=[L_2\to A_2]$ with $L_2$ a lattice and $A_2$ an abelian variety, we are left to show that
\begin{gather*}
\Hom_{\DM_{\gm,\et}^\eff\boxtimes \Q}(\uG_1,L_2[q+1])=0\text{ for } q>0\\
\Hom_{\DM_{\gm,\et}^\eff\boxtimes \Q}(\uG_1,\underline{A}_2[q])=0\text{ for } q> 0.
\end{gather*}

For this, we may reduce to the case where $G_1$
is either an abelian variety or $\G_m$.  If $G_1=\G_m$, $\uG_1$ is a direct summand of $M(\P^1)[-1]$ and the result follows. If $G_1$ is an abelian variety, it is isogenous to a direct summand of
$J(C)$ for
$C$ a smooth projective geometrically irreducible curve. Then $\uG_1$ is a direct
summand of
$M(C)$, and the result follows again since $L_2$ and $\underline{A}_2$ define locally constant
(flasque) sheaves for the Zariski topology.
\end{proof}

\begin{propose}\label{tens1} a) The functor $\LAlb^\Q:\DM_\gm^\eff\boxtimes\Q\to
D^b(\M\otimes\Q)$ is a localisation functor; it carries the tensor
structure $\otimes$ of $\DM_\gm^\eff\boxtimes\Q$ to a tensor structure $\otimes_1$ on
$D^b(\M\otimes\Q)$. \\
b) For $(M,N)\in \DM_\gm^\eff\boxtimes\Q\times D^b(\M\otimes\Q)$, we have
\[\LAlb^\Q(M\otimes \Tot(N))\simeq \LAlb^\Q(M)\otimes_1 N.\]
c) We have
\[
[\Z\to 0]\otimes_1 C = C\]
for any  $C\in D^b(\M\otimes\Q)$;
\[N_1\otimes_1 N_2 =[L\to G]\] 
for two Deligne $1$-motives  $N_1=[L_1\to G_1]$, $N_2=[L_2\to G_2]$, where
\[L = L_1\otimes L_2;\]
there is an extension
\[0\to \Biext(G_1,G_2;\G_m)^*\to G\to L_1\otimes G_2\oplus L_2\otimes G_1\to 0.
\]
d) The tensor product $\otimes_1$ is exact with respect
to the motivic $t$-structure and respects the weight filtration. Moreover, it is right exact
with respect to the homotopy $t$-structure.\\
e) For two $1$-motives $N_1,N_2$ and a semi-abelian variety $G$, we have
\[\Hom(N_1\otimes_1 N_2,[0\to G])\simeq \Biext(N_1,N_2;G)\otimes\Q.\]
\end{propose}

\begin{proof} a) The first statement is clear since $\LAlb$ is left adjoint
to the fully faithful functor $\Tot$. For the second, it suffices to see
that if $\LAlb^\Q(M)=0$ then $\LAlb^\Q(M\otimes N)=0$ for any $N\in
\DM_\gm^\eff\otimes\Q$. We may check this after applying $\Tot$. Note
that, by Proposition \ref{cd} and Remark \ref{formq}, $\Tot\LAlb^\Q(M)=d\1(M)=0$ is equivalent
to
$D\1(M)=0$.  We have:
\[
D\1(M\otimes N)=\ihom(M\otimes N,\Z(1))
=\ihom(N,\ihom(M,\Z(1)))=0.
\]

b) Let $M' = fibre(\Tot\LAlb^\Q(M)\to M)$: then $\LAlb^\Q(M')=0$. By definition of $\otimes_1$
we then have
\begin{multline*}
\LAlb^\Q(M)\otimes_1 N =\LAlb^\Q(\Tot\LAlb^\Q(M)\otimes \Tot(N))\\
\iso \LAlb^\Q(M\otimes\Tot(N)).
\end{multline*}

c) The first formula is obvious. For the second, we have an exact triangle
\begin{multline*}
G_1[-1]\otimes G_2[-1]\to \Tot(N_1)\otimes \Tot(N_2)\\
\to \Tot([L_1\otimes L_2\to
L_1\otimes G_2\oplus L_2\otimes G_1])\by{+1}
\end{multline*}
hence an exact triangle
\begin{multline*}
\ihom(\Tot([L_1\otimes L_2\to
L_1\otimes G_2\oplus L_2\otimes G_1],\Z(1))\\
\to \ihom(\Tot(N_1)\otimes
\Tot(N_2),\Z(1))
\to \ihom(G_1[-1]\otimes G_2[-1],\G_m) \by{+1}
\end{multline*}

By Lemma \ref{biextrigid}, the last term is $\Biext(G_1,G_2;\G_m)$, hence the claim.

d) Exactness and compatibility with weights follow from the second formula of b); right
exactness for the homotopy $t$-structure holds because it holds on $\DM_-^\eff\otimes\Q$.

e) We have:
\begin{multline*}
\Hom_{\M\otimes\Q}(N_1\otimes_1 N_2,[0\to G]) =
\Hom_{d\1\DM\otimes\Q}(\Tot(N_1\otimes_1 N_2),G[-1])\\
=\Hom_{\DM\otimes\Q}(\Tot(N_1)\otimes
\Tot(N_2),G[-1])=\Biext(N_1,N_2;G)\otimes\Q
\end{multline*}
by Proposition \ref{ext=ext} and formula \eqref{biextform}.
\end{proof}

\begin{remarks} 1) It would be interesting to try and define $\otimes_1$ a
priori, with integral coefficients, and to see whether it is compatible with the
tensor product of $\DM_\gm^\eff$ via the integral $\LAlb$.

2) It is likely that Proposition \ref{tens1} e) generalises to an isomorphism
\[\Hom_{\M\otimes\Q}(N_1\otimes_1 N_2,N) = \Biext(N_1,N_2;N)\otimes\Q\]
for three $1$-motives $N_1,N_2,N$, where the right hand side is the biextension group
introduced by Cristiana Bertolin \cite{bert}, but we have not tried to check it. This would put
in perspective her desire to interpret these groups as Hom groups in the (future) tannakian
category generated by
$1$-motives.

More precisely, one expects that $\DM_\gm\boxtimes\Q$ carries a motivic $t$-struct\-ure whose
heart $\MM$ would be the searched-for abelian category of mixed motives. Then $\M\otimes\Q$
would be a full subcategory of $\MM$ and we might consider the thick tensor subcategory
$\M^\otimes\subseteq \MM$ generated  by $\M\otimes\Q$ and the Tate motive (inverse to the
Lefschetz motive): this is the putative category Bertolin has in mind.

Since the existence of the abelian category of mixed Tate motives (to be contained in
$\M^\otimes$!) depends on the truth of the Beilinson-Soul\'e conjecture, this basic obstruction
appears here too.

Extrapolating from Corollary \ref{cet} and Proposition \ref{tens1}, it seems that the embedding
$\M\otimes\Q\into \MM^\eff$ (where $\MM^\eff$ is to be the intersection of $\MM$ with
$\DM_\gm^\eff\boxtimes\Q$) is destined to have a left adjoint/left inverse
$\Alb^\Q=H_0\circ \LAlb^\Q_{|\MM^\eff}$, which would carry the tensor product of $\MM^\eff$ to
$\otimes_1$. Restricting
$\Alb^\Q$ to $\M^\otimes\cap \MM^\eff$ would provide the link between Bertolin's ideas and
Proposition
\ref{tens1} e).
\end{remarks}

\subsection{Internal Hom}

\begin{propose} \label{ihom1} a) The formula
\[\ihom(C_1,C_2) = (C_1\otimes_1 C_2^*)^*\]
defines an internal $\Hom$ on $D^b(\M\otimes\Q)$, right adjoint to the tensor product of
Proposition \ref{tens1}.\\
b) This internal $\Hom$ is exact for the motivic $t$-structure.
\end{propose}

\begin{proof} a) We have to get a natural isomorphism
\[\Hom(C_1\otimes_1 C_2,C_3)\simeq \Hom(C_1,\ihom(C_2,C_3))\]
for three objects $C_1,C_2,C_3\in D^b(\M\otimes \Q)$.

Let us still write $\otimes_1$ for the tensor product in
$d\1\DM_\gm^\eff\otimes\Q$ corresponding to the tensor product in $D^b(\M\otimes\Q)$. By
definition, we have for $M_1,M_2,M\in d\1\DM$
\begin{multline*}
\Hom_{d\1\DM}(M_1\otimes_1 M_2,M)\simeq \Hom_{\DM}(M_1\otimes M_2,M)\\
=\Hom_{\DM}(M_1, \ihom(M_2,M)).
\end{multline*}

In view of Theorem
\ref{teq}, we are left to show that
$\ihom(M_2,M)\simeq (M_2\otimes_1,M^\vee)^\vee$. By duality, we may replace $M$ by $M^\vee$.
Then:
\begin{multline*}\ihom(M_2,M^\vee)=\ihom(M_2,\ihom(M,\Z(1))\\
\simeq \ihom(M_2\otimes M,\Z(1))\simeq \ihom(M_2\otimes_1 M,\Z(1)) = (M_2\otimes_1 M)^\vee
\end{multline*}
where the second isomorphism follows from Proposition \ref{tens1} b).

b) This follows from the formula of a), Proposition \ref{tens1} d) and Proposition \ref{pcd}
c). 
\end{proof}

\part{Some computations}

\section{The Albanese complexes and their basic properties}\label{6}
We introduce homological and Borel-Moore Albanese complexes of an
algebraic variety providing a computation of their 1-motivic
homology.

We also consider a slightly more sophisticated cohomological
Albanese complex $\LAlb^* (X)$ which is only contravariantly
functorial for maps between schemes of the same dimension. All
these complexes coincide for smooth proper schemes.

\subsection{The homological Albanese complex}
 Let $p: X\to \Spec k$ be a smooth variety. Recall that $X$ has a motive
$M(X)\in
\DM_\gm^\eff$ \cite{V}: $M$ is a covariant functor from the
category $Sm(k)$ of smooth $k$-schemes of finite type to $\DM_\gm^\eff$.
The image of
$M(X)$ via the full embedding $\DM_\gm^\eff\to
\DM_-^\eff$ is given by the Suslin complex $C_*$ of the
representable Nisnevich sheaf with transfers $L(X)$ associated to $X$.

For $X$ an arbitrary $k$-scheme of fnite type, the formula
$M(X)=C_*(L(X))$ still defines an object of $\DM_-^\eff$; if
$\car k=0$, this object is in $\DM_\gm^\eff$ by
\cite[\S 4.1]{V}.

\noindent{\bf Convention.} In this subsection, ``scheme" means $k$-scheme
of finite type if
$\car k=0$ and smooth $k$-scheme of finite type if $\car k>0$.

\begin{defn}\label{LAlb-} We define the {\it homological Albanese
complex} of $X$ by
$$\LAlb(X)\df \LAlb (M (X)).$$
Define, for $i\in\Z$
$$\LA{i}(X)\df {}_tH_i(\LAlb(X))$$ the 1-motives with cotorsion (see Definition \ref{1cot} and
Notation \ref{not}) determined by the homology of the Albanese complex.
\end{defn}

The functor $\LAlb$ has the following properties, easily
deduced from \cite[2.2]{V}:

\subsubsection{Homotopy invariance} For any scheme $X$
the map
$$\LAlb(X\times \Aff^1)\to \LAlb(X)$$ is an isomorphism, thus
$$\LA{i}(X\times \Aff^1)\iso \LA{i}(X)$$
for all $i\in\Z$.

\subsubsection{Mayer-Vietoris} For a scheme $X$ and an open
covering
$X = U\cup V$ there is a distinguished triangle
\[\begin{CD}
\LAlb (U\cap V)  @>>> \LAlb(U)\oplus \LAlb(V)\\
{\scriptstyle +1}\nwarrow &&\swarrow\\
&\LAlb(X)&
\end{CD}\]
and therefore a long exact sequence of 1-motives
$$\cdots\to \LA{i}(U\cap V) \to\LA{i}(U)\oplus \LA{i}(V)\to
\LA{i}(X)\to \cdots$$

\subsubsection{Tate twists}\label{tatetwist} If $X$ is a
smooth scheme and
$n>0$, then
\[\Tot\LAlb(M(X)(n))=
\begin{cases}
0&\text{if $n>1$}\\
M(\pi_0(X))(1)&\text{if $n=1$}
\end{cases}
\]
where $\pi_0(X)$ is the scheme of constants of $X$, see Definition \ref{dpi0}. Indeed
\begin{multline*}
\Tot\LAlb(M(X)(n))=\ihom_\et\alpha^*(\ihom_\Nis(M(X)(n),\Z(1)),\Z(1))\\=
\ihom_\et\alpha^*(\ihom_\Nis(M(X)(n-1),\Z),\Z(1))
\end{multline*}
by the cancellation theorem \cite{voecan}. Now
\[\ihom_\Nis(M(X)(n-1),\Z)=
\begin{cases}
0&\text{if $n>1$}\\
\ihom_\Nis(M(\pi_0(X)),\Z)&\text{if $n=1$}.
\end{cases}
\]

The last formula should follow from \cite[Lemma 2.1 a)]{kmot} but
the formulation there is wrong; however, the formula
immediately follows from the argument in the proof of \loccit,
\ie considering the Zariski cohomology of $X$ with coefficients in
the flasque sheaf $\Z$.

This gives
\begin{equation}\label{eq8.1}
\LA{i}(M(X)(1))=
\begin{cases}
[0\to R_{\pi_0(X)/k}\G_m]&\text{if $i=0$}\\
0&\text{else}
\end{cases}
\end{equation}
where $R_{L/k}(-)$ is Weil's restriction of scalars.

From this computation we formally deduce:

\begin{propose}\label{ptate} For any $M\in \DM_{\gm}^{\eff}$, we have
\begin{itemize}
\item $\LAlb(M(n))=0$ for $n\ge 2$; 
\item $\LAlb(M(1))$ is a
complex of toric $1$-motives.\qed
\end{itemize}
\end{propose}

\subsubsection{Gysin}\label{Gysin} Let $Z$ be a closed smooth subscheme purely of
codimension $n$ of a smooth scheme $X$, and let $U=X-Z$. Then
\[\LAlb(U)\iso \LAlb(X) \text{ if } n>1\]
and if $n=1$ we have an exact triangle
\[\LAlb(U)\to \LAlb(X)\to [0\to R_{\pi_0(Z)/k}\G_m][2]\to
\LAlb(U)[1]\]
hence a long exact sequence of $1$-motives
\begin{multline*}
0\to\LA{2}(U)\to \LA{2}(X)\to [0\to
R_{\pi_0(X)/k}\G_m]\\\to
\LA{1}(U)\to \LA{1}(X)\to 0
\end{multline*}
and an isomorphism $\LA{0}(U)\to \LA{0}(X)$.

\subsubsection{Blow ups}\label{blowups} If $X$ is a scheme and $Z\subseteq X$ is a
closed subscheme, denote by $p :\tilde X\to X$ a proper surjective morphism such that
$p^{-1}(X-Z)\to X-Z$ is an isomorphism, \eg the blow up of $X$ at
$Z$. Then there is a distinguished triangle
$$\begin{CD}
\LAlb(\tilde Z) @>>> \LAlb(\tilde X)\oplus \LAlb(Z)\\
{\scriptstyle +1}\nwarrow &&\swarrow\\
&\LAlb(X)&
\end{CD}$$
with $\tilde Z = p^{-1}(Z)$, yielding a long exact sequence of 1-motives
\begin{equation}\label{exres} 
\cdots\to \LA{i}(\tilde Z) \to\LA{i}(\tilde X)\oplus
\LA{i}(Z)\to \LA{i}(X)\to \cdots\end{equation}

If $X$ and $Z$ are smooth, we get (using \cite[Prop. 3.5.3]{V}
and the above)
\[\LAlb(\tilde X)=\LAlb(X)\oplus [0\to R_{\pi_0(Z)/k}\G_m][2]\]
and corresponding formulas for homology.

\subsubsection{Albanese map} If $X$ is a scheme we have
the natural map  \eqref{amap} in $\DM_{\gm,\et}^{\eff}$
$$a_X :\alpha^*M(X) \to \Tot\LAlb(X)$$
inducing isomorphisms on \'etale motivic cohomology
\[\Hom (M(X), \Z (1)[j]) = \Hom (\LAlb(X),\G_m[j-1]).\]

\subsection{Cohomological Pic} Dual to \ref{LAlb-} we set:

\begin{defn}\label{RPic+} Define the {\it cohomological Picard
complex} of $X$ by
$$\RPic (X)\df \RPic (M (X)).$$
Define, for $i\in\Z$
$$\RA{i}(X)\df {}^tH^i(\RPic (X))$$ 
the 1-motives with torsion
determined by the cohomology of the Picard complex (see Notation \ref{not}).
\end{defn}

The functor $\RPic$ has similar properties to $\LAlb$, deduced by
duality. Homotopical invariance, Mayer-Vietoris, Gysin and the distinguished triangle for abstract blow-ups are clear, and moreover we have
\[ \RPic (M (X)(n))=
\begin{cases}
0&\text{if $n>1$}\\
\relax [\Z\pi_0(X)\to 0]&\text{if $n=1$}.
\end{cases}
\]

We also have that $\RPic(X)=\LAlb(X)^\vee$, hence
$$\RA{i}(X) = \LA{i}(X)^\vee.$$

We shall complete \S \ref{Gysin} by

\subsubsection{$\RPic$ and $\LAlb$ with supports} \label{supports} Let
$X\in Sm(k)$, $U$ a dense open subset of $X$ and $Z=X-U$ (reduced
structure). In $\DM$, we have the \emph{motive with supports} $M^Z(X)$
fitting in an exact triangle
\[M(U)\to M(X)\to M^Z(X)\by{+1}\]
hence the cohomological complex with supports
\[\RPic_Z(X):=\RPic(M^Z(X))\] fitting in an exact triangle
\[\RPic_Z(X)\to \RPic(X)\to \RPic(U)\by{+1}.\]

Dually, we have $\LAlb^Z(X):=\LAlb(M^Z(X))$.

\begin{sloppypar}
\begin{lemma} \label{lsupports} If $\dim X=d$, $\RPic_Z(X)\simeq
[\underline{CH}_{d-1}(Z)\to 0][-2]$, where $\underline{CH}_{d-1}(Z)$ is
the lattice corresponding to the Galois module $CH_{d-1}(Z_{\bar k})$.
\end{lemma}
\end{sloppypar}

Note that $CH_{d-1}(Z_{\bar k})$ is simply the free abelian group with
basis the irreducible components of $Z_{\bar k}$ which are of
codimentsion $1$ in $X_{\bar k}$.

\begin{proof} This follows readily from the exact sequence
\begin{multline*}0\to \Gamma(X_{\bar k},\G_m)\to \Gamma(U_{\bar
k},\G_m)\to CH_{d-1}(Z_{\bar k})\\ H^1(X_{\bar k},\G_m)\to H^1(U_{\bar
k},\G_m)\to 0.
\end{multline*}
\end{proof}

\subsection{Relative $\LAlb$ and $\RPic$}\label{relative}

For $f: Y\to X$ a map of schemes we let $M(X, Y)$ denote the cone of $M(Y)\to
M (X)$.
Note that for a closed embedding $f$ in a proper scheme $X$ we have $M(X,
Y)=M^c(X-Y)$.

We denote by $\LAlb (X, Y)$ and $\RPic (X, Y)$ the resulting complexes of
1-motives.\\

\section{Borel-Moore variants}\label{6c}

\subsection{The Borel-Moore Albanese complex}
Let $p: X\to \Spec k$ be a scheme of finite type over a field $k$ which
admits resolution of singularities. Recall that the motive with
compact support of $X$, denoted $M^c (X)\in\DM_{-}^{\eff}$, has
also been defined in \cite[Sect. 4]{V}. It is the Suslin complex
$C_*$ of the representable presheaf with transfers $L^c (X)$ given
by quasi-finite correspondences.  Since finite implies quasi-finite
we have a canonical map
$M (X) \to M^c (X)$ which is an isomorphism if $X$ is proper over
$k$.

In general, $M^c : Sch^c(k) \to\DM_{\gm}^{\eff}$ is a covariant
functor from the category of schemes of finite type over $k$ and
proper maps between them.

\begin{defn}\label{LAlbc} We define the {\it Borel-Moore Albanese
complex} of $X$ by
$$\LAlb^c (X)\df \LAlb (M^c (X)).$$
Define, for $i\in\Z$
$$\LA{i}^c(X)\df {}_tH_i(\LAlb^c (X))$$ the 1-motivic homology of this complex.
\end{defn}\index{$\LAlb^c $}

Note that we have the following properties:

\subsubsection{Functoriality} The functor $X\mapsto \LAlb^c (X)$ is covariant
for proper maps and contravariant with respect to flat
morphisms of relative dimension zero, for example \'etale morphisms. We have a canonical, covariantly functorial map
$$\LAlb (X)\to \LAlb^c (X)$$
which is an isomorphism if $X$ is proper.

\subsubsection{Localisation triangle} For any closed subscheme
$Y$ of a scheme $X$ we have a triangle
\[\begin{CD}
\LAlb^c (Y)  @>>> \LAlb^c (X)\\
{\scriptstyle +1}\nwarrow &&\swarrow\\
&\LAlb^c(X-Y)&
\end{CD}\]
and therefore a long exact sequence of 1-motives
\begin{equation}\label{loc}
\dots\to \LA{i}^c(Y) \to\LA{i}^c(X)\to
\LA{i}^c(X-Y)\to \LA{i-1}^c(Y) \to\dots
\end{equation}

In particular, let $X$ be a scheme
obtained by removing a divisor $Y$ from a proper scheme $\bar X$,
\ie $X = \bar X -Y$. Then
\begin{multline*}
\cdots\to \LA{1}(Y) \to\LA{1}(\bar X)\to \LA{1}^c(X)\to
\LA{0}(Y)\\
\to\LA{0}(\bar X)\to \LA{0}^c(X)\to \dots.
\end{multline*}

\subsubsection{Albanese map} We have the following natural map \eqref{amap}
\[a_X^c :\alpha^*M^c (X) \to \Tot\LAlb^c(X)\]
which is an isomorphism if $\dim (X)\leq 1$. In general, for any
$X$, $a_X^c$ induces an isomorphism on motivic cohomology with
compact supports, \ie $H^j_c(X, \Q (1)) = \Hom (\LAlb^c
(X),\G_m[j-1]).$

\subsection{Cohomological Albanese complex} 

\begin{lemma}\label{leffe} Suppose $p=1$ (\ie $\car k=0$), and let $n\ge 0$.  For any $X$ of
dimension $\le n$, the motive $M (X)^*(n)[2n]$ is effective. (Here, contrary to the rest of the
paper, $M(X)^*$ denotes the ``usual" dual $\Hom(M(X),\Z)$ in $\DM_\gm$.)
\end{lemma}

\begin{proof} First assume $X$ irreducible. Let $\tilde X\to X$ be a resolution of
singularities of $X$. With notation as in \S \ref{blowups}, we have an exact triangle
\[M(X)^*(n)\to M(\tilde X)^*(n)\oplus M(Z)^*(n)\oplus M(\tilde Z)^*(n)\by{+1}.\]

Since $\tilde X$ is smooth, $M(\tilde X)^*(n)\simeq M^c(X)[-2n]$ is effective by \cite[Th.
4.3.2]{V}; by induction on $n$, so are $M(Z)^*(n)$ and $M(\tilde Z)^*(n)$ and therefore
$M(X)^*(n)$ is effective.

In general, let $X_1,\dots,X_r$ be the irreducible components of $X$. Suppose $r\ge 2$ and
let $Y=X_2\cap\dots \cap X_n$: since $(X_1,Y)$ is a cdh cover of $X$, we have an exact triangle
\[M(X)^*(n)\to M(X_1)^*(n)\oplus M(Y)^*(n)\oplus M(X_1\cap Y)^*(n)\by{+1}.\]

The same argument then shows that $M(X)^*(n)$ is effective, by induction on $r$.
\end{proof}

We can therefore apply our functor $\LAlb$ and obtain another complex  $\LAlb (M (X)^*(n)[2n])$
of 1-motives. If $X$ is smooth this is just the Borel-Moore Albanese.

\begin{defn}\label{LAlb*} We define the {\it cohomological Albanese
complex} of a scheme $X$ of dimension $n$ by
$$\LAlb^* (X)\df \LAlb (M (X^{(n)})^*(n)[2n])$$
where $X^{(n)}$ is the union of the $n$-dimensional components of $X$. Define, for $i\in\Z$
$$\LA{i}^*(X)\df {}_tH_i(\LAlb^* (X))$$ the 1-motivic homology of this complex.
\end{defn}\index{$\LAlb^*$}

\begin{lemma}\label{l9.2} a) If $Z_1,\dots Z_n$ are the irreducible components of dimension $n$ of $X$, then the
cone of the natural map
\[\LAlb^*(X)\to \bigoplus \LAlb^*(Z_i)\]
is a complex of groups of multiplicative type.\\
b) If $X$ is integral and $\tilde X$ is a desingularisation of $X$, then the cone of the
natural map
\[\LAlb^*(X)\to \LAlb^*(\tilde X)\]
is a complex of groups of multiplicative type.
\end{lemma}

\begin{proof} a) and b) follow from dualising the abstract blow-up exact triangles of \cite[2.2]{V} and applying Proposition \ref{ptate}.
\end{proof}

\subsection{Compactly supported and homological Pic} We now
consider the dual complexes of the Borel-Moore and cohomological
Albanese.
\begin{defn} Define the {\it compactly supported Picard
complex} of any scheme $X$ by
$$\RPic^c (X)\df \RPic (M^c (X))$$
and the {\it homological Picard complex} of an equidimensional
scheme $X$ of dimension $n$ by
$$\RPic^* (X)\df \RPic (M (X)^*(n)[2n]).$$
Denote $\RA{i}^c(X)\df {}^tH^i(\RPic^c (X))$ and $\RA{i}^*(X)\df
{}^tH^i(\RPic^* (X))$ the 1-motives with torsion determined by the
homology of these Picard complexes.
\end{defn}\index{$\RPic^c$, $\RPic^*$}

Recall that $\RPic^c (X) = \RPic (X)$ if $X$ is proper and
$\RPic^c (X) = \RPic^* (X)$ if $X$ is smooth.

\subsection{Topological invariance} To conclude this section and the previous one, we note the
following useful

\begin{lemma} \label{l12.3} Suppose that $f:Y\to X$ is a universal topological homeomorphism,
in the sense that $1_U\times f:U\times Y\to U\times X$ is a homeomorphism of topological spaces
for any smooth $U$ (in particular $f$ is proper). Then $f$ induces isomorphisms $\LAlb(Y)\iso
\LAlb(X)$, $\RPic(X)\iso\RPic(Y)$,  $\LAlb^c(Y)\iso \LAlb^c(X)$ and $\RPic^c(X)\iso
\RPic^c(Y)$. Similarly, $\LAlb^*(X)\iso \LAlb^*(Y)$ and $\RPic^*(Y)\iso \RPic^*(X)$. This
applies in particular to
$Y=$ the semi-normalisation of
$X$.
\end{lemma}

\begin{proof} It suffices to notice that $f$ induces isomorphisms $L(Y)\iso L(X)$ and
$L^c(Y)\to L^c(X)$, since by definition these sheaves only depend on the underlying topological
structures.
\end{proof}

Lemma \ref{l12.3} implies that in order to compute $\LAlb(X)$, etc., we may always assume $X$
semi-normal if we wish so.

\section{Computing $\LAlb(X)$ and $\RPic(X)$ for smooth $X$}\label{comp}

\subsection{The Albanese scheme} Let $X$ be a reduced $k$-scheme of finite type. ``Recall"
(\cite[Sect. 1]{ram}, \cite{spsz})  the Albanese scheme
$\cA_{X/k}$ \index{$\cA_{X/k}$} fitting in the following extension
\begin{equation}\label{unext}
0\to \cA_{X/k}^0\to \cA_{X/k}\to \Z[\pi_0(X)]\to 0
\end{equation}
where $\cA_{X/k}^0$ is Serre's generalised Albanese semi-abelian
variety, and $\pi_0(X)$ is the scheme of constants of $X$ viewed as
an \'etale sheaf on $Sm(k)$.\footnote{In the said references,
$\cA_{X/k}^0$ is denoted by
$Alb_X$ and
$\cA_{X/k}$ is denoted by $\widetilde{Alb}_X$.} In particular,
\[\cA_{X/k}\in \AbS \text{ (see Definition \ref{s1a}).}
\]

There is a
canonical morphism
\begin{equation}\label{alb}
\bar a_X: X
\to\cA_{X/k}
\end{equation}
which is  universal for morphisms from $X$ to group schemes of
the same type. 

For the existence of $\cA_{X/k}$, the reference 
\cite[Sect. 1]{ram} is sufficient if $X$ is a variety (integral $k$-scheme of finite type), hence if $X$ is normal (for example smooth): this will be sufficient in this section. For the general case, see \S \ref{albs}.

We shall denote the object $\Tot^{-1}(\underline{\cA}_{X/k})\in D^b(\M[1/p])$
simply by $\cA_{X/k}$. As seen in Lemma \ref{l4.3.1}, we have 
\[H_i(\cA_{X/k})=
\begin{cases}
\relax [\Z[\pi_0(X)]\to 0]&\text{for $i=0$}\\
\relax [0\to \cA_{X/k}^0]&\text{for $i=1$}\\
0&\text{for $i\ne 0,1$.}
\end{cases}
\]

\subsection{The main theorem} Suppose $X$ smooth. Via \eqref{eq1.4}, \eqref{alb} induces  a
composite map 
\begin{equation}\label{precan}
M_\et(X) \to M_\et(\cA_{X/k})\to \cA_{X/k}.
\end{equation}

Theorem \ref{teq} gives:

\begin{lemma}\label{l7.2.1} We have an exact triangle
\[\Z[\pi_0(X)]^*[0]\to\ihom(\cA_{X/k},\Z(1))\to (\cA_{X/k}^0)^*[-2]\by{+1}\qed\] 
\end{lemma}

By Lemma \ref{l7.2.1}, the map
\[\ihom(\cA_{X/k},\Z(1))\to\ihom(M_\et(X),\Z(1))\]
deduced from \eqref{precan} factors into a map
\begin{equation}\label{eqdual}
\ihom(\cA_{X/k},\Z(1))\to\tau_{\le 2}\ihom(M_\et(X),\Z(1)).
\end{equation}

Applying Proposition \ref{cd} and Lemma \ref{l2.1}, we therefore get a canonical map in
$D^b(\M[1/p])$
\begin{equation}\label{can}
\LAlb (X) \to \cA_{X/k}.
\end{equation}

\begin{sloppypar}
\begin{thm} \label{trunc} Suppose $X$ smooth. Then the map \eqref{can} sits in an exact triangle
\[[0\to \NS_{X/k}^*][2]\to \LAlb (X) \to \cA_{X/k}\longby{+1}\]
where $\NS_{X/k}^*$ denotes the group of multiplicative type dual to $\NS_{X/k}$
(\cf Definition \ref{dNS} and Proposition \ref{belong}).
\end{thm}
\end{sloppypar}

This theorem says in particular that, on the object $\LAlb (X)$, the
motivic $t$-structure and the homotopy $t$-structure are compatible in a
strong sense.

\begin{cor}\label{HLAlb} For $X$ smooth over $k$ we have
$$\LA{i}(X) =
\begin{cases}
\relax [\Z[\pi_0(X)]\to 0]& \text{if $i = 0$}\\
\relax [0\to\cA_{X/k}^0] & \text{if $i= 1$}\\
\relax [0\to \NS_{X/k}^*] & \text{if $i= 2$}\\
0 & \text{otherwise.}
\end{cases}$$
\end{cor}

\begin{cor}\label{clalb} For $X$ smooth, $\LA{1}(X)$ is isomorphic to the homological Albanese
$1$-motive
$\Alb^- (X)$ of \cite{BSAP}.\qed
\end{cor}

\subsection{Reformulation of Theorem \ref{trunc}}   

It is sufficient to get an exact triangle after application of $D\1\circ\Tot$, so that we have
to compute the cone of the morphism \eqref{eqdual} in
$\DM_{\gm,\et}^\eff$. We shall use:

\begin{lemma}\label{l4.1} For $\cF\in \HI_\et^s$, the morphism $b$ of Proposition \ref{p1} is
induced by \eqref{eq1.4}.
\end{lemma}

\begin{proof} This is clear by construction, since
$\Hom_{\DM_\et}(M(G),\cF[1])=H^1_\et(G,\cF)$ \cite[Prop. 3.3.1]{V}.
\end{proof}

Taking the cohomology sheaves of \eqref{eqdual}, we get morphisms
\begin{gather}
\shom(\cA_{X/k},\G_m)\to p_*\G_{m,X}\label{homa}\\
f:\sext(\cA_{X/k},\G_m)\longby{b}\underline{\Pic}_{\cA_{X/k}/k}\longby{\bar a^*}
\underline{\Pic}_{X/k}\label{exta}
\end{gather}
where in \eqref{exta}, $b$ corresponds to the map of Proposition \ref{p1} thanks to Lemma
\ref{l4.1}. Thanks to Proposition \ref{p4.2}, Theorem \ref{trunc} is then equivalent to the
following

\begin{thm}\label{t6.3} Suppose $k$ algebraically closed; Then\\
a) \eqref{homa} yields an isomorphism $\Hom(\cA_{X/k},\G_m)\iso \Gamma(X,\G_m)$.\\
b) \eqref{exta} defines a short exact sequence
\begin{equation}\label{eq3}
0\to \Ext(\underline{\cA}_{X/k},\G_m)\longby{f} \Pic(X)\longby{e}\NS(X)\to 0
\end{equation}
where $e$ is the natural map.
\end{thm}

Before proving Theorem \ref{t6.3}, it is convenient to prove Lemma
\ref{l6.1} below. Let $A_{X/k}$  be the abelian part of $\cA_{X/k}^0$;
then the sheaf $\sext(\underline{A}_{X/k},\G_m)$ is represented by the dual
abelian variety $A^*_{X/k}$. Composing with the map $f$ of \eqref{exta},
we get a map of $1$-motivic sheaves
\begin{equation}\label{eqadj}
\underline{A}^*_{X/k}\to \underline{\Pic}_{X/k}.
\end{equation}

\begin{lemma}\label{l6.1} The map \eqref{eqadj}
induces an isogeny in $\SAb$
\[A^*_{X/k}\onto \gamma(\underline{\Pic}_{X/k})=\underline{\Pic}^0_{X/k}\]
where $\gamma$ is the adjoint functor appearing in Theorem \ref{t3.2.4}
a).
\end{lemma}

\begin{proof} We proceed in 3 steps:
\begin{enumerate}
\item The lemma is true if $X$ is smooth projective: this follows from
the representability of $\Pic^0_X$ and the duality between the Picard and the Albanese varieties.
\item Let $j:U\to X$ be an open immersion: then the lemma is true for $X$
if and only if it is true for $U$. This is clear since $\underline{\Pic}_{X/k}\to\underline{\Pic}_{U/k}$ is an epimorphism with discrete kernel.
\item Let $\phi:Y\to X$ be an \'etale covering. If the lemma is true for
$Y$, then it is true for $X$. This follows from the existence of
transfer maps $\phi_*:A^*_{Y/k}\to A^*_{X/k}$, $\underline{\Pic}_{Y/k}\to \underline{\Pic}_{X/k}$
commuting with the map of the lemma, plus the usual transfer argument. 
\end{enumerate}
We conclude by de Jong's theorem \cite[Th. 4.1]{DJ}. 
\end{proof}

\subsection{Proof of Theorem \ref{t6.3}}

We may obviously suppose that $X$ is irreducible.

a) is obvious from the universal property of $\cA_{X/k}$. For b) we
proceed in two steps:

\begin{enumerate}
\item Verification of $ef=0$.
\item Proof that the sequence is exact.
\end{enumerate}

(1) As above, let $A=A_{X/k}$ be the abelian part of $\cA_{X/k}^0$. In the diagram
\[\Ext(A,\G_m)\to \Ext(\cA_{X/k}^0,\G_m)\leftarrow\Ext(\cA_{X/k},\G_m)\]
the first map is surjective and the second map is an isomorphism, hence
we get a surjective map
\[v:\Ext(A,\G_m)\to \Ext(\cA_{X/k},\G_m).\]

Choose a rational point $x\in X(\bar k)$.  We have a diagram
\[\begin{CD}
\Ext(A,\G_m)@>v>> \Ext(\cA_{X/k},\G_m)\\
@V{a}VV @V{f}VV\\
\Pic(A)@>x^*>> \Pic(X)\\
@V{c}VV @V{e}VV\\
\NS(A)@>x^*>> \NS(X)
\end{CD}\] 
in which
\begin{thlist}
\item $a$ is given by \cite[p. 170, prop. 5 and 6]{gacl} (or by
Proposition \ref{p1}).
\item $ca=0$ (ibid., p. 184, th. 6; ).
\item $x^*$ is induced by the ``canonical" map $X\to A$ sending
$x$ to $0$. 
\end{thlist}

Lemma \ref{trans} applied to $G=\cA_{X/k}$ implies that the top square
commutes (the bottom one trivially commutes too). Moreover,  since $v$ is
surjective and $ca=0$, we get $ef=0$.

(2) In the sequence \eqref{eq3}, the surjectivity of $e$ is clear. Let us prove the
injectivity of $f$: suppose that $f(\cE)$ is trivial. In the pull-back diagram
\[\begin{CD}
\bar a^*\cE@>\pi'>> X\\
@V{\bar a'}VV @V{\bar a}VV\\
\cE@>\pi>> \cA_{X/k}
\end{CD}\]
$\pi'$ has a section $\sigma'$. Observe that $\cE$ is a locally
semi-abelian scheme: by the universal property of $\cA_{X/k}$, the
morphism $\bar a'\sigma'$ factors canonically through $\bar a$. In other
words, there exists $\sigma:\cA_{X/k}\to \cE$ such that $\bar
a'\sigma'=\sigma\bar a$. Then
\[\pi\sigma\bar a=\pi \bar a'\sigma'=\bar a\pi'\sigma'=\bar a\]
hence $\pi\sigma=1$ by reapplying the universal property of $\cA_{X/k}$, and $\cE$ is trivial.
Finally, exactness in the middle follows immediately from Proposition \ref{belong} and Lemma
\ref{l6.1}.\qed

\begin{cor} The isogeny of Lemma \ref{l6.1} is an isomorphism.
\end{cor}

\begin{proof} This follows from the injectivity of $f$ in \eqref{eq3}.
\end{proof}

\subsection{An application}

\begin{cor}\label{divzero} Let $X$ be a smooth $k$-variety of dimension $d$, $U$ a dense open
subset and $Z=X-U$ (reduced structure). Then the morphism $\cA_{U/k}\to \cA_{X/k}$ is epi; its
kernel $T_{X/U,k}$ is a torus whose character group fits into a short exact sequence
\[0\to T_{X/U,k}^*\to \underline{\CH}_{d-1}(Z)\to \NS_Z(X)\to 0\]
where $\underline{\CH}_{d-1}(Z)$ is as in Lemma \ref{lsupports} and $\NS_Z(X)=\ker(\NS(X)\to
\NS(U)$.
\end{cor}

\begin{proof} To see that $\cA_{U/k}\to \cA_{X/k}$ is epi with kernel of multiplicative type, it
is sufficient to see that $\pi_0(\cA_{U/k})\iso \pi_0(\cA_{X/k})$ and that $A_{U/k}\iso
A_{X/k}$. The first isomorphism is obvious and the second one follows from \cite[Th.
3.1]{milne}. The characterisation of $T_{X/U,k}$ is then an immediate consequence of Theorem
\ref{trunc} and Lemma \ref{lsupports}; in particular, it is a torus.
\end{proof}

\subsection{$\RPic(X)$} 

Recall that for $X$ smooth projective $\cA_{X/k}^0 = A_{X/k}$ is the classical Albanese
abelian variety $\Alb (X)$. In the case where $X$ is obtained by removing a divisor $Y$ from a
smooth proper scheme $\bar X$, $\cA_{X/k}^0$, can be described as follows (\cf \cite{BSAP}).
Consider the (cohomological Picard) 1-motive $\Pic^+ (X)\df [\Div_Y^0(\bar X)\to \Pic^0 (\bar
X)]$: its Cartier dual is $\cA_{X/k}^0$ which can be represented as a torus bundle 
\[0\to T_{\bar X/X,k} \to \cA_{X/k}^0\to \Alb (\bar X)\to 0\]
where $ T_{\bar X/X,k}$ has
character group $\Div_Y^0(\bar X)$ according to Corollary~\ref{divzero}.

From the previous remarks and Corollary~\ref{clalb}, we deduce:

\begin{cor}\label{HRPic} If $X$ is smooth,  $\RA{1}(X)$ is
isomorphic to the $1$-motive $\Pic^+(X)$ of \cite{BSAP} (the Cartier dual of $\Alb^-(X)$). If $\bar X$ is a smooth compactification of $X$, then
$$\RA{i}(X) =
\begin{cases}
\relax [0\to \Z[\pi_0(X)]^*]& \text{if $i = 0$}\\
\relax [\Div_Y^0(\bar X)\to\Pic^0 (\bar X)] & \text{if $i= 1$}\\
\relax [\NS (X)\to 0] & \text{if $i= 2$}\\
0 & \text{otherwise}
\end{cases}$$
where $Y=\bar X - X$. 
\end{cor}

\section{1-motivic homology and cohomology of singular schemes}\label{comps}

\subsection{$\cA_{X/k}$ for $X\in Sch(k)$}\label{albs} In this subsection, we extend the
construction of $\cA_{X/k}$ to arbitrary reduced $k$-schemes of finite type, starting
from the case where $X$ is integral (which is treated in \cite[Sect. 1]{ram}). So far, $k$ may
be of any characteristic.

To make the definition clear:

\begin{defn} Let $X\in Sch(k)$. We say that \emph{$\cA_{X/k}$ exists} if the functor
\begin{align*}
\AbS&\to Ab\\
G&\mapsto G(X)
\end{align*}
is corepresentable.
\end{defn}

First note that $\cA_{X/k}$ does not exist (as a semi-abelian scheme, at least) if $X$ is not
reduced. For example, for $X=\Spec k[\epsilon]$ with $\epsilon^2=0$, we have an exact sequence
\[0\to \G_a(k)\to \Map_k(X,\G_m)\to \G_m(k)\to 0\]
which cannot be described by $\Hom(\cA,\G_m)$ for any semi-abelian scheme $\cA$.

On the other hand, $M(X)=M(X_\red)$ for any $X\in Sch(k)$, where $X_\red$ is the reduced
subscheme of $X$ (see proof of Lemma \ref{l12.3}), so we are naturally led to
neglect nonreduced schemes.

\begin{lemma}\label{lred} Let $Z\in Sch(k)$, $G\in \AbS$ and $f_1,f_2:Z\rightrightarrows G$
two morphisms which coincide on the underlying topological spaces (thus, $f_1 = f_2$ if $Z$ is
reduced). Then there exists a largest quotient $\bar G$ of $G$ such that $\pi_0(G)\iso
\pi_0(\bar G)$ and the two compositions
\[Z\rightrightarrows G\to \bar G\]
coincide.
\end{lemma}

\begin{proof} The set $S$ of such quotients $\bar G$ is in one-to-one correspondence with the
set of closed subgroups $H^0\subseteq G^0$. Clearly $\pi_0(G)\in S$, and if $\bar
G_1=G/H_1^0\in S$, $\bar G_2=G/H_2^0\in S$, then $G_3 = G/(H_1^0\cap H_2^0)\in S$. Therefore
$S$ has  a smallest element, since it is Artinian (compare proof of Proposition \ref{pladj}).
\end{proof}

\begin{propose}\label{albexists} $\cA_{X/k}$ exists for any reduced $X\in Sch(k)$.
\end{propose}

\begin{proof} When $X$ is integral, this is \cite[Sect. 1]{ram}. Starting from this case, we
argue by induction on $\dim X$. Let $Z_1,\dots, Z_n$ be the irreducible components of $X$ and
$Z_{ij} = Z_i\cap Z_j$. 

By induction, $\cA_{ij}\df\cA_{(Z_{ij})_\red/k}$ exists for any $(i,j)$. Consider
\[\cA=\coker\left(\bigoplus \cA_{ij}\to \bigoplus \cA_i\right)\]
with $\cA_i = \cA_{Z_i/k}$. Let $Z = \coprod Z_{ij}$ and $f_1,f_2:Z\rightrightarrows \coprod
Z_i$ the two inclusions: the compositions $f_1,f_2:Z\rightrightarrows \cA$ verify the
hypothesis of Lemma \ref{lred}. Hence there is a largest quotient $\cA'$ of $\cA$ with
$\pi_0(\cA)\iso \pi_0(\cA')$, equalising $f_1$ and $f_2$. Then the composition
\[\coprod Z_i\to \bigoplus \cA_i\to \cA'\]
glues down to a morphism $X\to \cA'$. It is clear that $\cA'=\cA_{X/k}$ since,  for any
commutative group scheme $G$, the sequence
\[0\to \Map_k(X,G)\to \bigoplus \Map_k(Z_i,G)\to \bigoplus \Map_k(Z_{ij},G)\]
is exact. 
\end{proof}

Unfortunately this result is only useful to understand $\LA{1}(X)$ for $X$ ``strictly reduced",
as we shall see below. In general, we shall have to consider Albanese schemes for the $\eh$
topology.

\subsection{The $\eh$ topology} In this subsection and the next ones, we assume that $k$ is of characteristic $0$. Recall that $\HI_\et=\HI_\et^s$ in this case by Proposition \ref{pD.1.4}.

The following \'etale analogue of the cdh topology was first considered by Thomas Geisser
\cite{geisser}:

\begin{defn}\label{deh} The \emph{$\eh$ topology} on $Sch(k)$ is the topology generated  by
the \'etale topology and coverings defined by abstract blow-ups (it is the same as \cite[Def.
4.1.9]{V} by replacing the Nisnevich topology by the \'etale topology).
\end{defn}

As in \cite[Th. 4.1.10]{V} (see \cite[Proof of Th. 14.20]{VL} for more details), one has:

\begin{propose}\label{peh} Let $C\in \DM_{-,\et}^\eff$. Then, for any $X\in Sch(k)$
and any $q\in\Z$ one has 
\[\Hom_{\DM_{-,\et}^\eff}(M_{\et}(X),C[q])\simeq H^q_\eh(X,C_\eh[q]).\]
In particular, if $X$ is smooth then $H^q_\et(X,C_\et)\iso H^q_\eh(X,C_\eh)$.\qed
\end{propose}

(See \cite[Th. 4.3]{geisser} for a different proof of the second statement.)

The following lemma will be used many times:

\begin{lemma}[Blow-up induction] \label{lbl} Let $\cA$ be an abelian category.\\
a) Let $\cB\subseteq \cA$ be a thick subcategory and $H^*:Sch(k)^{op}\to \cA^{(\N)}$ a functor with the following property: given an abstract blow-up  as in \S \ref{blowups}, we have a long exact sequence
\[\dots\to H^i(X)\to H^i(\tilde X)\oplus H^i(Z)\to H^i(\tilde X)\to H^{i+1}(X)\to\dots\]
Let $n\ge 0$, and assume that $H^i(X)\in \cB$ for $i\le n$ and $X\in Sm(k)$. Then $H^i(X)\in \cB$ for $i\le n$ and all $X\in Sch(k)$.\\
b) Let $H_1^*,H_2^*$ be two functors as in a) and $\phi^*:H_1^*\to H_2^*$ be a natural transformation. Let $n\ge 0$, and suppose that $\phi^i_X$ is an isomorphism for all $X\in Sm(k)$ and $i\le n$. Then $\phi^i_X$ is an isomorphism for all $X\in Sch(k)$ and $i\le n$.\\
We get the same statements as a) and b) by replacing ``$i\le n$" by ``$i\ge n+\dim X$".
\end{lemma}

\begin{proof} Induction on $\dim X$ in two steps: 1) if $X$ is integral, choose a resolution of singularities $\tilde X\to X$; 2) in general, if $Z_1,\dots, Z_r$ are the irreducible components of $X$, choose $\tilde X=\coprod Z_i$ and $Z=\bigcup_{i\ne j} Z_i\cap Z_j$.
\end{proof}

\begin{examples}\label{ex11.1} 1) Thanks to \cite[Th. 4.1.10]{V} and Proposition \ref{peh}, cdh
or
$\eh$ cohomology with coefficients in an object of $\DM_-^\eff$ or $\DM_{-,\et}^\eff$ satisfy the hypothesis of a) (here $\cA=$ abelian groups).
(See also \cite[Prop. 3.2]{geisser} for a different proof.)

2) \'Etale cohomology with torsion coefficients satisfies the hypothesis of a) by \cite[Prop.
2.1]{pw} (recall that the proof of \loccit relies on the proper base change theorem).
\end{examples}

Here is a variant of Lemma \ref{lbl}:

\begin{propose}\label{pdescent} a) Let $X\in Sch(k)$ and $\Xs \to X$ be a hyperenvelope in the sense of Gillet-Soul\'e \cite[1.4.1]{gs}. Let $\tau = \cdh$ or $\eh$. Then, for any (bounded below) complex of sheaves $C$ over $Sch(X)_\tau$, the augmentation map
\[H^*_\tau(X,C)\to H^*_\tau(\Xs,C)\]
is an isomorphism.\\
b) Suppose that $X_0$ and $X_1$ are smooth and $\cF$ is a homotopy invariant Nisnevich  (if $\tau=\cdh$) or \'etale  (if $\tau=\eh$) sheaf with transfers. Then we have
\[\cF_\tau(X) = \ker(\cF(X_0)\to \cF(X_1)).\]
\end{propose}

\begin{proof} a) By \cite[Lemma 12.26]{VL}, $\Xs$ is a proper $\tau$-hypercovering (\cf \cite[p. 46]{bondarko}). Therefore the proposition follows from the standard theory of cohomological descent.

b) Let us take  $C=\cF_\tau[0]$.  By a) we have a descent spectral sequence which gives a short exact sequence
\[0\to\cF_\tau(X)\to \cF_\tau(X_0)\to \cF_\tau(X_1)\]
and the conclusion now follows from \ref{peh}.
\end{proof}

\subsection{$\LA{i}(X)$ for $X$ singular} 
The following is a general method for computing the $1$-motivic homology of $\LAlb^\Q (X)$:

\begin{propose} \label{cdhex}
If $\car k=0$ and $X\in Sch(k)$ consider cdh cohomology groups $\HH^i_\cdh(X,\pi^*(N))_{\sQ}$,
where
$\pi:Sch(k)_\cdh\to Sch(k)_\Zar$ is the canonical map from the cdh
site to the big Zariski site. Then we have short exact sequences,
for all $i\in \Z$
\begin {multline*}
0\to \Ext^1 (\LA{i-1}^\Q(X),N)\to \HH^i_\cdh(X,\pi^*(N))_{\sQ}\\
\to\Hom (\LA{i}^\Q(X),N)\to 0
\end{multline*}
\begin{multline*}
0\to\Ext^1 (N, \LA{i+1}^\Q(X))\to\EExt^{-i} (N, \LAlb^\Q (X))\\ \to\Hom (N,
\LA{i} (X))\to 0.
\end{multline*}
\end{propose}

\begin{proof}
For any 1-motive $N\in \M$ we have a spectral sequence
\begin{equation}\label{eqss}
E^{p,q}_2 = \Ext^p (\LA{q}(X),N)\ \implies\  \EExt^{p+q}
(\LAlb(X), N)
\end{equation} 
yielding the following short exact sequence
\begin{multline*}
0\to \Ext^1 (\LA{i-1}^\Q(X),N)\to \EExt^{i} (\LAlb^\Q(X),N)\to \\
\Hom (\LA{i}^\Q(X),N)\to 0
\end{multline*}
because of Proposition~\ref{iso1}. By adjunction we
also obtain
$$\EExt^{i} (\LAlb^\Q(X),N) = \Hom (\LAlb^\Q(X),N[i])\cong
\Hom (M (X),\Tot N[i]).$$

Now from \cite[Thm. 3.2.6 and Cor. 3.2.7]{V}, for $X$ smooth
we have
\[\Hom (M (X),\Tot N[i])\cong \HH^i_\Zar(X, N)_{\sQ}.\]

If $k$ is of characteristic $0$ and $X$ is arbitrary, we get the same
isomorphism with cdh hypercohomology by \cite[Thm. 4.1.10]{V}.

The proof for the second short exact sequence is similar.
\end{proof}

One gets similar computations integrally by replacing the cdh topology by the \'eh topology,
but here the spectral sequence does not degenerate. In any case we shall obtain the following
integral results directly (except for Proposition \ref{pl1}).

The following proposition folllows readily by blow-up induction (Lem\-ma \ref{lbl}) from Corollary \ref{HLAlb} and the exact sequences \eqref{exres}:

\begin{propose}\label{c3.1} For any $X\in Sch(k)$ of dimension $d$ in characteristic $0$, we
have\\
a) $\LA{i}(X) = 0$ if   $i <0$.\\
b) $\LA{0}(X) = [\Z[\pi_0(X)]\to 0]$.\\
c) $\LA{i}(X) = 0$ for $i >\max(2,d+1)$.\\
d) $\LA{d+1}(X)$ is a group of multiplicative type.
\qed
\end{propose}

\subsection{The cohomological 1-motives $\RA{i}(X)$}

If $X\in Sch(k)$, we quote the following variant of Proposition \ref{cdhex}:

\begin{lemma} \label{cdhexp}
Let $N\in \M\otimes\Q$ and $X\in Sch(k)$. We have a short exact
sequence, for all $i\in \Z$
\begin {multline*}
0\to \Ext (N, \RA{i-1}(X))\to \HH^i_\cdh(X,\pi^*(N^*))_{\sQ}\\\to \Hom (N, \RA{i}(X))\to 0
\end{multline*}
here $\pi : Sch(k)_\cdh\to Sch(k)_\Zar$ and $N^*$ is the
Cartier dual.
\end{lemma}
\begin{proof} The spectral sequence
$$E^{p,q}_2 = \Ext^p (N, \RA{q}(X))\ \implies\  \EExt^{p+q}
(N, \RPic (X))$$ yields the following short exact sequence
\begin{multline*}
0\to \Ext (N, \RA{i-1}(X))\to \EExt^{i} (N,\RPic (X))\to \\
\Hom (N, \RA{i}(X))\to 0
\end{multline*}
and by Cartier duality, the universal property and \cite[Thm.
4.1.10]{V} we obtain:
\begin{multline*}
\EExt^{i} (N,\RPic (X)) = \Hom (N,\RPic (X)[i])\cong \Hom (\LAlb
(X), N^*[i])=\\ \Hom (M (X),N^*[i])\cong \HH^i_\cdh(X,
\pi^*(N^*))_{\sQ}.
\end{multline*}
\end{proof}

On the other hand, here is a dual to Proposition \ref{c3.1}:

\begin{propose}\label{c3.1*} For any $X\in Sch(k)$ of dimension $d$ in characteristic $0$, we
have\\
a) $\RA{i}(X) = 0$ if   $i <0$.\\
b) $\RA{0}(X) = [0\to \Z[\pi_0(X)]^*]$.\\
c) $\RA{i}(X) = 0$ for $i >\max(2,d+1)$.\\
d) $\RA{d+1}(X)$ is discrete.
\qed
\end{propose}

\subsection{Borel-Moore variants}

\begin{defn}\label{d11.5} For $X\in Sch(k)$, we denote by $\pi_0^c(X)$ the disjoint union of $\pi_0(Z_i)$
where $Z_i$ runs through the \emph{proper} connected components of $X$: this is the
\emph{scheme of proper constants}.
\end{defn}

\begin{propose}\label{c3.1c} Let $X\in Sch(k)$ of dimension $d$. Then:\\
a) $\LA{i}^c(X) = 0$ if   $i <0$.\\
b) $\LA{0}^c(X) = [\Z[\pi_0^c(X)]\to 0]$. In particular, $\LA{0}^c(X)=0$ if no connected
component is proper.\\ 
c) $\LA{i}^c(X) = 0$ for $i >\max(2,d+1)$.\\
d) $\LA{d+1}^c(X)$ is a group of multiplicative type.
\end{propose}

\begin{proof} If $X$ is proper, this is Proposition \ref{c3.1}. In
general, we may choose a compactification $\bar X$ of $X$; if $Z = \bar X -X$, with $\dim
Z<\dim X$, the claim follows inductively by the long exact sequence \eqref{loc}.
\end{proof}

We leave it to the reader to formulate the dual of this proposition for $\RPic^c(X)$.

\section{1-motivic homology and cohomology of curves}\label{scurves}

\subsection{``Chow-K\"unneth" decomposition for a curve} Note that for any curve $C$, the map $a_C$ is an isomorphism by Proposition \ref{cd2}. Moreover, since the category of 1-motives up to isogeny is of cohomological dimension 1 (see Prop.~\ref{iso1}), the complex $\LAlb^\Q(C)$ can be represented by a complex with zero differentials. Using Proposition \ref{c3.1} c), we then have:

\begin{cor} If $C$ is a curve then the motive $M(C)$ decomposes in $\DM_\gm^\eff\otimes\Q$ as
\[M (C) = M_0 (C) \oplus M_1 (C)\oplus M_2 (C)\] 
where $M_i (C)\df \Tot\LA{i}^\Q(C)[i]$.
\end{cor}

\subsection{$\LA{i}$ and $\RA{i}$ of curves} Here we shall complete the computation of
Proposition \ref{c3.1} in the case of a curve $C$.

Let $\tilde C$ denote the normalisation of
$C$. Let $\bar C$ be a smooth compactification of $\tilde C$ so
that $F = \bar C -\tilde C$ is a finite set of closed points.
Consider the following cartesian square
\[ \begin{CD}
\tilde S @>>> \tilde C\\
@V{}VV  @V{}VV  \\
S@>>> C
\end{CD}\]
where $S$ denote the singular locus. Let $\bar S$ denote $\tilde
S$ regarded in $\bar C$. Note that $S = \pi_0(S)$, $\tilde S =
\pi_0(\tilde S)$ and $\pi_0(\tilde S) \by{=} \pi_0(S)$ if $\tilde C\to C$ is
radicial, yielding $M (\tilde C)\by{=} M(C)$ in this case. In
general, we have the following.
\begin{thm}\label{lasc}
 Let $C$, $\tilde C$, $\bar C$, $S$, $\tilde S$, $\bar S$ and
$F$ as above. Then
$$\LA{i}(C) =
\begin{cases}
\relax [\Z[\pi_0(C)]\to 0]& \text{if $i = 0$}\\
\relax [\Div_{\bar S/S}^0(\bar C, F)\by{u}\Pic^0(\bar C, F)] & \text{if
$i= 1$}\\ 
\relax [0\to\NS_{\tilde C/k}^*] & \text{if $i= 2$}\\
0 & \text{otherwise}
\end{cases}$$
where: $\Div_{\bar S/S}^0(\bar C, F) =  \Div_{\tilde S/S}^0(\tilde C)$ here is the free group
of degree zero divisors generated by $\tilde S$ having trivial push-forward on $S$ and the map
$u$ is the canonical map (\cf \cite[Def. 2.2.1]{BSAP}); $\NS_{\tilde C/k}$ is the sheaf
associated to the free abelian group on the proper irreducible components of $\tilde C$.
\end{thm}

\begin{proof} We use the long exact sequence \eqref{exres}
\[\dots\to\LA{i}(\tilde S)\to \LA{i}(\tilde C)\oplus \LA{i}(S)\to \LA{i}(C)\to\LA{i-1}(\tilde
S)\to\dots\]

Since $S$ and $\tilde S$ are $0$-dimensional we have $\LA{i}(\tilde S)=
\LA{i}(S)=0$ for $i>0$, therefore
\[\LA{i}(C)= \LA{i}(\tilde C) \text{ for } i\ge 2\]
and by~\ref{HLAlb} we get the claimed vanishing and
description of $\LA{2}(C)$. For $i=0$ see Corollary~\ref{c3.1}.
If $i=1$ then $\LA{1}(C)$ is here represented as an element of $\Ext ([\Lambda \to 0] ,
\LA{1}(\tilde C))$ where $\Lambda \df \ker (\Z[\pi_0(\tilde S)]\allowbreak\to
\Z[\pi_0(\tilde C)]\oplus \Z[\pi_0(S)]$.
Recall, see~\ref{HLAlb}, that $\LA{1}(\tilde C)=
[0\to \cA_{\tilde C/k}^0]$ thus $\Ext (\Lambda , \LA{1}(\tilde
C))= \Hom_k (\Lambda, \cA_{\tilde C/k}^0)$ and
$$\LA{1}(C) =[\Lambda \by{u} \cA_{\tilde C/k}^0].$$
Now $\Lambda = \Div_{\bar S/S}^0(\bar C, F)$,
$\cA_{\tilde C/k}^0 =\Pic^0(\bar C, F)$ and the map $u$ is induced
by the following canonical map.

Consider $\varphi_{\tilde C} : \tilde C \to \Pic (\bar C, F)$
where $\varphi_{\tilde C}(P)\df (\cO_{\bar C}(P), 1)$ yielding
$\cA_{\tilde C/k}= \Pic (\bar C, F)$ and such that
$$0\to \Div_F^0(\bar C)^*\to \Pic^0(\bar C, F)\to \Pic^0(\bar C)\to 0.$$
Thus $\LA{1}(\tilde C)= [0\to \Pic^0(\bar C, F)]$. Note that
$\Z[\pi_{0}(\tilde S)]=\Div_{\tilde S}(\tilde C) = \Div_{\bar
S}(\bar C, F)$, the map $\Z[\pi_{0}(\tilde S)]\to
\Z[\pi_{0}(\tilde C)]$ is the degree map and the following map
$\Z[\pi_{0}(\tilde S)]\to \Z[\pi_{0}(S)]$ is the proper
push-forward of Weil divisors, \ie $\Lambda = \Div_{\bar
S/S}^0(\bar C, F)$. The map $\varphi_{\tilde C}$ then induces the
mapping $u\in \Hom_k (\Lambda, \Pic^0(\bar C, F))$ which also is
the canonical lifting of the universal map $D\mapsto\cO_{\bar C}(D)$ as the
support of $D$ is disjoint from $F$ (\cf \cite[Lemma 3.1.3]{BSAP}).
\end{proof}

\begin{remark} We remark that $\LA{1}(C)$ coincides with the homological
Albanese 1-motive $\Alb^-(C)$ ($=\Pic^-(C)$ for curves, see
\cite{BSAP}). The $\LA{i}(C)$ also coincide with Lichtenbaum-Deligne motivic
homology $h_i (C)$ of the curve $C$, \cf \cite{LI}.
\end{remark}

Note that $\RA{1}(C)= \LA{1}(C)^*$ is Deligne's motivic cohomology
$H^1_m(C)(1)$ of the singular curve $C$ by \cite[Prop.
3.1.2]{BSAP}. Hence:

\begin{cor}\label{rasc}
 Let $C$ be a curve, $C'$ its seminormalisation, $\bar C'$  a compactification of $C'$, and $F = \bar C' - C'$. Let further $\tilde C$ denote the normalisation of $C$. Then
$$\RA{i}(C) =
\begin{cases}
\relax [0\to \G_m[\pi_0(C)]]& \text{if $i = 0$}\\
\relax [\Div_{F}^0(\bar C')\to\Pic^0(\bar C')] & \text{if $i= 1$}\\
\relax [\NS (\tilde C)\to 0] & \text{if $i= 2$}\\
0 & \text{otherwise}
\end{cases}$$
where $\NS (\tilde C)= \Z [\pi_0^c(\tilde C)]$ and $\pi_0^c(\tilde
C)$ is the scheme of proper constants.
\end{cor}

\subsection{Borel-Moore variants}

\begin{thm} \label{bmlac}
Let $C$ be a smooth curve, $\bar C$ a smooth compactification of
$C$ and $F =\bar C - C$ the finite set of closed points at
infinity. Then
$$\LA{i}^c(C) =
\begin{cases}
\relax [\Z[\pi_0^c(C)]\to 0]& \text{if $i = 0$}\\
\relax [\Div_F^0(\bar C)\to\Pic^0(\bar C)] & \text{if $i= 1$}\\
\relax [0\to\NS_{\bar C/k}^*] & \text{if $i= 2$}\\
0 & \text{otherwise}
\end{cases}$$
where $\NS (\bar C)=\Z [\pi_0 (\bar C)]$ and $\pi_0^c(C)$ is the
scheme of proper constants.
\end{thm}
\begin{proof} It follows from the distinguished triangle
\[\begin{CD}
\LAlb (F)  @>>> \LAlb (\bar C)\\
{\scriptstyle +1}\nwarrow &&\swarrow\\
&\LAlb^c(C)&
\end{CD}\]
and Corollary~\ref{HLAlb}, yielding the claimed description:
$\LA{0}^c(C)= \coker (\LA{0}(F) \to\LA{0}(\bar C))$ moreover we
have
$$[\Div_F^0(\bar C)\to 0]= \ker (\LA{0}(F) \to\LA{0}(\bar C))$$ and the
following extension
$$0\to \LA{1}(\bar C)\to \LA{1}^c(C)\to
[\Div_F^0(\bar C)\to0]\to 0.$$ Finally,  $\LA{i}(\bar C)=
\LA{i}^c(C)$ for $i\geq 2$.
\end{proof}

\begin{cor} Let $C$ be a smooth curve, $\bar C$ a smooth compactification of
$C$ and $F =\bar C - C$ the finite set of closed points at
infinity. Then
$$\RA{i}^c(C) =
\begin{cases}
\relax [0\to\G_m[\pi_0^c(C)]]& \text{if $i = 0$}\\
\relax [0\to\Pic^0(\bar C, F)] & \text{if $i= 1$}\\
\relax [\NS (\bar C)\to 0] & \text{if $i= 2$}\\
0 & \text{otherwise}
\end{cases}$$
where $\NS (\bar C)=\Z [\pi_0 (\bar C)]$ and $\pi_0^c(C)$ is the
scheme of proper constants.
\end{cor}

Here we have that $\RA{1}^c(C)=\RA{1}^*(C)$ is also the Albanese
variety of the smooth curve.

Note that $\LA{1}^c(C)$ (= $\Pic^+ (C) = \Alb^+(C)$ for curves,
see \cite{BSAP}) coincide with Deligne's motivic $H^1_m(C)(1)$ of
the smooth curve $C$. This is due to the Poincar\'e duality
isomorphism $M^c (C)= M (C)^*(1)[2]$.

\section{Comparison with $\Pic^+,\Pic^-,\Alb^+$ and $\Alb^-$}\label{12}

In this section, we want to study $\LA{1}(X)$ and its variants in more detail. In particular,
we show in Proposition \ref{p11.3a} c) that it is always a Deligne $1$-motive, and show in
Corollaries \ref{c12.2.2} and \ref{c12.3} that, if $X$ is normal or proper, it is canonically isomorphic to the
$1$-motive $\Alb^-(X)$ of \cite{BSAP}. Precise descriptions of $\LA{1}(X)$ are given in
Proposition \ref{pl1} and Corollary \ref{c12.2.1}.  

We also describe $\LA{1}^c(X)$ in Proposition \ref{c3.1d}; more precisely, we prove in Theorem
\ref{t12.9} that its dual $\RA{1}^c(X)$ is canonically isomorphic to $\Pic^0(\bar X,Z)/\cU$, where $\bar
X$ is a compactification of $X$ with complement $Z$ and $\cU$ is the unipotent radical of the
commutative algebraic group $\Pic^0(\bar X,Z)$. 
Finally, we prove in Theorem \ref{*=-}
that $\LA{1}^*(X)$ is abstractly isomorphic to the $1$-motive $\Alb^+(X)$ of \cite{BSAP}.

We start with some comparison results between $\eh$ and \'etale cohomology for non smooth
schemes. 

Let $\epsilon:Sch_\eh\to Sch_\et$ be the obvious morphism of sites. If
$\cF$ is an \'etale sheaf on $Sch$, we denote by $\cF_\eh$ its $\eh$ sheafification (that is,
$\cF_\eh=\epsilon_*\epsilon^*\cF$). We shall abbreviate $H^*_\eh(X,\cF_\eh)$ to
$H^*_\eh(X,\cF)$. 

\subsection{Torsion sheaves} The first basic result is a variant of \cite[Cor.
7.8 and Th. 10.7]{suvo}: it follows from Proposition \ref{peh} and Examples
\ref{ex11.1} via Lemma \ref{lbl} b).

\begin{propose}\label{p11.2} Let $C$ be a bounded below complex of torsion
sheaves on $(\Spec k)_\et$. Then, for any $X\in Sch$ and any $n\in\Z$,
$H^n_\et(X,C)\iso H^n_\eh(X,C)$.\qed
\end{propose}

(See \cite[Th. 3.6]{geisser} for a different proof.)

\subsection{Discrete sheaves}

\begin{lemma}\label{l11.2} a) If $\cF$ is discrete, then $\cF\iso \cF_\eh$. More precisely, for
any $X\in Sch$, $\cF(\pi_0(X))\iso \cF(X)\iso \cF_\eh(X)$.\\
b) If $f:Y\to X$ is surjective with geometrically connected fibres, then $\cF_\eh\iso f_*
f^*\cF_\eh$.
\end{lemma}

\begin{proof} a) We may assume $X$ reduced. Clearly it suffices to prove that $\cF(\pi_0(X))\iso
\cF_\eh(X)$ for any $X\in Sch$.  In the situation of \S \ref{blowups}, we have a commutative
diagram of exact sequences
\[\begin{CD}
0@>>> \cF_\eh(X)@>>> \cF_\eh(\tilde X)\oplus \cF_\eh(Z)@>>> \cF_\eh(\tilde Z)\\
&&@AAA @AAA @AAA\\
0@>>> \cF(\pi_0(X))@>>> \cF(\pi_0(\tilde X))\oplus \cF(\pi_0(Z))@>>> \cF(\pi_0(\tilde Z)).
\end{CD}\]

The proof then goes exactly as the one of Proposition \ref{p11.2}. b) follows from a).
\end{proof}

It is well-known that $H^1_\et(X,\cF)=0$ for any geometrically unibranch scheme $X\in Sch$ if
$\cF$ is constant and torsion-free (\cf \cite[IX, Prop. 3.6 (ii)]{sga4}). The following lemma
shows that this is also true for the $\eh$ topology, at least if $X$ is normal.

\begin{lemma}[compare \protect{\cite[Ex. 12.31 and 12.32]{VL}}]\label{l11.3} Let $\cF$ be a constant torsion-free sheaf on $Sch(k)$.\\
a) For any $X\in Sch$, $H^1_\eh(X,\cF)$ is torsion-free. It is finitely generated if $\cF$ is a
lattice.\\ 
b) Let $f:\tilde X\to X$ be a surjective morphism. Then $H^1_\eh(X,\cF)\to
H^1_\eh(\tilde X,\cF)$ is injective in the following cases:
\begin{thlist}
\item The geometric fibres of $f$ are connected.
\item $f$ is finite and flat.
\end{thlist} 
c) If $X$ is normal, $H^1_\eh(X,\cF)=0$.
\end{lemma}

\begin{proof} a) The first assertion follows immediately from Lemma \ref{l11.2} (consider the
exact sequence of multiplication by $n$ on $\cF$). The second assertion follows by blow-up induction from the fact that $H^1_\eh(X,\cF)=0$ if $X$ is
smooth, by Proposition \ref{peh}. 

b) In the first case, the Leray spectral sequence yields an injection
\[H^1_\eh(X,f_*\cF)\into H^1_\eh(\tilde X,\cF)\]
and $f_*\cF=\cF$ by Lemma \ref{l11.2} b). In the second case, the theory of trace \cite[XVII, Th.
6.2.3]{sga4} provides $\cF$, hence $\cF_\eh$, with a morphism
$Tr_f:f_*\cF\to \cF$ whose composition with the natural morphism is (on each connected component
of $X$) multiplication by some nonzero integer. This shows that the kernel of $H^1_\eh(X,\cF)\to
H^1_\eh(X,f_*\cF)$ is torsion, hence $0$ by a).

c) follows from b) with $\tilde X$ a desingularisation of $X$: by Proposition \ref{p11.2},
$H^1_\et(\tilde X,\cF)\iso H^1_\eh(\tilde X,\cF)$ and it is well-known that the first group is
$0$; on the other hand, the fibres of $f$ are geometrically connected by Zariski's main theorem.
\end{proof}

The following is a version of \cite[Lemma 5.6]{weibel}:

\begin{lemma} \label{lweibel} Let $f:\tilde X\to X$ be a finite birational morphism, $i:Z\into
X$ a closed subset and $\tilde Z=p^{-1}(Z)$. Then, for any discrete sheaf, we have a long 
exact sequence:
\begin{multline*}
\dots\to H^i_\et(X,\cF)\to H^i_\et(\tilde X,\cF)\oplus H^i_\et(Z,\cF)\\
\to H^i_\et(\tilde Z,\cF)\to H^{i+1}_\et(X,\cF)\to \dots
\end{multline*}
\end{lemma}

\begin{proof} Let $g:\tilde Z\to Z$ be the induced map. Then $f_*,i_*$ and $g_*$ are exact for
the \'etale topology. Thus it suffices to show that the sequence of sheaves
\[0\to \cF\to f_*f^*\cF\oplus i_*i^*\cF\to (ig)_*(ig)^*\cF\to 0\]
is exact. The assertion is local for the \'etale topology, hence we may assume that $X$ is
strictly local. Then $Z,\tilde X$ and $\tilde Z$ are strictly local as well, hence connected,
thus the statement is obvious.
\end{proof}

We can now prove:

\begin{propose}\label{l11.2.6} For any $X\in Sch(k)$ and any discrete sheaf $\cF$, the map
$H^1_\et(X,\cF)\to H^1_\eh(X,\cF)$ is an isomorphism.
\end{propose}

\begin{proof} Let $f:\tilde X\to X$ be the normalisation of $X$, and take for $Z$ the
non-normal locus of $X$ in Lemma \ref{lweibel}. The result now follows from comparing the exact
sequence of this lemma with the one for $\eh$ topology, and using Lemma \ref{l11.3} c).
\end{proof}

\begin{cor} The exact sequence of Lemma \ref{lweibel} holds up to $i=1$ for a general abstract
blow-up.\qed
\end{cor}

\subsection{Strictly reduced schemes} 

If $G$ is a commutative $k$-group scheme, the associated presheaf $\uG$ is an \'etale sheaf on
reduced $k$-schemes of finite type. However, $\uG(X)\to \uG_\eh(X)$ is not an isomorphism in
general if $X$ is not smooth. Nevertheless we have some nice results in Lemma \ref{leh} below.

\begin{defn}\label{dsred} A $k$-scheme of finite type $X$ is \emph{strictly reduced} (a
recursive definition) if it is reduced and
\begin{thlist}
\item If $X$ is irreducible: $X_{sing}$, considered with its reduced structure, is strictly reduced.
\item If $Z_1,\dots, Z_n$ are the irreducible components of $X$: all $Z_i$ are strictly
reduced and the scheme-theoretic intersection $Z_i\cap Z_j$ is reduced for any $i\ne j$.  
\end{thlist}
\end{defn}

\begin{examples} 1) If $\dim X=0$, $X$ is strictly reduced.\\
2) The union of a line and a tangent parabola is not strictly reduced.\\
3) If $X$ is normal and of dimension $\le 2$, it is strictly reduced.\\
4) M. Ojanguren provided the following example of a normal $3$-fold which is not strictly reduced: take the affine hypersurface with equation $uv = x^2(y^2-x)^2$.
\end{examples}

\begin{lemma}\label{l11.2.1} Let $G$ be an affine group scheme and $f:Y\to X$ a proper surjective map with geometrically connected fibres. Then $\uG(X)\allowbreak\iso \uG(Y)$, and $H^0_\tau(X,\uG)\iso H^0_\tau(Y,\uG)$ for any Grothendieck topology $\tau$ stronger than the Zariski topology.
\end{lemma}

\begin{proof} The first statement is clear, and the second folllows because the hypothesis on $f$ is stable under any base change.
\end{proof}

\begin{lemma}\label{leh} a) If $X$ is reduced, then the map 
\begin{equation}\label{eq12.2}
\uG(X)\to \uG_\eh(X)
\end{equation}
 is injective for any semi-abelian $k$-scheme $G$.\\
b) If $X$ is strictly reduced, \eqref{eq12.2} is an isomorphism.\\
c) If $X$ is proper and $G$ is a torus, the maps $\uG(\pi_0(X))\to\uG_\et(X)\to \uG_\eh(X)$ are isomorphisms. If moreover $X$ is reduced, \eqref{eq12.2} is an isomorphism.
\end{lemma}

\begin{proof} a) Let $Z_i$ be the irreducible components of $X$, and for each $i$ let
$p_i:\tilde Z_i\to Z_i$ be a resolution of singularities. We have a commutative diagram
\[\begin{CD}
\uG_\eh(X)@>>> \bigoplus \uG_\eh(Z_i)@>(p_i^*)>>\bigoplus \uG_\eh(\tilde Z_i)\\
@AAA @AAA @AAA\\
\uG(X)@>>> \bigoplus \uG(Z_i)@>(p_i^*)>>\bigoplus \uG(\tilde Z_i).
\end{CD}\]

The bottom horizontal maps are injective; the right vertical map is an isomorphism by
Proposition \ref{peh}. The claim follows.

b) We argue by induction on $d=\dim X$. If $d=0$ this is trivial. If $d>0$, we first assume $X$
irreducible. Let
$Z$ be its singular locus, and choose a desingularisation $p:\tilde X\to X$ with $p$
proper surjective, $\tilde X$ smooth, $\tilde Z = p^{-1}(Z)$ a divisor with
normal crossings (in particular reduced) and $p_{|\tilde X-\tilde Z}$ an isomorphism. We now
have a commutative diagram
\[\begin{CD}
0@>>> \uG_\eh(X)@>>> \uG_\eh(Z)\oplus \uG_\eh(\tilde X)@>>>\bigoplus \uG_\eh(\tilde Z)\\
&& @AAA @AAA @AAA\\
0@>>> \uG(X)@>>> \uG(Z)\oplus \uG(\tilde X)@>>>\bigoplus \uG(\tilde Z)
\end{CD}\]
where the lower sequence is exact, the middle vertical map is bijective by induction on $d$ and
the smooth case (Proposition \ref{peh}) and the right vertical map is injective by a). It
follows that the left vertical map is surjective. 

In general, write $Z_1,\dots,Z_n$ for the irreducible components of $X$: by assumption, the
two-fold intersections
$Z_{ij}$ are reduced. The commutative diagram
\[\begin{CD}
0@>>> \uG_\eh(X)@>>> \bigoplus \uG_\eh(Z_i)@>>>\bigoplus \uG_\eh(Z_{ij})\\
&& @AAA @AAA @AAA\\
0@>>> \uG(X)@>>> \bigoplus \uG(Z_i)@>>>\bigoplus \uG(Z_{ij})
\end{CD}\]
then has the same formal properties as the previous one, and we conclude.

For c), same proof as for Lemma \ref{l11.2} a). (The second statement of c) is true because $\uG(\pi_0(X))\iso \uG(X)$ if $X$ is proper and reduced.)
\end{proof}

\subsection{Normal schemes} The main result of this subsection is:

\begin{thm}\label{tnormal} Let $X$ be normal. Then, for any $\cF\in \Shv_1$, the map $\cF(X)\to \cF_\eh(X)$ is bijective and the map $H^1_\et(X,\cF)\to H^1_\eh(X,\cF_\eh)$ is injective (with torsion-free cokernel by Proposition \ref{p11.2}).
\end{thm}

\begin{proof} In several steps:

{\it Step 1.} The first result implies the second for a given sheaf $\cF$: let $\epsilon :Sch(k)_\eh\to Sch(k)_\et$ be the projection morphism. The associated Leray spectral sequence gives an injection
\[H^1_\et(X,\epsilon_*\cF_\eh)\into H^1_\eh(X,\cF_\eh).\]

But any scheme \'etale over $X$ is normal \cite[Exp. I, Cor. 9.10]{SGA1}, therefore $\cF\to \epsilon_*\cF_\eh$ is an isomorphism over the small \'etale site.

{\it Step 2.} Let $0\to \cF'\to \cF\to \cF''\to 0$ be a short exact sequence in $\Shv_1$. If the theorem is true for $\cF'$ and $\cF''$, it is true for $\cF$. This follows readily from {\it Step 1} and a diagram chase.

{\it Step 3.} Given the structure of $1$-motivic sheaves, {\it Step 1 - 2} reduce us to prove that $\cF(X)\iso \cF_\eh(X)$ separately when $\cF$ is discrete, a torus or an abelian variety. The discrete case follows from Lemma \ref{l11.2} a).

{\it Step 4.} If $G$ is a torus, let $\pi:\tilde X\to X$ be a desingularisation of $X$. We have a commutative diagram
\[\begin{CD}
\uG_\eh(X)@>>>  \uG_\eh(\tilde X)\\
@AAA @AAA\\
\uG(X)@>>> \uG(\tilde X).
\end{CD}\]

Here the right vertical map is an isomorphism because $\tilde X$ is smooth and the two
horizontal maps are also  isomorphisms by Lemma \ref{l11.2.1} applied to $\pi$ (Zariski's main
theorem). The result follows.

{\it Step 5.} Let finally $G$ be an abelian variety. This time, it is not true in general that $\uG(X)\iso \uG(\tilde X)$ for a smooth desingularisation $\tilde X$ of $X$. However, we get the result from Proposition \ref{pdescent} b) and the following general lemma.
\end{proof}

\begin{lemma}\label{ldescent} Let $X\in Sch(k)$ be normal, $p:X_0\to X$ a proper surjective map such that the restriction of $p$ to a suitable connected component $X'_0$ of $X_0$ is birational. Let $X_1\begin{smallmatrix}p_0\\\rightrightarrows\\p_1\end{smallmatrix} X_0$ be two morphisms such that $pp_0 = pp_1$ and that the induced map $\Psi:X_1\to X_0\times_X X_0$ is proper surjective. Let $Y\in Sch(k)$ and let $f:X_0\to Y$ be such that $fp_0=fp_1$.
\[\xymatrix{
X_1
\ar@<-.7ex>[d]_{p_0} \ar@<.7ex>[d]^{p_1}\\
X_0
\ar[d]_p\ar[dr]^f\\
X&Y
}\]
Then there exists a unique morphism $\bar f:X\to Y$ such that $f=\bar f p$.
\end{lemma}

\begin{proof} We may assume $X$ connected. Since $\Psi$ is proper surjective, the
hypothesis is true by replacing $X_1$ by $X_0\times_X X_0$, which we shall assume
henceforth. Let $x\in X$ and $K=k(x)$. Base-changing by the morphism $\Spec K\to
X$, we find (by faithful flatness) that $f$ is constant on $p^{-1}(x)$. Since $p$
is surjective, this defines $\bar f$ as a set-theoretic map, and this map is
continuous for the Zariski topology because $p$ is also proper.

It remains to show that $\bar f$ is a map of locally ringed spaces. Let $x\in X$,
$y=\bar f(x)$ and $x'\in p^{-1}(x)\cap X'_0$. Then $f^\sharp(\cO_{Y,y})\subseteq
\cO_{X'_0,x'}$. Note that $X$ and $X'_0$ have the same function field $L$, and
$\cO_{X,x}\subseteq \cO_{X'_0,x'}\subseteq L$. Now, since $X$ is normal,
$\cO_{X,x}$ is the intersection of the valuation rings containing it. 

Let $\cO$ be such a valuation ring, so that $x$ is the centre of $\cO$ on $X$. By
the valuative criterion of properness, we may find $x'\in p^{-1}(x)\cap X'_0$ such
that $\cO_{X_0',x'}\subseteq \cO$. This shows that
\[\cO_{X,x} =\bigcap_{x'\in p^{-1}(x)\cap X'_0}\cO_{X_0',x'}\]
and therefore that $f^\sharp(\cO_{Y,y})\subseteq \cO_{X,x}$. Moreover, the
corresponding map $\bar f^\sharp:\cO_{Y,y}\to \cO_{X,x}$ is local since $f^\sharp$
is. 

(Alternatively, observe that $f$ and its topological factorisation induce a map
\[f^\#: \cO_Y\to f_*\cO_{X_0}\simeq \bar f_* p_*\cO_{X_0}= \bar f_*\cO_X.)\]
\end{proof}

\subsection{Some representability results}

\begin{propose}\label{p11.3} Let $\pi^X$ be the structural morphism of $X$ and $(\pi_*^X)^\eh$
the induced direct image  morphism on the $\eh$ sites. For any $\cF\in \HI_\et$, let
us denote the restriction of $R^q(\pi_*^X)^\eh \cF_\eh$ to $Sm$ by
$\uR^q\pi_*^X\cF$ (in other words, $R^q(\pi_*^X)^\eh \cF_\eh$ is the sheaf on
$Sm(k)_\et$ associated to the presheaf $U\mapsto H^q_\eh(X\times U,\cF_\eh)$): it is an object
of $\HI_\et$. Then\\   
a) For any lattice $L$, $\uR^q\pi_*^X L$ is a
ind-discrete sheaf for all $q\ge 0$; it is a lattice for $q=0,1$.\\ 
b) For any torus $T$, $\uR^q\pi_*^X \uT$ is $1$-motivic for $q= 0,1$.
\end{propose}

\begin{proof} We apply Lemma \ref{lbl} a) in the following situation: $\cA=\HI_\et$,
$\cB=\Shv_0$, $H^i(X)=R^i\pi_*^XL$ in case a), $\cB=\Shv_1$, $H^i(X)=R^i\pi_*^X\uT$ in case b).
The smooth case is trivial in a)  and the lattice
assertions follow from lemmas \ref{l11.2} and \ref{l11.3} a).  In b), the smooth case follows
from Proposition \ref{p3.3.1}.
\end{proof}

\subsection{$\LA{1}(X)$ and the Albanese schemes} We now compute the $1$-motive $\LA{1}(X)=[L_1\to G_1]$ in important special cases This is done
in the following three propositions; in particular, we shall show that it always ``is" a Deligne $1$-motive. Note that, by definition of a $1$-motive with cotorsion, the pair $(L_1,G_1)$ is determined only up to a \qi: the last sentence means that we may choose this pair such that $G_1$ is connected.

\begin{propose} \label{c3.1bis} Let $X\in Sch(k)$. Then\\
a) $\cH_i(\LAlb(X))=0$ for $i<0$.\\
b) Let $\cF_X=\cH_0(\Tot\LAlb(X))$. Then $\cF_X$ corepresents the functor
\begin{align*}
\Shv_1&\to Ab\\
\cF&\mapsto \cF_\eh(X) \text{ (see Def. \ref{deh})}
\end{align*}
via the composition
\[\alpha^* M(X)\to \Tot\LAlb(X)\to \cF_X[0].\]
Moreover, we have an exact sequence, for any representative $[L_1\by{u_1}G_1]$ of $\LA{1}(X)$ :
\begin{equation}\label{eq11.1}
L_1\by{u_1} G_1\to \cF_X\to \Z\pi_0(X)\to 0.
\end{equation}
c) Let $\cA_{X/k}^\eh:=\Omega(\cF_X)$ (\cf Proposition \ref{pladj}). Then $\cA_{X/k}^\eh$ corepresents the functor \index{$\cA_{X/k}^\eh$}
\begin{align*}
{}^t\AbS&\to Ab\\
G&\mapsto \uG_\eh(X).
\end{align*}
Moreover we have an epimorphism
\begin{equation}\label{compalb}
 \cA_{X/k}^\eh\onto \cA_{X_\red/k}.
\end{equation}
d) If $X_\red$ is strictly reduced (Def. \ref{dsred}) or normal, \eqref{compalb} is an isomorphism.
\end{propose}

\begin{proof} a) is proven as in Proposition \ref{c3.1} by blow-up induction (reduction to the
smooth case). If $\cF\in \Shv_1$, we have 
\[\Hom_{\DM_{-,\et}^\eff}(\alpha^*M(X),\cF)=\cF_\eh(X)\] 
by Propositions \ref{ptransf}, \ref{peh} and  \ref{pD.1.4}. The latter group coincides with
$\Hom_{\Shv_1}(\cF_X,\cF)$ by \eqref{lalbuniv} and a), hence b); the exact sequence follows from Proposition \ref{p3.10}.  The sheaf $\cA_{X/k}^\eh$ clearly corepresents the said functor; the map then comes from the obvious
natural transformation in $G$: $G(X_\red)\to G_\eh(X)$ and its surjectivity follows from Lemma
\ref{leh} a), hence c). d) follows from Lemma \ref{leh} b), Theorem \ref{tnormal} and the universal property of $\cA_{X/k}$.
\end{proof}

\begin{remark}\label{r11.1} One could christen
$\cF_X$ and
$\cA_{X/k}^\eh$ the
\emph{universal $1$-motivic sheaf} and the
\emph{$\eh$-Albanese scheme} of $X$.
\end{remark}

\begin{propose}\label{p11.3a}
a) The sheaves $\cF_X$ and $\cA_{X/k}^\eh$ have $\pi_0$ equal to $\Z[\pi_0(X)]$; in particular,
$\cA_{X/k}^\eh\in \AbS$.\\
b) In \eqref{eq11.1}, the composition $L_1\by{u_1}G_1\to \pi_0(G_1)$ is surjective.\\
c) One may choose $\LA{1}(X)\simeq [L_1\to G_1]$ with $G_1$ connected (in other words, $\LA{1}(X)$ is a Deligne $1$-motive).
\end{propose}

\begin{proof}
In a), it suffices to prove the first assertion for $\cF_X$: then it follows from
its universal property and Lemma \ref{l11.2} a). The second assertion of a) is obvious. 

b) Let $0\to L'_1\to G'_1\to \cF_X\to \Z[\pi_0(X)]\to 0$ be the normalised presentation of $\cF_X$ given by Proposition \ref{p3.1}. We have a commutative diagram
\[\begin{CD}
0@>>> L'_1@>>> G'_1@>>> \cF_X@>>> \Z[\pi_0(X)]@>>> 0\\
&&@VVV @VVV ||& &||&\\
0@>>> \overline{u_1(L_1)}@>\bar u_1>> \bar G_1@>>> \cF_X@>>> \Z[\pi_0(X)]@>>> 0\\
&&@AAA @AAA ||& &||&\\
&&L_1@>u_1>> G_1@>>> \cF_X@>>> \Z[\pi_0(X)]@>>> 0
\end{CD}\]
with $\overline{u_1(L_1)}=u_1(L_1)/F$ and $\bar G_1=G_1/F$, where $F$ is the torsion subgroup of $u_1(L_1)$. Indeed, $\Ext(G'_1,\overline{u_1(L_1)})=0$ so we get the downwards vertical maps as in the proof of Proposition \ref{p3.1}. By uniqueness of the normalised presentation, $G'_1$ maps onto $\bar G_1^0$. A diagram chase then shows that the composition
\[\overline{u_1(L_1)}\by{\bar u_1} \bar G_1\to \pi_0(\bar G_1)\]
is onto, and another diagram chase shows the same for $u_1$.

c) The pull-back diagram
\[\begin{CD}
L_1^0@>>> G_1^0\\
@VVV @VVV\\
L_1 @>>> G_1
\end{CD}\]
is a quasi-isomorphism in ${}_t\M^\eff$, thanks to b).
\end{proof}

\begin{propose}\label{pl1}   Let $[L_1\by{u_1} G_1]$ be the Deligne $1$-motive that lies in the \qi class of $\LA{1}(X)$, thanks to Proposition \ref{p11.3a} c).\\
a) We have an isomorphism 
\[L_1\iso\shom(\uR^1\pi_*\Z,\Z)\]
(\cf Proposition \ref{p11.3}).\\ 
b) We have a canonical isomorphism
\[G_1/(L_1)_\Zar\iso (\cA_{X/k}^\eh)^0\]
where $(L_1)_\Zar$ is the Zariski closure of the image of $L_1$ in $G_1$ and
$\cA_{X/k}^\eh$ was defined in Proposition \ref{c3.1bis} ($(\cA_{X/k}^\eh)^0$ corepresents the functor $\SAb\ni G\mapsto G_\eh(X)$).
\end{propose}

\begin{proof} For the computations, we may assume $k$ algebraically closed.

a) Let $L$ be a lattice. We compute:
\[
H^1_\eh(X,L) = \Hom_{\DM_{\gm,\et}^\eff}(M_\et(X),L[1]) =
\Hom_{D^b(\M)}(\LAlb(X),L[1]).
\]

From the spectral sequence \eqref{eqss}, we get an exact sequence
\begin{multline*}
0\to \Ext^1(\LA{0}(X),L)\to \Hom_{D^b(\M)}(\LAlb(X),L[1])\\
\to\Hom(\LA{1}(X),L)\to
\Ext^2(\LA{0}(X),L).
\end{multline*}

Since the two Ext are $0$, we get an isomorphism
\[\Hom_{D^b(\M)}(\LAlb(X),L[1])\iso\Hom(\LA{1}(X),L).\]

Since $[L_1\to G_1]$ is a Deligne $1$-motive, the last group is isomorphic to $\Hom(L_1,L)$. This gives a), since we obviously have  
$H^1_\eh(X,L)=H^0_\et(k,\uR^1\pi_* L)=\uR^1\pi_*\Z\otimes L$ by
Proposition \ref{p11.3} a).

b) This follows directly from the definition of $\cA_{X/k}^\eh$.
\end{proof}

\begin{cor}\label{c12.2.1} Let $\LA{1}(X)=[L_1\to G_1]$, as a Deligne $1$-motive.\\
a) If $X$ is proper, then $G_1$ is an abelian variety.\\
b)  If $X$ is normal, then $\LA{1}(X)=[0\to \cA_{X/k}^0]$.\\
c)  If $X$ is normal and proper then $\RA{1}(X)=[0\to \Pic^0_{X/k}]$ is an abelian
variety with dual the Serre Albanese $\LA{1}(X)\allowbreak=[0\to \cA_{X/k}^0]$.
\end{cor}

\begin{proof} a) is seen easily by blow-up induction, by reducing to the smooth
projective case (Corollary \ref{HLAlb}). b) follows from Proposition \ref{pl1} a),
b), Lemma \ref{l11.3} c) and Proposition \ref{c3.1bis} d). c) follows immediately
from a) and b).
\end{proof}

\subsection{$\LA{1}(X)$ and $\Alb^-(X)$ for $X$ normal}

In this subsection, we prove that these two $1$-motives are isomorphic  in \ref{c12.2.2}. We begin with a slight improvement of Theorem \ref{tnormal} in the case of semi-abelian schemes:

\begin{lemma} \label{simpehalb} Let $X$ be normal, and let $\Xs$ be a smooth hyperenvelope (\cf Lemma \ref{ldescent}). Then we have
 $$\cA_{X/k} = \coker (\cA_{X_1}\longby{(p_0)_*-(p_1)_*}\cA_{X_0} )$$ and
 $$ (\cA_{X/k})^0 = \coker (\cA_{X_1}^0\longby{(p_0)_*-(p_1)_*}\cA_{X_0}^0 ).$$
\end{lemma}

\begin{proof} The first isomorphism follows from Lemma \ref{ldescent} applied with $Y$ running through torsors under semi-abelian varieties. To deduce the second isomorphism, consider the short  exact sequence of complexes
\[0\to \cA_{\Xs/k}^0\to \cA_{\Xs/k}\to \Z[\pi_0(\Xs)]\to 0\] 
and the resulting long exact sequence
\begin{equation}\label{eq13.7}
H_1(\Z[\pi_0(\Xs)])\to H_0(\cA_{\Xs/k}^0)\to H_0(\cA_{\Xs/k})\to H_0(\Z[\pi_0(\Xs)])\to 0.
\end{equation}

For any $i\ge 0$, $\Z[\pi_0(X_i)]$ is $\Z$-dual to the Galois module
$E_1^{i,0}$,
where $E_1^{p,q}= H^q_\et(X_p\times_k \bar k,\Z)$ is the $E_1$-term associated to the simplicial spectral sequence for $\Xs\times_k \bar k$. Since $H^1_\et(X_p\times_k \bar k,\Z)=0$ for all $p\ge 0$, we get
\[H_i(\Z[\pi_0(\Xs)])\simeq (H^i_\et(\Xs\times_k \bar k,\Z))^\vee\text{ for } i=0,1.\]

By Proposition \ref{peh}, these \'etale cohomology groups may be replaced by $\eh$ cohomology groups. By Proposition \ref{pdescent}, we then have
\[H^i_\et(\Xs\times_k \bar k,\Z)\simeq H^i_\eh(X\times_k \bar k,\Z).\]

Now, by Lemma \ref{l11.2} a), $H^0_\eh(X\times_k \bar k,\Z)$ is $\Z$-dual to $\Z[\pi_0(X)]$, and by Lemma \ref{l11.3} c), $H^1_\eh(X\times_k \bar k,\Z)=0$ because $X$ is normal. Hence \eqref{eq13.7} yields a short exact sequence
\[0\to  H_0(\cA_{\Xs/k}^0)\to \cA_{X/k}\to \Z[\pi_0(X)]\to 0\]
which identifies $H_0(\cA_{\Xs/k}^0)$ with $\cA_{X/k}^0$.
\end{proof}

\begin{propose} \label{c12.2.2} If $X$ is normal, $\RA{1}(X)$ and
$\LA{1}(X)$ are isomorphic,  respectively, to the $1$-motives $\Pic^+(X)$
and $\Alb^- (X)$ defined in \cite[Ch. 4-5]{BSAP}.
\end{propose}

\begin{proof} Let $\bar X$ be a normal compactification of $X$; choose a
smooth hyperenvelope $X_{\d}$ of $X$ along with $\bar X_{\d}$ a smooth
compactification with normal crossing boundary $Y_{\d}$ such that  $\bar X_{\d}\to
\bar X$ is an hyperenvelope. Now we have, in the notation of \cite[4.2]{BSAP}, a
commutative diagram with exact rows
\[
\begin{CD}
0@>{}>>\Pic^0_{\bar X/k} @>{}>>\Pic^0_{\bar X_0/k}@>{}>>  \Pic^0_{\bar
X_1/k}\\
& & @A{}AA @A{}AA @A{}AA  \\
0@>{}>>\Div_{Y_{\d}}^0(\bar X_{\d})@>{}>>\Div_{Y_{0}}^0(\bar X_{0})@>{}>>
\Div_{Y_{1}}^0(\bar X_{1})
\end{CD}
\]
where $\Pic^+(X) =[\Div_{Y_{\d}}^0(\bar X_{\d})\to \Pic^0_{\bar X/k}]$ since $\bar X$ is normal. Taking Cartier duals we get an exact sequence of $1$-motives
\[
[0\to \cA_{X_1/k}^0]\longby{(p_0)*-(p_1)_*}[0\to \cA_{X_0/k}^0]\to
\Alb^-(X)\to 0.
\]

Thus $\Alb^-(X)=[0\to \cA_{X/k}^0]$ by Lemma \ref{simpehalb}. We conclude by Corollary \ref{c12.2.1} b) since $X$ is normal and $\LA{1}(X)= [0\to \cA_{X/k}^0]$.
\end{proof}



\begin{remarks} 1) Note that, while $\LA{0}(X)$ and $\LA{1}(X)$ are Deligne $1$-motives, the same is not true of $\LA{2}(X)$ in general, already for $X$ smooth projective (see Corollary \ref{HLAlb}).\\
2) One could make use of Proposition~\ref{cdhex} to compute
$\LA{i} (X)$ for singular $X$ and $i>1$. However, $H^i_\eh(X, \G_m)_{\sQ}$
can be non-zero also for $i\geq 2$, therefore a precise computation for $X$ singular and higher
dimensional appears to be difficult. We did completely the case of curves in Sect. \ref{scurves}.
\end{remarks}

\subsection{$\RPic(X)$ and $H^*_\eh(X,\G_m)$}

By definition of $\RPic$, we have a morphism in $\DM_{-,\et}^\eff$
\begin{multline*}
\Tot\RPic(X)=\alpha^*\Hom_\Nis(M(X),\Z(1))\\
\to \Hom_\et(M_\et(X),\Z(1))=\uR\pi_*^X \G_m[-1].
\end{multline*}

This gives homomorphisms
\begin{equation}\label{eq12.1}
\sH^i(\Tot\RPic(X))\to \uR^{i-1}\pi_*^X \G_m,\quad i\ge 0.
\end{equation}

\begin{propose} \label{p12.6}For $i\le 2$, \eqref{eq12.1} is an isomorphism.
\end{propose}

\begin{proof} By blow-up induction, we reduce to the smooth case, where it follows from Hilbert's theorem 90.
\end{proof}

\subsection{$H^1_\et(X,\G_m)$ and $H^1_\eh(X,\G_m)$}

In this subsection, we assume $\pi^X:X\to \Spec k$ \emph{proper}. Recall that, then, the \'etale sheaf associated to the presheaf 
\[U\mapsto \Pic(X\times U)\]
is representable by a $k$-group scheme $\Pic_{X/k}$ locally of finite type (Gro\-then\-dieck-Murre \cite{murre}). Its connected component $\Pic^0_{X/k}$ is an extension of a semi-abelian variety by a unipotent subgroup $\cU$. By homotopy invariance of $\uR^1\pi_*^X\G_m$, we get a map
\begin{equation}\label{eq12.5}
\Pic_{X/k}/\cU\to \uR^1\pi_*^X\G_m.
\end{equation}

Recall that the right hand side is a $1$-motivic sheaf by Proposition \ref{p11.3}. We have:

\begin{propose}\label{p12.7} This map is injective with lattice cokernel.
\end{propose}

\begin{proof} Consider multiplication by an integer $n>1$ on both sides. Using the Kummer exact sequence, Proposition \ref{p11.2} and Lemma \ref{leh} c), we find that \eqref{eq12.5} is an isomorphism on $n$-torsion and injecitve on $n$-cotorsion. The conclusion then follows from Proposition \ref{p3.6}.
\end{proof}

\subsection{$\RA{1}(X)$ and $\Pic^+(X)$ for $X$ proper}

\begin{thm}\label{t12.8} For $X$ proper, the composition
\[\Pic_{X/k}/\cU\to \uR^1\pi_*^X\G_m\to \sH^2(\Tot\RPic(X))\]
where the first map is \eqref{eq12.5} and the second one is the inverse of the isomorphism \eqref{eq12.1}, induces an isomorphism
\[\Pic^+(X)\iso \RA{1}(X)\]
where $\Pic^+(X)$ is the $1$-motive defined in \cite[Ch. 4]{BSAP}.
\end{thm}

\begin{proof}  Proposition \ref{p3.10}  yields an exact sequence
\[L^1\to G^1\to \sH^2(\Tot\RPic(X))\to L^2\]
where we write $\RA{i}(X)=[L^i\to G^i]$. Propositions \ref{p12.6} and \ref{p12.7} then imply that the map of Theorem \ref{t12.8} induces an isomorphism $\Pic^0_{X/k}/\cU\iso G^1$. The conclusion follows, since on the one hand $\Pic^+(X)\simeq [0\to \Pic^0_{X/k}/\cU]$ by \cite[Lemma 5.1.2 and Remark 5.1.3]{BSAP}, and on the other hand the dual of Corollary \ref{c12.2.1} a) says that $L^1=0$.
\end{proof}

\begin{cor}\label{c12.3} For $X$ proper there is a canonical isomorphism 
\[\LA{1}(X)\iso \Alb^-(X).\qed\]
\end{cor}

\subsection{The Borel-Moore variant}

Let $X\in Sch$ be provided with a compactification $\bar X$ and closed complement $Z\by{i} \bar X$. The relative Picard functor is then representable by a $k$-group scheme locally of finite type $\Pic_{\bar X,Z}$, and we shall informally denote by $\cU$ its unipotent radical. Similarly to \eqref{eq12.1} and \eqref{eq12.5}, we have two canonical maps
\begin{equation}\label{eq12.9}
\sH^2(\Tot\RPic^c(X))\to \Pic_{\bar X,Z}^\eh\leftarrow \Pic_{\bar X,Z}/\cU
\end{equation}
where $\Pic_{\bar X,Z}^\eh$ is by definition the $1$-motivic sheaf associated to the presheaf $U\mapsto H^1_\eh(\bar X\times U,(\G_m)_{\bar X\times U}\to i_*(\G_m)_{Z\times U})$ (compare \cite[2.1]{BSAP}). Indeed, the latter group is canonically isomorphic to
\[\Hom_{\DM_{-,\et}^\eff}(M^c(X\times U),\Z(1)[2])\] 
via the localisation exact triangle. From Theorem \ref{t12.8} and Proposition \ref{c3.1*} b), we then deduce:

\begin{thm}\label{t12.9} The maps \eqref{eq12.9} induce an isomorphism
\[\RA{1}^c(X)\simeq [0\to\Pic^0(\bar X,Z)/\cU].\qed\]
\end{thm}

The following is a sequel of Proposition \ref{c3.1c}:

\begin{cor}\label{c3.1d} Let $X\in Sch(k)$ of dimension $d$. Then:\\
a) $\LA{1}^c(X)=[L_1\to A_1]$, where $A_1$ is an abelian
variety. In particular, $\LA{1}^c(X)$ is a Deligne $1$-motive.\\
b) If $X$ is normal connected and not proper, let $\bar X$ be a normal
compactification of $X$. Then $\rank L_1=\#\pi_0(\bar X - X)-1$.
\end{cor}

\begin{proof} a) follows immediately from Theorem \ref{t12.9}. For b), consider
the complex of discrete parts associated to the exact sequence \eqref{loc}: we
get with obvious notation an almost exact sequence
\[L_1(\bar X)\to L_1(X)\to L_0(\bar X-X)\to L_0(\bar X)\to L_0(X)\]
where ``almost exact" means that its homology is finite. The last group is $0$
and $L_0(\bar X)=\Z$; on the other hand,
$L_1(\bar X)=0$ by Corollary \ref{c12.2.1} b). Hence the claim.
\end{proof}

\begin{remark} As a consequence we see that in b), the number of connected
components of $\bar X-X$ only depends on $X$. Here is an elementary proof of this
fact: let $\bar X'$ be another normal compactification and $\bar X''$ the closure
of $X$ in $\bar X\times \bar X'$. Then the two maps
$\bar X''\to \bar X$ and $\bar X''\to \bar X'$ have connected fibres by Zariski's main theorem,
thus
$\bar X-X$ and $\bar X'-X$ have the same number of connected components as $\bar X''-X$. (The
second author is indebted to Marc Hindry for a discussion leading to this proof.)
\end{remark}

We shall also need the following computation in the next subsection.

\begin{thm}\label{t12.9.2} Let $\bar X$ be smooth and proper, $Z\subset \bar X$ a divisor with normal crossings and $X= \bar X-Z$. Let $Z_1,\dots,Z_r$ be the irreducible components of $Z$ and set
\[Z^{(p)} =
\begin{cases}
\bar X &\text{if $p=0$}\\
\coprod\limits_{i_1<\dots<i_p} Z_{i_1}\cap\dots \cap Z_{i_p}&\text{if $p>0$.}
\end{cases}\]
Let $\NS^{(p)}_c(X)$ (\resp $\Pic^{(p)}_c(X)$, $T^{(p)}_c(X)$ be the cohomology (\resp the connected component of the cohomology) in degree $p$ of the complex
\[\dots \to\NS(Z^{(p-1)})\to\NS(Z^{(p)})\to\NS(Z^{(p+1)})\to\dots\]
(\resp 
\[\dots \to\Pic^0(Z^{(p-1)})\to\Pic^0(Z^{(p)})\to\Pic^0(Z^{(p+1)})\to\dots\]
\[\dots \to\R_{\pi_0(Z^{(p-1)})/k} \G_m\to\R_{\pi_0(Z^{(p)})/k} \G_m\to\R_{\pi_0(Z^{(p+1)})/k} \G_m\to\dots).\]
Then, for all $n\ge 0$, $\RA{n}^c(X)$ is of the form $[\NS^{(n-2)}_c(X)\by{u^n}G^{(n)}_c]$, where $G^{(n)}_c$ is an extension of $\Pic^{(n-1)}_c(X)$ by $T^{(n)}_c(X)$.
\end{thm}

\begin{proof} A standard argument (compare e.g. \cite[3.3]{eklv}) yields a spectral sequence of cohomological type in ${}^t\M$:
\[E_1^{p,q} = \RA{q}^c(Z^{(p)})\Rightarrow \RA{p+q}^c(X).\]

By Corollary \ref{HRPic}, we have $E_2^{p,2}=[\NS^{(p)}_c(X)\to 0]$, $E_2^{p,1}=[0\to\Pic^{(p)}_c(X)]$ and $E_2^{p,0}=[0\to T^{(p)}_c(X)]$. By Proposition \ref{hom}, all $d_2$ differentials are $0$, hence the theorem.
\end{proof}

\begin{cor}\label{c12.9} With notation as in Theorem \ref{t12.9.2}, the complex $\RPic(M^c(X)(1)[2])$ is \qi to
\[\dots\to [\Z^{\pi_0(Z^{(p-2)})}\to 0]\to \dots\]
In particular, $\RA{0}(M^c(X)(1)[2])=\RA{1}(M^c(X)(1)[2])=0$ and $\RA{2}(M^c(X)(1)[2])=[\Z^{\pi_0^c(X)}\to 0]$ (see Definition \ref{d11.5}).
\end{cor}

\begin{proof} This follows from Theorem \ref{t12.9.2} via the formula $M^c(X\times\P^1) =M^c(X)\oplus M^c(X)(1)[2]$, noting that $\bar X\times \P^1$ is a smooth compactification of $X\times\P^1$ with $\bar X\times \P^1-X\times \P^1$ a divisor with normal crossings with components $Z_i\times \P^1$, and
\begin{align*}
\NS(Z^{(p)}\times\P^1)&=\NS(Z^{(p)})\oplus \Z^{\pi_0(Z^{(p)})}\\
 \Pic^0(Z^{(p)}\times \P^1)&=\Pic^0(Z^{(p)})\\ 
 \pi_0(Z^{(p)}\times\P^1)&=\pi_0(Z^{(p)}).
\end{align*}
\end{proof}

\begin{remark} Let $X$ be arbitrary, and filter it by its successive singular loci, \ie
\[X=X^{(0)}\supset X^{(1)}\supset\dots\]
where $X^{(i+1)}=X^{(i)}_{sing}$. Then we have a spectral sequence of cohomological type in ${}^t\M$
\[E_2^{p,q} =\RA{p+q}^c(X^{(q)}-X^{(q+1)})\Rightarrow \RA{p+q}^c(X) \]
in which the $E_2$-terms involve smooth varieties. This qualitatively reduces the computation of $\RA{*}^c(X)$ to the case of smooth varieties, but the actual computation may be complicated; we leave this to the interested reader.
\end{remark}

\subsection{$\LA{1}^*$ and $\Alb^+$}

\begin{lemma}\label{l11.6} Let $n>0$ and $Z\in Sch$ of dimension $<n$; then 
\begin{gather*}
\RA{i}(M(Z)^*(n)[2n])=0 \text{ for } i\le 1\\
\RA{2}(M(Z)^*(n)[2n])=[\Z^{\pi_0^c(Z^{[n-1]})}\to 0]
\end{gather*}
where $Z^{[n-1]}$ is the disjoint union of the irreducible components of $Z$ of dimension $n-1$.
\end{lemma}

\begin{proof} Let 
\[\begin{CD}
\tilde T@>>> \tilde Z\\
@VVV @VVV\\
T@>>> Z
\end{CD}\]
be an abstract blow-up square, with $\tilde Z$ smooth and $\dim T,\dim \tilde T<\dim Z$. By
Lemma \ref{leffe}, $M(T)^*(n-2)[2n-4]$ and $M(\tilde T)^*(n-2)[2n-4]$ are effective, so by
Proposition \ref{ptate}, the exact triangle
\begin{multline*}
\RPic(M(\tilde T)^*(n)[2n])\to \RPic(M(\tilde Z)^*(n)[2n])\oplus \RPic(M(T)^*(n)[2n])\\
\to \RPic(M(Z)^*(n)[2n])\by{+1}
\end{multline*}
degenerates into an isomorphism
\[\RPic(M(\tilde Z)^*(n)[2n])\iso \RPic(M(Z)^*(n)[2n]).\]

The lemma now follows from Corollary \ref{c12.9} by taking for $\tilde Z$ a desingularisation
of $Z^{[n-1]}$ and for $T$ the union of the singular locus of $Z$ and its irreducible
components of dimension $<n-1$ (note that $M(\tilde Z)^*(n)[2n]\simeq M^c(\tilde Z)(1)[2]$).
\end{proof}

\begin{lemma} \label{relPicplus} Let $\bar X$ a proper smooth scheme with a
pair $Y$ and $Z $ of disjoint  closed (reduced) subschemes of pure
codimension $1$ in $\bar X$. We then have 
$$\RA{1}(\bar X- Z, Y) \cong\Pic^+(\bar X- Z, Y)$$ 
(see \cite[2.2.1]{BSAP} for the definition of
relative $\Pic^+$). 
\end{lemma}

\begin{proof} The following exact sequence provides the weight filtration
$$0\to\RA{1}(\bar X,Y)\to \RA{1}(\bar X - Z,Y)\to \RA{2}_Z(\bar X, Y)$$
where $\RA{1}(\bar X, Y) \cong \RA{1}^c(\bar X-  Y)\cong [0\to \Pic^0(\bar X, Y)]$ by
Theorem \ref{t12.9} (here $\cU=0$ since $\bar X$  is smooth). Also
$\RA{2}_Z(\bar X, Y)\cong \RA{2}_Z(\bar X)=[\Div_Z(\bar X)\to 0]$ from
\ref{lsupports}: thus the discrete part of $\RA{1}(\bar X- Z, Y)$ is given by a subgroup $D$ of
$\Div_Z(\bar X)=\Div_Z(\bar X,Y)$. 

It remains to identify the map $u:D\to \Pic^0(\bar X,Y)$. Using now the exact sequence
$$\RA{0}(Y)\to \RA{1}(\bar X - Z,Y)\to \RA{1}(\bar X - Z)\to \RA{1}(Y)$$
where $\RA{i}(Y)$ is of weight $<0$ for $i\leq 1$ (\ref{c3.1} and \ref{c12.2.1}), we get that
$u$ is the canonical lifting of the map of the $1$-motive $\RA{1}(\bar X
- Z)$ described in \ref{HRPic}. Thus  $D=\Div_Z^0(\bar X, Y)$ and the claimed isomorphism is
clear.
\end{proof}

This proof also gives:

\begin{cor}\label{algeqzero} We have
$$[\Div_Z^0(\bar X, Y)\to 0]=\ker (\RA{2}_Z(\bar X, Y)\to \RA{2}(\bar X, Y)).$$
\end{cor}

We shall need:

\begin{thm}[Relative duality]\label{treldual} Let $\bar X$,  $Y$ and $Z $ be
as above and further assume that $\bar X$ is $n$-dimensional. Then
$$M (\bar X- Z, Y)^*(n)[2n]\cong M(\bar X- Y, Z)$$
and therefore
$$\RPic^*(\bar X- Z, Y)\cong \RPic (\bar X- Y, Z)$$
and dually for $\LAlb$.
\end{thm}

\begin{proof} See \cite{rel_duality}.
\end{proof}

\begin{cor}\label{cycle} Let $Z$ be a divisor in $\bar X$ such that $Z\cap Y
=\emptyset$. There exists a  ``cycle class'' map $\eta$ fitting in the
following commutative diagram
$$
\begin{CD}
\RA{2}(M(Z)^*(n)[2n])@>{\eta}>>  \RA{2}^c(\bar X - Y)\\
 @V{||}VV @V{||}VV \\
 \RA{2}_Z(\bar X, Y)@>{}>> \RA{2}(\bar X, Y)
\end{CD}
$$
Writing $Z= \cup Z_i$ as union of its irreducible components  we have that
$\eta$ on $Z_i$  is the ``fundamental class'' of $Z_i$ in $\bar X$ modulo
algebraic equivalence.\end{cor}

\begin{proof} We have a map $M(Z)\to M(\bar X -Y)$, and the vertical isomorphisms
in the following commutative square
$$
\begin{CD}
M(\bar X- Y)^*(n)[2n]@>{}>>  M(Z)^*(n)[2n]\\
 @A{||}AA @A{||}AA \\
M(\bar X, Y)@>{}>> M^Z(\bar X, Y)
\end{CD}
$$ are given by relative duality.
\end{proof}

\begin{thm}\label{*=-} For $X\in Sch$ we have
$$\RA{1}^*(X)\cong \Pic^-(X).$$
\end{thm}

\begin{proof} We are left to consider $X\in Sch$ purely of dimension $n$
with the following associated set of data and conditions.

For the irreducible components $X_1,\dots,X_r$ of $X$ we let $\tilde X$ be a
desingularisation of $\coprod X_i$, $S\df X_{sing}\cup \bigcup _{i\ne j}
S_i\cap S_j$ and $\tilde S$ the inverse image of $S$ in $\tilde X$. We let
$\bar X$ be a smooth proper compactification with normal crossing divisor
$Y$. Let $\bar S$ denote the Zariski closure of $\tilde S$ in $\bar X$.
Assume that $Y+\bar S$ is a reduced normal crossing divisor  in $\bar X$.
Finally denote by $Z$ the union of all compact components of divisors in
$\tilde S$ (\cf \cite[2.2]{BSAP}).

We have an exact sequence coming from the abstract blow-up square
associated to the above picture: 
\begin{multline*}
\dots\to\RA{1}^*(\tilde X)\oplus \RA{1}(M(S)^*(n)[2n])\to\RA{1}^*(X)\\
\to
\RA{2}(M(\tilde S)^*(n)[2n])\to
\RA{2}^*(\tilde X)\oplus\RA{2}(M(S)^*(n)[2n])
\end{multline*}

Now:
\begin{itemize}
\item the first map is injective (Lemma \ref{l11.6}),
\item $\RA{1}^*(\tilde X)=\RA{1}^c(\tilde
X)=[0\to\Pic^0(\bar X,Y)]$ since $\tilde X$ is smooth (Theorem \ref{t12.9}; note that $\cU=0$
by the smoothness of $\tilde X$),
\item $\RA{1}(M(S)^*(n)[2n])=0$ (Lemma
\ref{l11.6}),
\item $\RA{2}(M(\tilde S)^*(n)[2n])= [\Z^{\pi_0^c(\tilde S^{[n-1]})}\to
0]\df [\Div_{\bar S}(\bar X, Y)\to 0]$ is given by the free abelian group on compact
irreducible components of $\tilde S$ (Lemma \ref{l11.6}),
\item $\RA{2}^*(\tilde X)=\RA{2}^c(\tilde X)= \RA{2}(\bar X, Y)=[\NS^{(0)}_c(\tilde
X)\by{u^2}G^{(2)}_c]$ (Theorem \ref{t12.9.2}),
\item $\RA{2}(M(S)^*(n)[2n])=[\Z^{\pi_0^c( S^{[n-1]})}\to 0]=  [\Div_{S}(X)\to 0]$ (Lem\-ma
\ref{l11.6}).
\end{itemize}

We may therefore rewrite the above exact sequence as follows:
\begin{multline*}
0\to[0\to\Pic^0(\bar X,Y)]\to\RA{1}^*(X)\to [\Div_{\bar S}(\bar X, Y)\to 0]\\
\by{\alpha} [\NS^{(0)}_c(\tilde X)\by{u^2}G^{(2)}_c]\oplus [\Div_{S}(X)\to 0].
\end{multline*}

The map $\Div_{\bar S}(\bar X, Y)\to \Div_{ S}(X)$  induced from $M (\tilde S)\to M(S)$ is
clearly the proper push-forward of Weil divisors. The map 
\[[\Div_{\bar S}(\bar X, Y)\to 0]\to
[\NS^{(0)}_c(\tilde X)\by{u^2}G^{(2)}_c]\] 
is the cycle class map described in Corollary
\ref{cycle}. By Corollary \ref{algeqzero} we then get
\[
\ker  \alpha=[\Div_{\bar S/S}^0(\bar X, Y)\to 0] 
\]
where the lattice $\Div_{\bar S/S}^0(\bar X, Y) $ is from the definition of
$\Pic^-$ (see \cite[2.2.1]{BSAP}).
In other terms, we have
\[\RA{1}^*(X)=[\Div_{\bar S/S}^0(\bar X, Y) \by{u}\Pic^0(\bar X,Y)]\]
and we are left to check that the mapping $u$ is the one described in
\cite{BSAP}. Just observe that, by Lemma \ref{l11.6} and Theorem \ref{treldual},
\begin{multline*}\RA{1}^*(X) \into \RA{1}^*(X, S) \cong \RA{1}^*(\tilde X,
\tilde S)\cong \RA{1}^* (\tilde X, Z) \\= \RA{1}^*(\bar X -Y, Z)
\cong\RA{1}(\bar X -Z, Y)
\end{multline*} and the latter is isomorphic to $\Pic^+(\bar X-Z,Y)$ by
\ref{relPicplus}.
Since, by construction, $\Pic^-(X)$ is a sub-1-motive of $\Pic^+(\bar
X-Z,Y)$ the isomorphism of \ref{relPicplus} restricts to the claimed one.
\end{proof}

\section{Generalisations of Ro\v\i tman's theorem}\label{rovi}

In this section we give a unified treatment of Ro\v\i tman's theorem on torsion $0$-cycles on a smooth projective variety and its various generalisations. 

\subsection{Motivic and classical Albanese} 

Let $X\in Sch(k)$; we assume $X$ smooth if $p>1$ and $X$ semi-normal (in particular reduced)
if $p=1$, see Lemma \ref{l12.3}. Recall that Suslin's singular algebraic homology is
\[H_j(X)\df \Hom_{\DM_-^\eff}(\Z[j], M(X))= \HH^{-j}_\Nis(k,C_*(X)) \]
for any scheme $p: X\to k$. On the other hand, we may define
\[H^\et_j(X)\df \Hom_{\DM_{-,\et}^\eff}(\Z[j], M_\et(X))= \HH^{-j}_\et(k,\alpha^*C_*(X)).\]

We also have versions with torsion coefficients:
\[H_j^\et(X,\Z/n)=\Hom_{\DM_{-,\et}^\eff}(\Z[j], M_\et(X)\otimes \Z/n), (n,p)=1\] 
(this convention is chosen so that we have the usual long exact sequences.) 

We shall use the following notation throughout:

\begin{notation} For any $M\in \DM_\gm^\eff$ and any abelian group $A$, we write 
$H_j^{(1)}(M,A)$ for the abelian group \index{$H_j^{(1)}$}
\[\Hom_{\DM_{-,\et}^\eff}(\Z[j],\Tot\LAlb(M)\otimes
A)\simeq \HH^{-j}_\et(k,\Tot\LAlb(M)\oo^L A).\] 
This is \emph{Suslin $1$-motivic homology
of $M$ with coefficients in $A$}. If $M=M(X)$, we write $H_j^{(1)}(X,A)$ for $H_j^{(1)}(M,A)$.
We drop $A$ in the case where $A=\Z$.
\end{notation}

The motivic Albanese map
\eqref{amap} then gives  maps
\begin{align}
H_j^\et(X)& \to H_j^{(1)}(X)\label{eqalb}\\
H_j^\et(X,\Z/n)& \to H_j^{(1)}(X,\Z/n).\label{eqalbt}
\intertext{hence in the limit}
H_j^\et(X,(\Q/\Z)')& \to H_j^{(1)}(X,(\Q/\Z)')\label{eqalbt2}
\end{align}
where, as usual, $(\Q/\Z)'\df \bigoplus_{l\ne p} \Q_\ell/\Z_\ell$.

\begin{propose} \label{sus}
If $X$ is a smooth curve (or any curve in characteristic $0$), the maps \eqref{eqalb},
\eqref{eqalbt} and
\eqref{eqalbt2} are isomorphisms for any $j$. 
\end{propose}

\begin{proof} This follows immediately from Proposition \ref{cd2}.
\end{proof}

Note that if $X =\bar X -Y$ is a smooth curve obtained by removing
a finite set of closed points from a projective smooth curve $\bar
X $ then $\cA_{X/k}= \Pic_{(\bar X, Y)/k}$ is the relative Picard
scheme (see \cite{BSAP} for its representability) and the Albanese
map just send a point $P\in X$ to $(\cO_{\bar X}(P), 1)$ where 1
is the tautological section, trivialising $\cO_{\bar X}(P)$ on
$X$. We then have the following result  (\cf \cite[Lect. 7, Th.
7.16]{VL}).

\begin{cor}\label{relpic} If $X =\bar X -Y$ is a smooth curve,
$$H_0^\et(X)\to \Pic (\bar X, Y)[1/p]$$
is an isomorphism.
\end{cor}

Now let $\cA_{X/k}^\eh$ be as in Proposition \ref{c3.1bis} and Remark \ref{r11.1}: recall that
$\cA_{X/k}^\eh=\cA_{X/k}$ if $X$ is strictly reduced. The map
$\Tot\LAlb(X)\to
\cA_{X/k}^\eh$ of
\loccit induces a homomorphism
\begin{equation}\label{eqhom}
H_0^{(1)}(X)\to \cA_{X/k}^\eh(k)[1/p]
\end{equation}
which is not an isomorphism in general (but see Lemma \ref{l7.3.1}).

Finally, composing \eqref{eqhom}, \eqref{eqalb} and the obvious map $H_0(X)[1/p]\to H_0^\et(X)$,
we get a map
\begin{equation}\label{classalb}
H_0(X)[1/p]\to \cA_{X/k}^\eh(k)[1/p].
\end{equation}

We may further restrict to parts of degree $0$, getting a map
\[H_0(X)^0[1/p]\to (\cA_{X/k}^\eh)^0(k)[1/p].\]

If $X$ is smooth, this is the $\Z[1/p]$-localisation of the generalised Albanese map of
Spie\ss-Szamuely \cite[(2)]{spsz}.

Dually to Lemma \ref{ltlexact}, the functor 
\begin{align*}
D^b(\M[1/p])&\to \DM_{-,\et}^\eff\\
C&\mapsto \Tot(C)\otimes (\Q/\Z)'
\end{align*}
is exact with respect to the ${}_t\M$ t-structure on the left and the homotopy t-structure on
the right; in other words:

\begin{lemma}\label{l12.1} For any $C\in D^b(\M[1/p])$, there are canonical isomorphisms of
sheaves
\[\sH_j(\Tot(C)\otimes (\Q/\Z)')\simeq \Tot({}_t H_j(C))\otimes(\Q/\Z)'\]
(note that the right hand side is a single sheaf!) In particular, for $C=\LAlb(M)$ and $k$
algebraiclly closed:
\[H_j^{(1)}(M, (\Q/\Z)')\simeq \Gamma(k,\Tot(\LA{j}(M))\otimes(\Q/\Z)').\]
\end{lemma}

\subsection{A proof of Ro\v\i tman's and Spie\ss-Szamuely's theorems}

\emph{Until the end of this section, we assume that $k$ is algebraically closed.} In this subsection, we only deal with smooth schemes and the characteristic is arbitrary: we shall
show how the results of Section \ref{comp} allows us to recover the
classical  theorem of Ro\v\i tman on torsion $0$-cycles up to $p$-torsion, as well as its generalisation to
smooth varieties by Spie\ss-Szamuely \cite{spsz}. The reader should compare our argument with
theirs (\loccit, \S 5).

Since $k$ is algebraically closed, Corollary \ref{cD.1} implies

\begin{lemma}\label{l7.3} For any $j\in\Z$, $H^\et_j(X)=H_j(X)[1/p]$; similarly with finite or divisible coefficients.\qed
\end{lemma}

Moreover, it is easy to evaluate
$H_j^{(1)}(X)=\sH_j(\Tot\LAlb(X))(k)$ out of Theorem \ref{trunc}: if
$\LA{n}(X)=[L_n\to G_n]$, we have a long exact sequence coming from Proposition \ref{p3.10}
\begin{multline}\label{eq13.1}
L_{n+1}(k)[1/p]\to G_{n+1}(k)[1/p]\to H_j^{(1)}(X)\\
\to L_n(k)[1/p]\to G_n(k)[1/p]\to \dots
\end{multline}

Thus:

\begin{lemma}\label{l7.3.1} For $X$ smooth, \eqref{eqhom} is an
isomorphism and we have
\begin{align}
H_1^{(1)}(X)&\simeq \NS_{X/k}^*(k)[1/p]\label{eqhom1}\\
H_j^{(1)}(X)&=0 \text{ if } j\ne 0,1.\notag
\end{align}
\end{lemma}

Here is now the main lemma:

\begin{lemma}\label{l7.2} The map \eqref{eqalbt} is an isomorphism for $j=0,1$ and surjective
for $j=2$.
\end{lemma}

\begin{proof} It is sufficient to show that $\Hom(\eqref{eqalbt},\Z/n(1))$ is an
isomorphism (\resp injective). Since $k$ is algebraically closed, by finite group duality (!)
this is nothing else than the map
\[H^j(\Tot\RPic(X)\otimes\Z/n)\to H^j_\et(X,\Z/n(1)).\]

By definition, the left hand side is 
\[H^j(\alpha^*\ihom_\Nis(M(X),\Z(1))\otimes\Z/n)=H^j_\Nis(X,\Z/n(1))\]
(using again that $k$ is algebraically closed), hence the result follows from Hilbert's theorem
90.
\end{proof}

From this and Lemma \ref{l12.1} we deduce:
\begin{cor}\label{c12.2} The homomorphism \eqref{eqalbt2}
\[H_j(X,(\Q/\Z)')\to H_j^{(1)}(X,(\Q/\Z)')\]
(see Lemma \ref{l7.3}) is bijective for $j=0,1$ and surjective for $j=2$.
\end{cor}

The following theorem extends in particular \cite[Th. 1.1]{spsz} to all smooth
varieties\footnote{In \loccit, $X$ is supposed to be admit an open embedding in a smooth
projective variety.}.

\begin{thm}\label{troitclas} a) The map \eqref{classalb} is an isomorphism on torsion prime
to $p$.\\ 
b) $H_1(X)\otimes (\Q/\Z)'=0$.\\
c) The map \eqref{eqalb} yields a surjection
\[H_1(X)\{p'\}\onto \NS_{X/k}^*(k)\{p'\}\]
where $M\{p'\}$ denotes the torsion prime to $p$ in an abelian group $M$.
\end{thm}

\begin{proof} Lemmas \ref{l7.3} and \ref{l7.3.1} reduce us to show that \eqref{eqalb} is an
isomorphism on torsion for $j=0$. We have commutative diagrams with exact
rows:
\begin{equation}\label{eq7.7}
\begin{CD}
0&\to&H_j(X)\otimes(\Q/\Z)'&\to& H_j(X,(\Q/\Z)')&\to& H_{j-1}(X)\{p'\}&\to& 0\\
&&@VVV @VVV @VVV\\
0&\to&H_j^{(1)}(X)\otimes (\Q/\Z)'&\to&H_j^{(1)}(X,(\Q/\Z)') &\to&
H_{j-1}^{(1)}(X)\{p'\}&\to& 0 
\end{CD}
\end{equation}

For $j=1$, the middle vertical map is an isomorphism by Lemma \ref{l7.2} or Corollary
\ref{c12.2} and
$H_1^{(1)}(X)\otimes
\Q/\Z=0$ by \eqref{eqhom1}, which gives a) and b). For $j=2$, the middle map is surjective by
the same lemma and corollary, which gives c). The proof is complete.
\end{proof}

\begin{remark} If $X$ is smooth projective of dimension $n$, $H_j(X)$ is isomorphic to the
higher Chow group $CH^n(X,j)$. In \eqref{eq7.7} for $j=2$, the lower left term is $0$ by Lemma
\ref{l7.3.1}. The composite map
\[
H_2(X,(\Q/\Z)')\to H_1(X)\{p'\}\to H_1^{(1)}(X)\{p'\}=\NS_{X/k}^*(k)\{p'\}
\] 
is ``dual" to the map
\[\NS(X)\otimes (\Q/\Z)'\to H^2_\et(X,\Q/\Z(1))\]
whose cokernel is $Br(X)\{p'\}$. Let
\[Br(X)^D=\varinjlim_{(n,p)=1} \Hom({}_n
Br(X),\mu_n):\] 
a diagram chase in \eqref{eq7.7} for $j=2$ then yields an exact sequence
\begin{multline*}
0\to CH^n(X,2)\otimes (\Q/\Z)'\to Br(X)^D\\
\to CH^n(X,1)\{p'\}\to
\NS_{X/k}^*(k)\{p'\}\to 0.
\end{multline*}
Together with $CH^n(X,1)\otimes (\Q/\Z)'=0$, this should be considered as a natural complement
to Ro\v\i tman's theorem.
\end{remark}

\subsection{Generalisation to singular schemes} We now assume $\car k\allowbreak=0$, and show
how the results of Section \ref{comps} allow us to extend the results of the previous
subsection to singular schemes. By blow-up induction and the 5 lemma, we get:

\begin{propose} \label{p12.3} The isomorphisms and surjection of Lemma \allowbreak\ref{l7.2} and
Corollary
\ref{c12.2} extend to all $X\in Sch$.\qed
\end{propose}

Let $\LA{1}(X)=[L_1\by{u_1} G_1]$. Proposition \ref{p12.3}, the exact sequence \eqref{eq13.1}
and the snake chase in the proof of Theorem \ref{troitclas} give:

\begin{cor} For $X\in Sch$, we have exact sequences
\begin{gather*}
0\to H_1(X)\otimes \Q/\Z\to \ker(u_1)\otimes \Q/\Z\to H_0(X)_\tors\to \coker(u_1)_\tors\to 0\\
0\to H_1(X)\otimes \Q/\Z\to H_1^{(1)}(X,\Q/\Z)\to H_0(X)_\tors\to 0.
\end{gather*}
\end{cor}

The second exact sequence is more intrinsic than the first, but note that it does not give
information on $H_1(X)\otimes \Q/\Z$.

\begin{cor}\label{c13.3} If $X$ is normal, $H_1(X)\otimes \Q/\Z=0$ and there is an isomorphism
\[\cA_{X/k}(k)_\tors\iso H_0(X)_\tors.\]
\end{cor}

\begin{proof} This follows from the previous corollary and Corollary \ref{c12.2.1} c).
\end{proof}

\begin{remark} Theorem \ref{c12.3} shows that the second isomorphism of Proposition
\ref{p12.3} coincides with the one of Geisser in \cite[Th. 6.2]{geisser2} when $X$ is proper.
When $X$ is further normal, the isomorphism of Corollary \ref{c13.3} also coincides with the
one of his Theorem 6.1.
\end{remark}

\begin{remarks} Note that the reformulation of ``Roitman's theorem'' involving $\ker u_1$ is
the best possible!\\
1) Let X be a proper scheme such that $\Pic^0(X)/\cU =
\G_m^r$ is a torus (more likely such that $\Alb (X_0) = 0$ where $X_0\to X$
is a resolution, according with the description in \cite[pag. 68]{BSAP}).
Then $\RA{1} (X)^*=\LA{1} (X) = [\Z^r\to 0]$ is the character group (\cf
\cite[5.1.4]{BSAP}).
For example, take a nodal projective curve $X$ with resolution $X_0=\P^1$.
In this case the map \eqref{eqalb} is an isomorphism for all $j$ and  thus
$\ker (u_1)\otimes \Q/\Z = H_1(X)\otimes  \Q/\Z=(\Q/\Z)^r$.\\
2) For Borel-Moore and $\LA{1}^c (X) = \LA{1}^*(X)$ for $X$ smooth open is
Cartier dual of $\Pic^0 (\bar X, Y)$ then (\cf \cite[pag. 47]{BSAP} the
$\ker u_1^c$ can be non-zero take $\bar X = \P^1$ and $Y =$ finite number of
points.
\end{remarks} 

\subsection{Borel-Moore Ro\v\i tman} We are still in characteristic $0$. Recall that the
Borel-Moore motivic homology group 
\[H_j^c(X,\Z)\df \Hom(\Z[j],M^c(X))\] 
is canonically isomorphic to Bloch's higher Chow group $CH_0(X,j)$.  Similarly to the previous
sections, we have maps
\begin{gather*}
H_j^c(X,\Z)\to H_j^{(1)}(M^c(X))=:H_j^{c,(1)}(X)\\
H_j^c(X,\Q/\Z)\to H_j^{(1)}(M^c(X),\Q/\Z)=:H_j^{c,(1)}(X,\Q/\Z)
\end{gather*}
and

\begin{propose}\label{p13.4} The second map is an isomorphism for $j=0,1$ and surjective for
$j=2$.
\end{propose}

\begin{proof} By localisation induction, reduce to $X$ proper and use Proposition \ref{p12.3}.
\end{proof}

\begin{cor} For $X\in Sch$, we have exact sequences
\begin{multline*}
0\to CH_0(X,1)\otimes \Q/\Z\to \ker(u_1^c)\otimes \Q/\Z\\
\to CH_0(X)_\tors\to \coker(u_1^c)_\tors\to 0
\end{multline*}
\[
0\to CH_0(X,1)\otimes \Q/\Z\to H_j^{c,(1)}(X,\Q/\Z)\to CH_0(X)_\tors\to 0
\]
where we write $\LA{1}^c(X)=[L_1^c\by{u_1^c} G_1^c]$.
\end{cor}

\subsection{``Cohomological" Ro\v\i tman}

We are still in characteristic $0$.

\begin{lemma}\label{lvan} Let $0< r\le n$. Then for any $Z\in Sch$ of dimension $\le n-r$ and
any $i> 2(n-r)$, we have $H^{i}_\cdh(Z,\Q/\Z(n))=0$.
\end{lemma}

\begin{proof} By blow-up induction we reduce to the case where $Z$ is smooth of pure dimension
$n-r$; then $H^i_\cdh(Z,\Q/\Z(n))=H^i_\Nis(Z,\Q/\Z(n))$. Since $k$ is algebraically closed, and
$n\ge \dim Z$, $H^i_\Nis(Z,\Q/\Z(n))\simeq H^i_\et(Z,\Q/\Z(n))$ by Suslin's theorem
\cite{suslin} and the vanishing follows from the known bound for \'etale cohomological
dimension.
\end{proof}

Now consider the $1$-motive $\LA{1}^*(X)$ for $X$ of
dimension $n$. This time, we have  maps
\begin{gather}
H^{2n-j}_\cdh(X,\Z(n))\to H_j^{(1)}(M(X)^*(n)[2n])=:H^{2n-j}_{(1)}(X,\Z(n))\notag\\
H^{2n-j}_\cdh(X,\Q/\Z(n))\to
H_j^{(1)}(M(X)^*(n)[2n],\Q/\Z)=:H^{2n-j}_{(1)}(X,\Q/\Z(n)).\label{corot}
\end{gather}

\begin{lemma} \label{lsmallis} Let $Z\in Sch$ be of dimension $<n$. Then the map
\[H^{2n-2}_\cdh(Z,\Q/\Z(n))\to H_2^{(1)}(M(Z)^*(n)[2n],\Q/\Z)\]
is an isomorphism.
\end{lemma}

\begin{proof} For notational simplicity, write $H^*(Y,n)$ for $H^*_\cdh(Y,\Q/\Z(n))$ and
$F_j(Y)$ for $H_j^{(1)}(M(Y)^*(n)[2n],\Q/\Z))$, where $Y$ is a scheme of dimension $\le
n$. Let $\tilde Z,T,\tilde T$ be as in the proof of Lemma \ref{l11.6}. Then Lemma
\ref{lvan} and proposition \ref{ptate} yield a commutative diagram
\[\begin{CD}
H^{2n-2}(Z,n)@>>> H^{2n-2}(\tilde Z,n)\\
@VVV @VVV\\
F_2(Z)@>>> F_2(\tilde Z)
\end{CD}\]
in which both horizontal maps are isomorphisms. Therefore, it suffices to prove the lemma when
$Z$ is smooth quasiprojective of dimension $n-1$. 

The motive $\RA{2}(M(Z)^*(n)[2n])\simeq\RA{2}(M^c(Z)(1)[2])$ was computed in Corollary
\ref{c12.9}: it is $[\Z^{\pi_0^c(Z)}\to 0]$. Therefore, we get
\[F_2(Z)\simeq \Q/\Z(1)[\pi_0^c(Z)].\]

On the other hand, the trace map defines an isomorphism 
\[H^{2n-2}(Z,n)\iso \Q/\Z(1)[\pi_0^c(Z)]\]
and the issue is to prove that the vertical map in the diagram is this isomorphism. For this,
we first may reduce to $Z$ projective and connected. Now we propose the following argument: take
a chain of smooth closed subvarieties $Z\supset Z_2\supset\dots \supset Z_{n}$, with
$Z_i$ of dimension $n-i$ and connected (take multiple hyperplane sections up to $Z_{n-1}$ and
then a single point of $Z_{n-1}$ for $Z_n$. The Gysin exact triangles give commutative diagrams
\[\begin{CD}
H^{2n-2i-2}(Z_{i+1},n-i)@>>> H^{2n-2i}(Z_i,n-i+1)\\
@VVV @VVV\\
F_2(Z_{i+1})@>>> F_2(Z_i)
\end{CD}\]
in which both horizontal maps are isomorphisms: thus we are reduced to the case $\dim Z=0$,
where it follows from Proposition \ref{sus} applied to $X=\P^1$.
\end{proof}

\begin{thm} The map \eqref{corot} is an isomorphism for $j=0,1$.
\end{thm}

\begin{proof} This is easy and left to the reader for $j=0$. For $j=1$, we argue as usual by
blowup induction.  In the situation of
\ref{blowups}, we then have a commutative diagram of exact sequences
\[\begin{CD}
\scriptstyle H^{2n-2}(\tilde X,n)\oplus H^{2n-2}(Z,n)& \to& \scriptstyle
H^{2n-2}(\tilde Z,n)&\to&\scriptstyle  H^{2n-1}(X,n)& \to&
\scriptstyle H^{2n-1}(\tilde X,n)\oplus H^{2n-1}(Z,n)\\ 
@VVV @VVV @VVV @VVV \\
\scriptstyle F_2(\tilde X)\oplus F_2(Z)& \to&\scriptstyle   F_2(\tilde
Z)& \to&\scriptstyle  F_1(X)& \to&\scriptstyle  F_1(\tilde X)\oplus
F_1(Z).
\end{CD}\]

In this diagram, we have $F_1(Z)=0$ by Lemma \ref{l11.6} and $H^{2n-1}(Z,n)\allowbreak=0$ by
Lemma
\ref{lvan}, and the same lemmas imply that both rightmost horizontal maps are surjective. The
rightmost vertical map is now an isomorphism by Proposition \ref{p13.4}, which also gives the
surjectivity of $H^{2n-2}(\tilde X,n)\to F_2(\tilde X)$. Finally, Lemma \ref{lsmallis} implies
that $H^{2n-2}(\tilde Z,n)\to F_2(\tilde Z)$ and $H^{2n-2}(Z,n)\to F_2(Z)$ are isomorphisms,
and the conclusion follows from the 5 lemma.
\end{proof}

\begin{cor} \label{c14.4} For $X\in Sch$ of dimension $n$, we have exact sequences
\begin{multline*}
0\to H^{2n-1}(X,\Z(n))\otimes \Q/\Z\to \ker(u^*_1)\otimes \Q/\Z\\
\to H^{2n}(X,\Z(n))_\tors\to\coker(u^*_1)_\tors\to 0
\end{multline*}
\begin{multline*}
0\to H^{2n-1}(X,\Z(n))\otimes \Q/\Z\to H^{2n-1}_{(1)}(X,\Q/\Z(n))\\
\to H^{2n}(X,\Z(n))_\tors\to 0
\end{multline*}
where $u^*_1$ is the map involved in the $1$-motive $\LA{1}^*(X)$ (which is isomorphic to $\Alb^+(X)$ by the dual of Theorem \ref{*=-}).
\end{cor}

\begin{cor}\label{c14.5} If $X$ is a \emph{proper}\/ scheme of dimension $n$ we then get $H^{2n-1}(X,\Z(n))\otimes \Q/\Z=0$ and an isomorphism
\[\Alb^+(X)(k)_\tors\iso H^{2n}(X,\Z(n))_\tors.\]
\end{cor}

\begin{proof} If $X$ is proper then $\LA{1}^*(X)\cong \Alb^+(X)$ is semiabelian and the claim follows from the previous corollary.\end{proof}

\begin{remark}\label{r14.5} Marc Levine outlined us how to construct a ``cycle map" $ c\ell^\cdh$ from $CH^n_{LW}(X)$ to
$H^{2n}_\cdh(X,\Z(n))$, where $CH^n_{LW}(X)$ is the Levine-Weibel cohomological Chow group of zero cycles. This gives a map
$$ c\ell^\cdh_\tors :CH^n_{LW}(X)_\tors\to H^{2n}_\cdh(X,\Z(n))_\tors $$
which most likely fits in a commutative diagram (for $X$ projective)
\[\begin{CD}
CH^n_{LW}(X)_\tors@>c\ell^\cdh_\tors>> H^{2n}_\cdh(X,\Z(n))_\tors\\
@V{a^+_\tors}VV @V{\wr}VV\\
\Alb^+(X)(k)_\tors@>\sim>> \LA{1}^*(X)(k)_\tors
\end{CD}\]
where: the horizontal bottom isomorphism is that induced by Theorem \ref{*=-} and the right vertical one comes from the previous Corollary \ref{c14.4}; the left vertical map is the one induced, on torsion, by the universal regular homomorphism $a^+:CH^n_{LW}(X)_{\deg 0}\to \Alb^+(X)(k) $ constructed in \cite[6.4.1]{BSAP}. This would imply that 

\begin{center}
$c\ell^\cdh_\tors$  is an isomorphism $\iff$ $a^+_\tors$  is an isomorphism.
\end{center}

If $X$ is normal and $k = \bar k$ or for any $X$ projective if $k= \C$ then
$a^+_\tors$ is known to be an isomorphism, see \cite{KVS}. For $X$ projective over
any algebraically closed field, see Mallick \cite{mallick}.

We expect that Levine's ``cycle map" $ c\ell^\cdh$ is surjective with uniquely divisible kernel (probably representable by a unipotent group).
\end{remark}

\part*{Appendices}

\appendix

\section{Homological algebra}\label{appendixA}

\subsection{Some comparison lemmas}

The following lemma is probably well-known:

\begin{lemma}\label{lA.2} Let $T:\cT\to\cT'$ be a full triangulated
functor between two triangulated categories. Then $T$ is conservative if
and only if it is faithful.
\end{lemma}

\begin{proof} ``If" is obvious. For ``only if", let $f:X\to X'$ be a
morphism of
$\cT$ such that
$T(f)=0$. Let $g:X'\to X''$ denote a cone of $f$. Then $T(g)$ has a
retraction
$\rho$. Applying fullness, we get an equality $\rho=T(r)$. Applying
conservativity, $u=rg$ is an isomorphism. Then $r'=u^{-1}r$ is a
retraction of $g$, which implies that $f=0$.
\end{proof}

\begin{propose}\label{derxact} a) Let $i: \E \into \A$ be an exact full 
sub-category of an abelian category $\A$, closed under kernels.
Assume further that for each
$A^{\d}\in C^b(\A)$ there exists
$E^{\d}\in  C^b(\E)$ and a quasi-isomorphism  $i (E^{\d})\to A^{\d}$
in $K^b(\A)$.  Then  $i: D^b(\E)\to D^b(\A)$ is an equivalence
of categories.\\
b) The hypothesis of a) is granted when every object in $\A$ has a
finite left resolution by objects in $\E$.
\end{propose}

\begin{proof} a) Clearly, the functor $D^b(\E)\to D^b(\A)$ is
conservative. The assumption implies that
$i$ is essentially surjective: thanks to Lemma \ref{lA.2}, in order to
conclude it remains to see that $i$ is full.

Let $f\in D^b(\A)(i (D^{\d}), i (E^{\d}))$. Since
$D^b(\A)$ has left calculus of fractions there  exists a
quasi-isomorphism $s$ such that $f = f^{\prime} s^{-1}$ where 
$f^{\prime} : A^{\d}\to i (E^{\d})$ is a map in $K^b(\A)$, which
then  lifts to a map in $C^b(\A)$. By hypothesis there exists $F^{\d}\in
C^b(\E)$ and a  quasi-isomorphism  $s^{\prime} :i (F^{\d})\to
A^{\d}$. Set 
$f^{\prime\prime}\df f^{\prime} s^{\prime}: i (F^{\d})\to i (E^{\d})$. 
Then $f = f^{\prime\prime} (s s^{\prime})^{-1}$ where $s s^{\prime} : i 
(F^{\d})\to i (D^{\d})$ is a quasi-isomorphism. By conservativity of $i$,
we are reduced to check fullness for effective maps, \ie arising from true
maps in $C^b(\A)$: this easily follows from the fullness of the
functor $C^b(\E)\into C^b(\A)$. 

b) This follows  by adapting the argument in \cite[I, Lemma~4.6]{HA}.
\end{proof}

\begin{propose}\label{triexact} Let $\E \into \A$ be an exact category. Let
$D$ be a triangulated category and let
$T : D^b(\E)\to D$ be a triangulated functor such that
$$\Hom_{D^b(\E)}(E', E[i]) \longby{\cong} \Hom_{D}(T(E'), T(E[i])) $$
for all $E', E\in \E$ and $i\in \Z$. Then $T$ is fully faithful.
\end{propose}

\begin{proof} Let $C,C'\in C^b(\E)$: we want to show that the map
\[\Hom_{D^b(\E)}(C, C') \to \Hom_{D}(T(C), T(C'))
\]
is bijective. We argue by induction on the lengths of $C$ and $C'$.
\end{proof}

Finally, we have the following very useful criterion for a full embedding of derived
categories, that we learned from Pierre Schapira.

\begin{propose}[\protect{\cite[p. 329, Th. 13.2.8]{KS}}]\label{pschapira} Let $\cA\into \cB$
be an exact full embedding of abelian categories. Assume that, given any mo\-no\-mor\-ph\-ism
$X'\into X$ in $\cB$, with $X'\in \cA$, there exists a morphism $X\to X''$, with $X''\in \cA$,
such that the composite morphism $X\to X''$ is a monomorphism. Then the functor
\[D^*(\cA)\to D^*(\cB)\]
is fully faithful for $*=+,b$.
\end{propose}

\subsection{The $\Tot$ construction}

\begin{lemma}\label{ltot} Let $\cA$ be an abelian category and let
$\cA^{[0,1]}$ be the (abelian) category of complexes of length $1$ of
objects of $\cA$. Then the ``total complex" functor induces a
triangulated functor
\[D^*(\cA^{[0,1]})\to D^*(\cA)\]
for any decoration $^*$.
\end{lemma}

\begin{proof}
We may consider a complex of objects of $\cA^{[0,1]}$ as a double complex
of objects of $\cA$ and take the associated
total complex. This yields a functor
\[\Tot:C^*(\cA^{[0,1]})\to D^*(\cA).\]

(Note that if we consider a complex of objects of $\cA^{[0,1]}$
$$M^{\d} = [L^{\d}\by{u{\d}}G^{\d}]$$
as a map $u\d: L^{\d}\to G^{\d}$
of complexes of $\cA$, then
$\Tot (M^{\d})$ coincides with the cone of $u\d$.)

This functor factors through a triangulated functor from
$D^*(\cA^{[0,1]})$: indeed it is easily checked that a) $\Tot$ preserves
homotopies, b) the induced functor on $K^*(\cA^{[0,1]})$ is
triangulated; c) $\Tot$ of an acyclic complex is
$0$ (which follows from a spectral sequence argument). 
\end{proof}

\begin{lemma}\label{lloc} With notation as in Lemma \ref{ltot}, the set of \qi of $\cA^{[0,1]}$
enjoys a calculus of left and right fractions within the homotopy category
$K(\cA^{[0,1]})$ (same objects, morphisms modulo the homotopy relation).
\end{lemma}

\begin{proof} It is enough to show the calculus of right fractions (for left fractions, replace
$\cA$ by $\cA^{op}$).

a) Let 
\[\begin{CD}
&& \tilde C\\
&&@VuVV\\
D@>f>> C
\end{CD}\]
be a diagram in $\cA^{[0,1]}$. Consider the mapping fibre (= shifted mapping cone) $F$ of the
map $\tilde C\oplus D\to C$. The complex $F$ is in general concentrated in degrees $[0,1,2]$;
however, since $F\to D$ is a \qi, the truncation $\tau_{< 2} F$ is \qi to $F$; then $\tilde
D=\tau_{<2} F$ fills in the square in  $K(\cA^{[0,1]})$.

b) Let 
\[[C^0\to C^1]\by{f} [D^0\to D^1]\by{u} [\tilde D^0\to \tilde D^1]\]
be a chain of maps of $\cA^{[0,1]}$ such that $u$ is a \qi and $uf$ is homotopic to $0$. Let
$s:C^1\to \tilde D^0$ be a corresponding homotopy. Define $\tilde C^1$ as the fibre product of
$C^1$ and $D^0$ over $\tilde D^0$ (via $s$ and $u^0$), $\tilde s:\tilde C^1\to D^0$ the
corresponding map and $\tilde C^0$ the fibre product of $C^0$ and $\tilde C^1$ over $C^1$. One
then checks that $v:[\tilde C^0\to \tilde C^1]\to [C^0\to C^1]$ is a \qi and that $\tilde s$
defines a homotopy from $fv$ to $0$.
\end{proof}

\begin{remark} One can probably extend these two lemmas to complexes of a fixed length $n$ by
the same arguments: we leave this to the interested reader.
\end{remark}

\section{Torsion objects in additive categories} 

\subsection{Additive categories}
\begin{defn}\label{dA.2} Let $\cA$ be an additive category, and let $A$ be
a subring of $\Q$.\\
a) We write $\cA\otimes A$ \index{$\cA\otimes R$} for
the category with the same objects as $\cA$ but morphisms
\[(\cA\otimes A)(X,Y)\df\cA(X,Y)\otimes A.\]
b) We denote by $\cA\{A\}$ the full subcategory of $\cA$:
\[\{X\in \cA\mid \exists n>0 \text{ invertible in } A, n 1_X=0\}.\]
For $A=\Q$, we write $\cA\{A\} =\cA_\tors$. We say that $X\in \cA\{A\}$ is an
\emph{$A$-torsion object} (a torsion object if $A=\Q$).\\ 
c) A morphism $f:X\to Y$ in $\cA$ is an
\emph{$A$-isogeny} (an \emph{isogeny} if $A=\Q$) if there exists a morphism $g:Y\to X$ and an integer $n$ invertible
in $A$ such that $fg= n 1_Y$ and $gf=n1_X$. We denote by $\Sigma_A(\cA)$ the collection of
$A$-isogenies of $\cA$.\\ d) We say that two objects $X,Y\in \cA$ are \emph{$A$-isogenous} if they can be linked by a chain of $A$-isogenies (not necessarily pointing in the same direction).
\end{defn}

\begin{propose}\label{pB.1.2} a) The subcategory $\cA\{A\}$ is additive and closed under direct
summands.\\ b) The $A$-isogenies $\Sigma_A(\cA)$ form a multiplicative system of morphisms in
$\cA$, enjoying calculi of left and right fractions. The corresponding localisation of $\cA$ is
isomorphic to
$\cA\otimes A$.
\end{propose}

\begin{proof} a) is clear. For b), consider the obvious functor $P:\cA\to \cA\otimes A$. We
claim that
\[\Sigma_A(\cA)=\{f\mid P(f)\text{ is invertible.}\}\]

One inclusion is clear. Conversely, let $f:X\to Y$ be such that $P(f)$ is invertible. This
means that there exists $\gamma\in (\cA\otimes A)(Y,X)$ such that $P(f)\gamma =1_Y$ and $\gamma
P(f) = 1_X$. Choose an integer $m\in A -\{0\}$ such that $m\gamma = P(g_1)$ for some $g_1$.
Then there is another integer $n\in A-\{0\}$ such that
\[n(fg_1 - m1_Y) = 0\text{ and } n(g_1 f - m1_X) = 0.\]

Taking $g = ng_1$ shows that $f\in \Sigma_A(\cA)$.

It is also clear that homotheties by nonzero integers of $A$ form a cofinal system in
$\Sigma_A(\cA)$. This shows immediately that we have calculi of left and right fractions.

It remains to show that the induced functor
\[\Sigma_A(\cA)^{-1}\cA\to \cA\otimes A\]
is an isomorphism of categories; but this is immediate from the well-known formula, in the
presence of calculus of fractions:
\[\Sigma_A(\cA)^{-1}\cA(X,Y) = \varinjlim_{X'\by{f} X\in \Sigma} \cA(X,Y) =\varinjlim_{X\by{n}
X, n\in A-\{0\}} \cA(X,Y).\]
\end{proof}

The following lemma is clear.

\begin{lemma}\label{lB.1.3} Let $\cB$ be a full additive subcategory of
$\cA$, and suppose that every object of $\cB$ is $A$-isogenous to an object of $\cA$. Then
$\cB\otimes A\iso \cA\otimes A$.\qed
\end{lemma}

\subsection{Triangulated categories} (See \cite[A.2.1]{riou} for a different treatment.)

\begin{propose}\label{p1.1} Let $\cT$ be a triangulated category. Then\\
a) The subcategory $\cT\{A\}$ is triangulated and thick.\\ 
b) Let $X\in \cT$ and $n\in A-\{0\}$. Then ``the" cone $X/n$ of multiplication by $n$
on
$X$ belongs to $\cT\{A\}$.\\   
c) The localised category
$\cT/\cT\{A\}$ is canonically isomorphic to
$\cT\otimes A$. In particular, $\cT\otimes A$ is triangulated.\\
d) A morphism $f$ of $\cT$ belongs to $\Sigma_A(\cT)$ if and only if $cone(f)\in
\cT\{A\}$.
\end{propose}

\begin{proof} a) It is clear that $\cT\{A\}$ is stable under direct
summands; it remains to see that it is triangulated. Let $X,Y\in
\cT\{A\}$, $f:X\to Y$ a morphism and $Z$ a cone of $f$. We may assume
that $n1_X=n1_Y=0$. The commutative diagram
\[\begin{CD}
Y@>>> Z@>>> X[1]\\
@V{n=0}VV @V{n}VV @V{n=0}VV \\
Y@>>> Z@>>> X[1]
\end{CD}\]
show that multiplication by $n$ on $Z$ factors through $Y$; this implies
that $n^2 1_Z=0$.

b) Exactly the same argument as in a) shows that multiplication by $n$ on
$X/n$ factors through $X$, hence that $n^2 1_{X/n}=0$.

c) Let $f\in \cT$ be such that $C:=cone(f)\in \cT\{A\}$, and let $n>0$
be such that $n 1_C=0$. The same trick as in a) and b) shows that there
exist factorisations $n=ff'=f''f$, hence that $f\in \Sigma_A(\cT)$. In particular, $f$ becomes
invertible under the canonical (additive) functor $\cT\to \cT\otimes A$. Hence an
induced (additive) functor
\[\cT/ \cT\{A\}\to \cT\otimes A\]
which is evidently bijective on objects; b) shows immediately that it is
fully faithful.

d) One implication has been seen in the proof of c). For the other, if $f\in \Sigma_A(\cT)$,
then $f$ becomes invertible in $\cA\otimes \Q$, hence $cone(f)\in \cT\{A\}$ by c).
\end{proof}

\begin{remark} As is well-known, the stable homotopy category gives a
counterexample to the expectation that in fact $n 1_{X/n}=0$ in b)
($X=S^0$,
$n=2$).
\end{remark}

\begin{lemma}\label{lB.2.3} Let $0\to \cT'\to \cT\to \cT''\to 0$ be a short exact sequence of
triangulated categories (by definition, this means that $\cT'$ is thick in $\cT$ and that
$\cT''$ is equivalent to $\cT/\cT'$). Then the sequence
\[0\to \cT'\otimes A\to \cT\otimes A\to \cT''\otimes A\to 0\]
is exact.
\end{lemma}

\begin{proof} We have to show that the functor
\[a:\frac{\cT\otimes A}{\cT'\otimes A}\to \cT''\otimes A\]
is an equivalence of categories. Since the left hand side is $A$-linear, the natural functor
\[\cT/\cT'\to \cT\otimes A/\cT'\otimes A\]
canonically extends to a functor
\[b:(\cT/\cT')\otimes A\to \cT\otimes A/\cT'\otimes A.\]

It is clear that $a$ and $b$ are inverse to each other. 
\end{proof}

\begin{propose}\label{p1.2} Let $T:\cS\to \cT$ be a triangulated functor
between triangulated categories. Then $T$ is fully faithful if and only
if the induced functors $T\{A\}:\cS\{A\}\to\cT\{A\}$ and
$T\otimes A:\cS\otimes A\to
\cT\otimes A$ are fully faithful. 
\end{propose}

\begin{proof} ``Only if" is obvious; let us prove ``if'. Let $X,Y\in
\cS$: we have to prove that $T:\cS(X,Y)\to \cT(T(X),T(Y))$ is bijective.
We do it in two steps:

1) $Y$ is torsion, say $n 1_Y=0$. The claim follows from the commutative
diagram with exact rows
\[\begin{CD}
\scriptstyle\cS(X/n[1],Y)&\to&\scriptstyle \cS(X,Y)@>n=0>>\scriptstyle
\cS(X,Y)&\to&
\scriptstyle\cS(X/n,Y)\\ 
 @V{T}V{\wr}V @V{T}VV @V{T}VV @V{T}V{\wr}V \\
\scriptstyle\cT(T(X)/n[1],T(Y))&\to&
\scriptstyle\cT(T(X),T(Y))@>n=0>>\scriptstyle\cT(T(X),T(Y))&\to&
\scriptstyle\cT(T(X)/n,T(Y))
\end{CD}\]
and the assumption (see Proposition \ref{p1.1} b)).

2) The general case. Let $n>0$. We have a commutative diagram with exact
rows
\[\begin{CD}
\scriptstyle 0 &\to&\scriptstyle \cS(X,Y)/n&\to&
\scriptstyle\cS(X,Y/n)&\to& \scriptstyle{}_n\cS(X,Y[1])&\to& \scriptstyle
0 \\  && @V{T}VV @V{T}V{\wr}V @V{T}VV \\
\scriptstyle 0 &\to& \scriptstyle\cT(T(X),T(Y))/n&\to&
\scriptstyle\cT(T(X),T(Y)/n)&\to& \scriptstyle {}_n\cT(T(X),T(Y)[1])&\to&
\scriptstyle 0 
\end{CD}\]
where the middle isomorphism follows from 1). The snake lemma yields an
exact sequence
\begin{multline*}
0\to \cS(X,Y)/n\longby{T} \cT(T(X),T(Y))/n\\
\to
{}_n\cS(T(X),T(Y)[1])\longby{T}{}_n\cT(T(X),T(Y)[1])\to 0.
\end{multline*}

Passing to the limit over $n$, we get another exact sequence
\begin{multline}\label{eqB.1}
0\to \cS(X,Y)\otimes A/\Z\longby{T} \cT(T(X),T(Y))\otimes A/\Z\\
\to
\cS(T(X),T(Y)[1])\{A\}\longby{T}\cT(T(X),T(Y)[1])\{A\}\to 0.
\end{multline}

Consider now the commutative diagram with exact rows
\[\begin{CD}
\scriptstyle 0&\to& \scriptstyle \cS(X,Y)\{A\}&\to&
\scriptstyle \cS(X,Y)&\to&\scriptstyle
\cS(X,Y)\otimes A&\to&\scriptstyle \cS(X,Y)\otimes A/\Z&\to& 0\\
&& @V{T}V\underline{1}V @V{T}V\underline{2}V @V{T}V{\wr}V @V{T}V\underline{4}V\\
\scriptstyle 0&\to& \scriptstyle \cT(T(X),T(Y))\{A\}&\to&
\scriptstyle \cT(T(X),T(Y))&\to&\scriptstyle
\cT(T(X),T(Y))\otimes A&\to&\scriptstyle \cT(T(X),T(Y))\otimes A/\Z&\to&
0
\end{CD}\]
where the isomorphism is by assumption. By this diagram and \eqref{eqB.1}, $\underline{4}$ is
an isomorphism. Using this fact in \eqref{eqB.1} applied with $Y[-1]$, we get that
$\underline{1}$ is an isomorphism; then $\underline{2}$ is an isomorphism by the 5
lemma, as desired.
\end{proof}

\subsection{Torsion objects in an abelian category} The proof of the following proposition is
similar to that of Proposition \ref{p1.1} and is left to the reader.

\begin{propose}\label{p1.1ab} Let $\cA$ be an abelian category. Then\\
a) The full subcategory $\cA\{A\}$
is thick (a Serre subcategory, in another terminology).\\ 
b) Let $X\in \cA$ and $n>0$
invertible in $A$. Then the kernel and cokernel of multiplication by
$n$ on $X$ belong to $\cA\{A\}$.\\  
c) The localised category
$\cA/\cA\{A\}$ is canonically isomorphic to
$\cA\otimes A$. In particular, $\cA\otimes A$ is abelian.\\
d) A morphism $f\in \cA$ is in $\Sigma_A(\cA)$ if and only if $\ker f\in \cA\{A\}$ and
$\coker f\in \cA\{A\}$.\qed
\end{propose}

The following corollary is a direct consequence of Proposition \ref{p1.1ab} and Lemma
\ref{lB.1.3}:

\begin{cor}\label{cB.1.3} Let $\cA$ be an abelian category. Let $\cB$ be a full additive
subcategory of
$\cA$, and suppose that every object of $\cB$ is $A$-isogenous to an object of $\cA$ (see
Definition
\ref{dA.2}). Then $\cB\otimes A$ is abelian, and in particular idempotent-complete.\qed
\end{cor}

\subsection{Abelian and derived categories}

\begin{propose}\label{pB.4.1} Let $\cA$ be an abelian category. Then the natural functor
$D^b(\cA)\to D^b(\cA\otimes A)$ induces an equivalence of categories
\[D^b(\cA)\otimes A\iso D^b(\cA\otimes A).\]
In particular, $D^b(\cA)\otimes A$ is idempotent-complete.
\end{propose}

\begin{proof} In 3 steps:

1) The natural functor $C^b(\cA)\otimes A\to C^b(\cA\otimes A)$ is an equivalence of
categories. Full faithfulness is clear. For essential surjectivity, take a bounded complex $C$
of objects of $\cA\otimes A$. Find a common denominator to all differentials involved in $C$.
Then the corresponding morphisms of $\cA$ have torsion composition; since they are finitely
many, we may multiply by a common bigger integer so that they compose to $0$. The resulting
complex of $C^b(\cA)$ then becomes isomorphic to $C$ in $C^b(\cA\otimes A)$.

2) The functor of 1) induces an equivalence of categories $K^b(\cA)\otimes A\iso
K^b(\cA\otimes A)$. Fullness is clear, and faithfulness is obtained by the same technique as
in a).

3) The functor of 2) induces the desired equivalence of categories. First, the functor
\[D^b(\cA)/D^b_{\cA\{A\}}(\cA)\to D^b(\cA/\cA\{A\})\]
is obviously conservative. But clearly $D^b_{\cA\{A\}}(\cA)=D^b(\cA)\{A\}$. Hence,
by Propositions \ref{p1.1} and \ref{p1.1ab}, this functor translates as
\[D^b(\cA)\otimes A\to D^b(\cA\otimes A).\]

Let $A^b(\cA)$ denote the thick subcategory of $K^b(\cA)$ consiting of acyclic complexes. By
Lemma \ref{lB.2.3} we have a commutative diagram of exact sequences of triangulated categories
\[\begin{CD}
0@>>> A^b(\cA)\otimes A@>>> K^b(\cA)\otimes A@>>> D^b(\cA)\otimes A@>>> 0\\
&&@VVV @VVV @VVV\\
0@>>> A^b(\cA\otimes A)@>>> K^b(\cA\otimes A)@>>> D^b(\cA\otimes A)@>>> 0.
\end{CD}\]

We have just seen that the right vertical functor is conservative, and by b), the middle one
is an equivalence. Hence the left one is essentially surjective, and the result follows. 
\end{proof}

\section{$1$-motives with torsion}\label{AppendixB}

Effective $1$-motives which admit torsion are introduced in \cite[\S
1]{BRS} (in characteristic $0$). We investigate some properties (over a perfect field of  exponential characteristics $p>1$) which are not included in op. cit. as a supplement  to our Sect. 1.

\subsection{Effective $1$-motives}
An \emph{effective $1$-motive with torsion} over $k$ is a complex of group 
schemes $M= [ L \by{u} G]$ where $L$ is finitely generated locally constant for the
{\'e}tale  topology \ie a discrete sheaf of Def.~\ref{d1.1.1}, and
$G$ is a  semi-abelian $k$-scheme. Therefore
$L$ can be represented by an extension 
$$0\to L_{\tor}\to L \to L_{\fr}\to 0$$ where $L_{\tor}$ is a finite \'etale $k$-group scheme and $L_{\fr}$ is free, \ie a lattice. Also $G$ can be represented 
by an extension of an abelian $k$-scheme $A$ by a $k$-torus $T$.

\begin{defn}\label{eff1mot}
 An {\it effective} map from $M = [ L \by{u} G]$ to $M'= [ L' \by{u'} 
G']$ is a commutative square
\[ \begin{CD}
 L @>u>>  G\\
@V{f}VV  @V{g}VV  \\
L'@>u'>> G'
\end{CD}\]
in the category of group schemes. Denote by $(f, g):M\to M'$ such a
map. The natural composition of squares makes  up a category, denoted by
${}^t\M^{\eff}$.  We will denote by $\Hom_{\eff}(M, M')$ the abelian group
of effective  morphisms.\index{${}^t\M$, ${}^t\M^\eff$}
\end{defn}

For a given $1$-motive $M = [L \by{u} G]$  we have (in the abelian 
category of commutative group schemes) a commutative diagram 
\begin{equation}\label{basic}
\begin{CD}
&&&&0&  &0 \\
&&&&@V{}VV @V{}VV  \\
0@>>> \ker (u)\cap L_{\tor}@>>>  L_{\tor}@>{u}>> u(L_{\tor})@>>> 0\\
&&@V{}VV  @V{}VV @V{}VV  \\
0@>>> \ker (u)@>>> L @>{u}>> G\\
&&&&@V{}VV  @V{}VV  \\
&&&&L_{\fr} @>{\bar u}>> G/u(L_{\tor})\\
&&&&@V{}VV  @V{}VV  \\
&&&&0&  &0 
\end{CD}
\end{equation}
with exact rows and columns. We set
\begin{itemize}
\item $M_{\fr}\df [L_{\fr}  \by{\bar u}  G/u(L_{\tor})]$
\item $M_{\tor}\df [\ker (u)\cap L_{\tor}\to 0]$
\item $M_{\tf}\df [L/\ker (u)\cap L_{\tor}  \by{u} G]$
\end{itemize}
considered as effective $1$-motives. From Diagram \eqref{basic}
there  are canonical effective maps $M \to M_{\tf}$, $M_{\tor}\to
M$ and
$M_{\tf} 
\to M_{\fr}$.
\index{$M_{\fr}$, $M_{\tor}$, $M_{\tf}$}
\begin{defn}\label{frtortf}
A $1$-motive $M = [L \by{u} G]$ is {\it free}\, if $L$ is free, \ie if 
$M = M_{\fr}$. $M$ is  {\it torsion}\, if $L$ is torsion and $G =0$, \ie 
if $M=M_{\tor}$, and 
{\it torsion-free}\, if $\ker (u)\cap L_{\tor} =0$, \ie if $M=M_{\tf}$. 

Denote by ${}^t\M^{\eff ,\fr}$, ${}^t\M^{\eff ,\tor}$ and 
${}^t\M^{\eff ,\tf}$, the full sub-categories of 
${}^t\M^{\eff}$ given by free, torsion and torsion-free $1$-motives
respectively.
\end{defn}

The category ${}^t\M^{\eff ,\fr}$ is nothing else than the category $\M$ of Deligne $1$-motives
and we shall henceforth use this notation. It is clear that ${}^t\M^{\eff,\tor}$  is equivalent
to the category of finite \'etale group schemes. If $M$ is torsion-free then Diagram \eqref{basic} is a
pull-back, \ie $L$  is the pull-back of $L_{\fr}$ along the isogeny $G\to G/L_{\tor}$.

\begin{propose} \label{lim} The categories ${}^t\M^{\eff}$ and $\M$ have all finite limits and
colimits. 
\end{propose}

\begin{proof} Since these are additive categories (with biproducts), it is
enough to show that they have kernels, dually cokernels. Now let
$\phi=(f,g):M\to M'$ be an effective map. We claim that 
$$\ker \phi = [\ker^0(f)  \by{u}  \ker^0(g)]$$
is the pull-back of $\ker^0(g)$ along $u$, where $ \ker^0(g)$ is the (reduced)
connected component of the identity of the kernel of $g : G\to G'$ and
$\ker^0(f) \subseteq
\ker(f)$. We have to show that the following diagram of effective
$1$-motives
\[
\begin{CD}
0&& 0 \\
@V{}VV @V{}VV  \\
\ker^0(f) @>{u}>>  \ker^0(g)\\
@V{}VV @V{}VV  \\
L @>{u}>> G\\
@V{f}VV  @V{g}VV  \\
L' @>{u'}>> G'\\
\end{CD}
\]
satisfies the universal property for kernels. Suppose that $M''= [L'' 
\by{u''}  G'']$ is mapping to $M$ in such a way that the composition 
$M'' \to M \to M'$ is the zero map. Then $L''$ maps to $\ker (f)$ and
$G''$ maps to $\ker (g)$. Since $G''$ is connected, it actually maps
to $\ker^0(g)$ and, by the universal property of pull-backs in the
category of group schemes, $L''$ then maps to $\ker^0(f)$. Finally note
that if $L$ is free then also $\ker^0(f)$ is free. 

For cokernels, we see that 
$$[\coker (f)  \by{\bar u'}  \coker(g)]$$
is an effective $1$-motive which is clearly a cokernel of $\phi$.

For $\M$, it is enough to take the free part of the
cokernel, \ie  given $(f, g) : M \to M'$ then $[\coker (f) \to \coker
(g)]_{\fr}$ meets the universal property
for coker of free $1$-motives.
\end{proof}

\subsection{Quasi-isomorphisms} (\cf \cite[\S 1]{BRS}). 

\begin{defn} An effective morphism of $1$-motives $M \to M'$, here $M = [L \by{u} G]$ and $M' = [L'
\by{u} G']$, is a  {\it quasi-isomorphism}\, (\qi for short) of $1$-motives if it yields a
pull-back diagram
\begin{equation}\label{qi1mot}
\begin{CD}
0&  &0 \\
@V{}VV @V{}VV  \\
  F @= F\\
@V{}VV @V{}VV  \\
 L @>{u}>>  G\\
@V{}VV @V{}VV  \\
L' @>{u'}>> G'\\
@V{}VV@V{}VV  \\
0&  &0 
\end{CD}
\end{equation}
where $F$ is a finite \'etale group.
\end{defn}

\begin{remarks}\label{qi=qi}
1) Note that kernel and cokernel of a quasi-iso\-mor\-phism of
$1$-motives are $0$ but, in general, a quasi-isomorphism is not an
isomorphism in
${}^t\M^{\eff}$. Hence the category ${}^t\M^{\eff}$ is not abelian.\\
2) A \qi of $1$-motives $M \to M'$ is actually a \qi of complexes of group
schemes. In fact, an effective map of $1$-motives $M \to M'$ is a \qi of
complexes if and only if we have the following diagram
\[ \begin{CD}
0&\to& \ker (u)@>>> L @>{u}>> G@>>> \coker (u)&\to& 0\\
&&\veq &&@V{}VV@V{}VV\veq  \\
0&\to & \ker (u')@>>> L' @>{u'}>> G'@>>> \coker (u')&\to& 0.\\
\end{CD}\]

Therefore $\ker$ and $\coker$ of $L\to L'$ and $G\to G'$ are equal. Then 
$\coker (G\to G')= 0$, since it is connected and discrete,
and $\ker (G\to G')$ is a finite group. Conversely, Diagram
\eqref{qi1mot} clearly yields a \qi of complexes. In particular, it
easily follows that the class of \qi of $1$-motives is closed under
composition of effective morphisms.
\end{remarks}

\begin{propose}\label{B1.1} Quasi-isomorphisms are simplifiable on the
left and on the right.
\end{propose}

\begin{proof} The assertion ``on the right" is obvious since the two
components of a \qi are epimorphisms. For the left, let $\phi=(f,g): M\to
M'$ and  $\sigma=(s,t):M' \to\tilde M$ a \qi such that $\sigma\phi=0$.
In the diagram
\[\begin{CD}
 L  @>{u}>>  G\\
@V{f}VV  @V{g}VV  \\
L' @>{u'}>> G'\\
@V{s}VV   @V{s}VV  \\
 \tilde L @>{\tilde u}>> \tilde G
\end{CD}\]
we have $ \tilde L = L'/F$, $\tilde G = G'/F$, for some finite group $F$,
$\im (f)\subseteq F$ and $\im (g) \subseteq F$. Now $u'$ restricts to the
identity on $F$ thus  $\im (f)\subseteq \im (g)$  
and $\im (g) =0$, since $\im (g)$ is connected, hence $\phi=0$.  
\end{proof}

\begin{propose}\label{calfrac} The class of \qi admits a calculus of right
fractions in the sense of (the dual of) \cite[Ch. I, \S 2.3]{GZ}.
\end{propose}

\begin{proof} By \cite[Lemma 1.2]{BRS}, the first condition of calculus of right fractions
is verified, and Proposition \ref{B1.1} shows that the second one is
verified as well.  (Note that we here only consider isogenies with \'etale kernel.)
\end{proof}

\begin{remark} The example of the diagram
\[\begin{CD}
[L\to G]@>\sigma>> [L'\to G']\\
@V{(1,0)}VV \\
[L\to 0]
\end{CD}\]
where $\sigma$ is a nontrivial \qi shows that calculus of left fractions
fails in general.
\end{remark}

\begin{lemma}\label{fact} Let $s,t,u$ be three maps in ${}^t\M^{\eff}$, with
$su=t$. If $s$ and $t$ are $\qi$, then so is $u$.\qed
\end{lemma}

\begin{proof} Consider the exact sequence of complexes of sheaves
\[0\to\ker u\to\ker t\to \ker s\to \coker u\to\coker t\to\coker s\to 0.\]

Since $s$ and $t$ are $\qi$, the last two terms are $0$. Hence $\coker
(u) =[L\to G]$ is a quotient of $\ker (s)$; since $G$ is connected, we must
have $G=0$. On the other hand, as a cokernel of a map of acyclic
complexes of length $1$, $\coker (u)$ is acyclic, hence $L=0$. Similarly,
$\ker (u)$ is acyclic.
\end{proof}

\subsection{$1$-motives} We now
define the category of $1$-motives with torsion from ${}^t\M^{\eff}$
by formally inverting quasi-isomorphisms.

\begin{defn}\label{1tors} The category ${}^t\M$ of  \emph{$1$-motives with torsion} is the localisation of ${}^t\M^{\eff}$ with respect to the multiplicative class $\{\qi\}$ of
quasi-isomor\-phisms. \index{${}^t\M$, ${}^t\M^\eff$}
\end{defn}

\begin{remark}
Note that there are no nontrivial \qi between free (or torsion)
$1$-motives. However, the canonical map  $M_{\tf}  \to M_{\fr}$ is
a quasi-isomorphism (it is an effective isomorphism when
$u(L_{\tor})=0$). 
\end{remark}

It follows from Proposition \ref{calfrac} that the Hom sets in ${}^t\M$ are
given by the formula
$$\Hom (M, M') = \limdir{\rm q.i.} \Hom_{\eff}  (\tilde M, M')$$
where the limit is taken over the filtering set of all quasi-isomorphisms
$\tilde M\to M$. A morphism of $1$-motives $M\to M'$ can be represented as
a diagram 
$$ \begin{array}{ccc}
 M\ &  &\ M'\\
{\scriptstyle {\rm q.i.}}\nwarrow & &  \nearrow {\scriptstyle {\rm eff}}
\\ & \tilde M & 
\end{array}$$
and the composition is given by the existence of a $\hat M $
making the following diagram
$$ \begin{array}{ccccc}
 M\ &  &\ M'\ &&\ M''\\
{\scriptstyle {\rm q.i.}}\nwarrow & &  \nearrow \ \ 
\nwarrow & &  \nearrow {\scriptstyle {\rm eff}}\\
& \tilde M & & \tilde M'&\\
&{\scriptstyle {\rm q.i.}}\nwarrow & &  \nearrow {\scriptstyle {\rm
eff}}&\\ &  &\hat M & &
\end{array}$$
commutative. (This $\hat M$ is in fact unique, see \cite[Lemma 1.2]{BRS}.)

\subsection{Strict morphisms}
The notion of strict morphism is essential in order to show that the
$\Z[1/p]$-linear category of $1$-motives is abelian (\cf \cite[\S 1]{BRS}).

\begin{defn}  We say that an effective morphism $(f,g): M\to M'$ is
\emph{strict} if we have 
\[\ker (f, g)= [\ker (f)\to\ker (g)]\]
\ie if $\ker (g)$ is (connected) semiabelian.
\end{defn}

To get a feeling on the notion of strict morphism, note:

\begin{lemma} Let $\phi=(f,g):M\to N$ be a strict morphism, with $g$ onto.
Suppose that $\phi=\sigma\tilde\phi$, where $\sigma$ is a \qi Then
$\sigma$ is an isomorphism.
\end{lemma}

Conversely, we obtain: 

\begin{propose}[\protect{\cite[Prop. 1.3]{BRS}}]\label{pstrict}
Any effective morphism
$\phi \in \Hom_{\eff}(M,M')$ can be factored as follows
\begin{equation}\label{strict}
\begin{array}{c}
M  \longby{\phi} M'\\
\scriptstyle \tilde\phi\displaystyle
\searrow\hspace*{0.5cm} \nearrow\\
\tilde M
\end{array}
\end{equation}
where  $\tilde
\phi$ is a strict morphism and
$\tilde M \to M'$ is a \qi or a $p$-power isogeny.
\end{propose}
\begin{proof} {\it (Sketch)}\,
Note that if $\phi  = (f,g)$ we always have the following natural
factorisation of the map $g$ between semi-abelian schemes
$$\begin{array}{c}
G  \longby{g} G'\\
\searrow\hspace*{0.5cm} \nearrow\\
G/\ker^0(g)
\end{array}$$
If $g$ is a surjection we get the claimed factorisation
by taking $\tilde M = [ \tilde L \to G/\ker^0(g)]$ where $\tilde L$ is the
pull-back of $L'$, the lifting of $f$ is granted by the universal
property of pull-backs. In general, we can extend the so obtained isogeny  
on the image of $g$ to an isogeny of  $G'$ (see the proof of Prop.
1.3 in \cite{BRS} for details).
\end{proof}

\begin{lemma}\label{strictqi}
Let
\[\begin{CD}
M'@>f>> M\\
@A{u}AA @A{t}AA \\
N' @>h>> N 
\end{CD}\]
be a commutative diagram in ${}^t\M^{\eff}$, where $f$ is strict and $u,t$ are
\qi Then the induced map $v:\coker (h)\to \coker (f)$ is a \qi
\end{lemma}

\begin{proof} In all this proof, the term ``kernel" is taken in the
sense of kernel of complexes of sheaves. Let $K$ and
$K'$ be the kernels of $f$ and
$h$ respectively:
\[\begin{CD}
0@>>> K@>>>M'@>f>> M@>>> \coker (f)\\
&&@A{w}AA @A{u}AA @A{t}AA @A{v}AA \\
0@>>> K'@>>> N' @>h>> N@>>> \coker (h) 
\end{CD}\]

By a diagram chase, we see that $v$ is onto and that $\ker t\to \ker v$
is surjective. To conclude, it will be sufficient to show that the
sequence of complexes 
\begin{equation}\label{eq1}
\ker (u)\to \ker (t)\to\ker (v)
\end{equation}
is exact termwise. For this, note that the second component of $w$ is onto
because $f$ is strict and by dimension reasons. This implies by a diagram
chase that the second component of \eqref{eq1} is exact. But then the
first component has to be exact too.
\end{proof}

\subsection{Exact sequences of $1$-motives}
We have the following basic properties of $1$-motives. 

\begin{propose} \label{faith} The canonical functor $${}^t\M^{\eff}\to {}^t\M$$
is left exact and faithful. 
\end{propose}

\begin{proof} Faithfulness immediately follows from Proposition
\ref{B1.1}, while left exactness follows from Proposition
\ref{calfrac} and (the dual of) \cite[Ch. I, Prop. 3.1]{GZ}.
\end{proof}

\begin{lemma}\label{coker} Let $f:M'\to M$ be an effective map. 
\begin{enumerate}
\item The canonical projection $\pi : M \to \coker (f)$ remains an
epimorphism  in ${}^t\M[1/p]$.
\item If $f$ is strict then $\pi$  remains a cokernel in ${}^t\M[1/p]$.
\item  Cokernels exist in ${}^t\M[1/p]$. 
\end{enumerate}
\end{lemma}

\begin{proof} To show Part 1 let $\pi: M \to  N$ be an effective
map.  One sees immediately that $\pi$ is epi in ${}^t\M[1/p]$ if and only if for
any commutative diagram
\[\begin{CD}
M@>\pi>> N\\
@A{s'}AA@A{s}AA \\
Q'@>\pi'>> Q  
\end{CD}\]
with $s,s'$ \qi, the map $\pi'$ is an epi in the effective category. Now
specialise to the case $N = \coker (f)$ and remark that (up to modding
out by $\ker f$) we may assume $f$ to be a monomorphism as a map of
complexes, thus strict. Take $\pi',s,s'$ as above. We have a commutative
diagram of effective maps
\[\begin{CD}
0@>>> M'@>f>> M&@>\pi>> &\coker (f)\\
&&&&&&&\scriptstyle t\displaystyle\nearrow\\
&& @A{s''}AA @A{s'}AA\coker(f')&&@A{s}AA \\
&&&&&\displaystyle\nearrow&&\scriptstyle u\displaystyle\searrow\\
&& Q'' @>f'>> Q'&@>\pi'>> &Q.
\end{CD}\]

\begin{sloppypar}
Here $s''$ is a \qi and $Q'',f',s''$ are obtained by calculus of right
fractions (Proposition \ref{calfrac}). By Lemma \ref{strictqi}, the
induced map  $t:\coker (f') \to \coker (f)$ is a \qi By Proposition
\ref{B1.1}, $\pi'f'=0$, hence the existence of $u$. By Lemma \ref{fact},
$u$ is a \qi. Hence $\pi'$ is a composition of two epimorphisms and
Part 1 is proven.
\end{sloppypar}

To show Part 2, let $gt^{-1}:M
\to M''$ be such that the composition  $M' \to M''$ is zero. By
calculus of right fractions we have a commutative diagram
\[\begin{CD}
M'@>f>> M@>\pi>> \coker (f)\\
@A{u}AA @A{t}AA @A{v}AA\\
N''' @>h>> N'' @>>> \coker (h)\\ 
&&&\scriptstyle{g}\searrow& @V{}VV\\ 
&&&&M''
\end{CD}\]
where all maps are effective and $u$ is a \qi. As above
we have $gh=0$, hence the factorisation of $g$ through $\coker (h)$.
Moreover $\coker (h)$ maps canonically to $\coker (f)$ via a map $v$ (say),
which is a
\qi by Lemma \ref{strictqi}. This shows that $gt^{-1}$ factors
through $\coker (f)$ in ${}^t\M[1/p]$.
Uniqueness of the factorisation is then granted by Part 1.

In a category, the existence of cokernels is
invariant by left or right composition by isomorphisms, hence Part 3 is a
consequence of Parts 1 and 2 via Proposition \ref{pstrict}.
\end{proof}

Now we can show the following key result (\cf \cite[Prop.~1.3]{BRS}).

\begin{thm} \label{1mtora}
The category ${}^t\M[1/p]$ is abelian.
\end{thm}

\begin{proof} Existence and description of kernels follows
from Propositions \ref{lim}, \ref{calfrac}, \ref{faith} and (the dual of)
\cite[Ch. I, Cor. 3.2]{GZ}, while existence of cokernels has been proven
in Lemma \ref{coker}.
We are then left to show that, for any (effective) strict map $\phi: M\to
M'$, the canonical effective morphism  from the coimage of $\phi$ to the
image of $\phi$ is a \qi of $1$-motives, \ie the canonical
morphism
\begin{equation}\label{coimage}
\coker (\ker \phi \to M)\to \ker (M'\to \coker \phi) 
\end{equation}
is a quasi-isomorphism. Since we can split $\phi$ in two short
exact sequences of complexes in which each term is an effective
$1$-motive we see that \eqref{coimage} is even a isomorphism in ${}^t\M^{\eff}[1/p]$.

\end{proof}
\begin{remark} Note that (even in characteristic zero) for a given non-strict effective map $(f,g):
M\to M'$ the effective morphism \eqref{coimage} is not a \qi of
$1$-motives. In fact, the following diagram
\[
\begin{CD}
0&  &0 \\
@V{}VV @V{}VV  \\
  \ker (f)/\ker^0(f) & {\subseteq} & \ker (g)/\ker^0(g)\\
@V{}VV @V{}VV  \\
 L/\ker^0(f) @>>> G/\ker^0(g)\\
@V{}VV @V{}VV  \\
\im (f) @>>> \im (g)\\
@V{}VV @V{}VV  \\
0&  &0 
\end{CD}
\]
is not a pull-back, in general. For example, let $g:G\to G'$ be with finite
kernel and a proper sub-group $F\subsetneq\ker (g)$, and consider 
$$ (0,g): [F\to G] \to [0\to G'].$$
\end{remark} 

\begin{cor} \label{corexseq} A short exact sequence of $1$-motives in ${}^t\M[1/p]$ 
\begin{equation} \label{exseq}
0\to M'\to M \to M''\to 0
\end{equation}
can be represented up to isomorphisms by a strict effective epimorphism
$(f,g) : M\to M''$ with kernel $M'$, \ie by an exact sequence of complexes.
\end{cor}
\begin{example} 
Let $M$ be a $1$-motive with torsion. We then always have a canonical
exact sequence in ${}^t\M[1/p]$
\begin{equation}\label{shortfree}
0\to M_{\tor}\to M \to M_{\fr}\to 0
\end{equation}
induced by \eqref{basic}, according to Definition~\ref{frtortf}. 
Note that in the following canonical factorisation
$$\begin{array}{c}
M  \longby{} M_{\fr}\\
\searrow\hspace*{0.4cm} \nearrow \\
M_{\tf}
\end{array}$$
the effective map $M \to M_{\tf}$ is a strict epimorphism with kernel
$M_{\tor}$ and $M_{\tf}\to M_{\fr}$ is a \qi (providing an example of
Proposition~\ref{pstrict}).
\end{example}

\subsection{$\ell$-adic realisation} 
Let $n: M\to M$ be the (effective) multiplication by $n$ on a $1$-motive
$M =[L\by{u} G]$ over a field $k$ where $n$ is prime to the characteristic of
$k$. It is then easy to see, \eg by the description of kernels in
Proposition~\ref{lim}, that 
$${}_nM \df \ker (M\longby{n} M) = [\ker (u)\cap {}_nL \to 0].$$
Thus ${}_nM = 0$ (all $n$ in characteristic zero)  if and only if $M$ is
torsion-free, \ie
$M_{\tor} =0$. Moreover, by Proposition~\ref{pstrict} and 
Lemma~\ref{coker}, we can see that
$$M/n\df \coker (M\longby{n} M)$$ is always a torsion $1$-motive. If $L=0$,
let simply $G$ denote, as usual, the $1$-motive $[0\to G]$. Then we
get an extension in $\M[1/p]$
\begin{equation}\label{ngexseq}
0\to G \longby{n} G \to {}_nG[1]\to 0
\end{equation}
where ${}_nG[1]$ is the torsion $1$-motive $[{}_nG\to 0]$. If $L \neq 0$
then $M/n$ can be regarded as an extension of $L/n$ by $\coker ({}_nL\to
{}_nG)$, \eg also by applying the snake lemma to the multiplication by $n$
on the following canonical short exact sequence (here $L[1] = [L\to 0]$ as
usual)
\begin{equation}\label{stexseq}
0\to G \to M \to L[1]\to 0
\end{equation}
of effective $1$-motives (which is also exact in $\M[1/p]$ by
Corollary~\ref{corexseq}). Summarizing up from \eqref{ngexseq},
\eqref{stexseq} we then get a long exact sequence in ${}^t\M^{\tor}[1/p]$
\begin{equation}\label{nexseq}
0 \to {}_nM \to {}_nL[1]\to {}_nG[1] \to M/n \to L/n[1] \to 0.
\end{equation}

Let now be $n = \ell^{\nu}$ where $\ell \neq {\rm char} (k)$. Set:

\begin{defn}\label{ladic} The {\it $\ell$-adic realisation}\, of a
$1$-motive
$M$ is $$T_{\ell}(M)\df ``\liminv{\nu}" L_\nu$$
in the category of $l$-adic sheaves, where $M/\ell^\nu=[L_\nu\to 0]$.
\end{defn}

Since the inverse system $``\varprojlim_{\nu}" {}_{l^\nu} L$ is Mittag-Leffler
trivial, we obtain a short exact sequence
$$0 \to T_{\ell}(G) \to T_{\ell}(M) \to L\otimes \Z_{\ell}
\to 0$$ 
where $T_{\ell}(G)$ is the Tate module of the semiabelian
variety $G$. More generally, using Corollary \ref{corexseq}, we have:

\begin{lemma}\label{ltlexact} The functor $T_\ell$ is exact on ${}^t\M[1/p]$,
and extends canonically to ${}^t\M\otimes\Z_\ell$.\qed
\end{lemma}

\subsection{Deligne $1$-motives} Let ${}^t\M^{\fr}[1/p]$, ${}^t\M^{\tor}[1/p]$ and 
${}^t\M^{\tf}[1/p]$ denote the corresponding full sub-categories of 
${}^t\M[1/p]$ given by free, torsion and torsion-free effective $1$-motives
respectively. The following $M\mapsto M_{\fr}$ (resp. $M\mapsto M_{\tor}$) define functors
from ${}^t\M[1/p]$ to ${}^t\M^{\fr}[1/p]$ (resp. from ${}^t\M[1/p]$ to ${}^t\M^{\tor}[1/p]$). We have (\cf \cite[(1.1.3)]{BRS}):

\begin{propose}\label{free} The natural functor 
\[\M[1/p]\to {}^t\M[1/p]\] 
from Deligne
$1$-motives to $1$-motives with torsion has a left adjoint/left inverse given by $M\mapsto
M_{\fr}$. In particular, it is fully faithful and makes $\M[1/p]$ an exact sub-category of ${}^t\M[1/p]$ in the sense of Quillen \cite[\S 2]{Q}. The above left adjoint defines equivalences 
$$\M[1/p]\cong{}^t\M^{\fr}[1/p]\cong {}^t\M^{\tf}[1/p].$$
\end{propose}

\begin{proof} Consider an effective map $(f, g):\tilde M
\to M'$, to a free $1$-motive $M'$, and a \qi $\tilde M \to M$, \ie $M
= [\tilde L/F \to \tilde G/F]$ for a finite group $F$. Since $M'$ is free
then $F$ is contained in the kernel of $f$ and the same holds for $g$. Thus $(f, g)$ induces an effective map $ M\to M'$. Let $M = [L\by{u}
G]$. Then $L_{\tor}\subseteq \ker (f)$ and also $u(L_{\tor})\subseteq \ker
(g)$ yielding an effective map $(f, g): M_{\fr}\to M'$. This proves the first assertion.

Since ${}^t\M^{\eff,\fr}\into {}^t\M^{\eff,\tf}$, the claimed equivalence is
obtained from the canonical \qi $M\to M_{\fr}$ for $M\in {}^t\M^{\eff,\tf}$, see
\eqref{basic}. 
Finally, consider the exact sequence \eqref{exseq} of $1$-motives with torsion such that
$M'_{\tor}=M''_{\tor}=0$. Since $M_{\tor}$ is mapped to zero in $M''$, it
injects in $M'$. Thus also $M$ is torsion-free, \ie $M=M_{\tf}$, and
quasi-isomorphic to $M_{\fr}$.
\end{proof}

\begin{remark} We also clearly have that the functor $M\mapsto
M_{\tor}$ is a right-adjoint to the embedding ${}^t\M^{\tor}[1/p]\into {}^t\M[1/p]$ , \ie 
$$\Hom_{\eff}(M, M'_{\tor}) \cong \Hom (M, M')$$
for $M\in{}^t\M^{\tor}[1/p]$ and $M'\in {}^t\M[1/p]$.
\end{remark}

\begin{cor}\label{isofree} We have ${}^t\M^{\tor}\otimes \Q =0$ and 
the full embedding $\M[1/p]\to {}^t\M[1/p]$ induces an equivalence 
$$\M\otimes \Q\iso {}^t\M\otimes \Q.$$
\end{cor}

\subsection{Non-connected $1$-motives}\label{noncon}
We consider a larger category allowing non-connected (reduced) group schemes as a supplement
of Proposition~\ref{isoab}.

\begin{defn} Let $\ncM^{\eff}$ denote the following category. The objects are
$N = [L \to G]$ complexes of \'etale sheaves over the field $k$ where $L$ is discrete and $G$ is a reduced group scheme locally of finite type over $k$ such that 
\begin{itemize}
\item[{\rm (i)}] the
connected component of the identity $G^0$ is semiabelian, and 
\item[{\rm (ii)}] $\pi_0 (G)$ is finitely generated.
\end{itemize}
The morphisms are just maps of complexes. 
We call $\ncM^{\eff}$ the category of {\it effective non-connected}\, 1-motives.

We denote $\anM^{\eff}$ \index{$\anM^{\eff}$} the full subcategory of $\ncM^{\eff}$ whose objects are  $N = [L \to G]$ as above such that $G$ is of finite type over $k$ (then
condition (ii) is automatically granted and $\pi_0 (G)$ is a finite
group scheme). We call $\anM^{\eff}$ the category of {\it algebraic effective
non-connected}\, 1-motives.
\end{defn}

Note that a representable presheaf on the category of schemes
over $k$ can be characterised by axiomatic methods, including the condition 
(i) above, \cf the Appendix in \cite{BSAP}.

Associated to $N =[L \to G]$ we have the following diagram
\[
\begin{CD}
0&  &0 \\
@V{}VV @V{}VV  \\
  L^0  @>>> G^0\\
@V{}VV @V{}VV  \\
L @>>> G\\
@V{}VV @V{}VV  \\
L/L^0  @>{\subseteq}>> \pi_0(G) \\
@V{}VV @V{}VV  \\
0&  &0 
\end{CD}
\]
here $L^0$ denote the pull-back of $G^0$ along $L \to G$. 
Let $$ N^0 \df [L^0 \to G^0]$$ denote the effective 1-motive associated to
$N$ and denote
$$ \pi_0(N)\df [L/L^0 \into \pi_0(G)].$$
We say that $D = [L \into L']$ is {\it discrete}\, if $ L'$ is a discrete sheaf
and $L$ injects into $L'$. Denote $\ncM^{\dis}$
the full subcategory of $\ncM^{\eff}$ given by discrete objects. Note that $N$ is
discrete if and only if $\pi_0(N) = N$ (if and only if $N^0 = 0$).

\begin{propose}\label{condis} The operation $N\mapsto N^0$ 
defines a functor $$c^{\eff}: \ncM^{\eff}\to {}^t\M^{\eff}$$ which is right adjoint to
the embedding $i_1 : {}^t\M^{\eff}\into \ncM^{\eff}$ and $c^{\eff}i_1 =1$. Moreover, we
have a functor
$$\pi_0:\ncM^{\eff}\to\ncM^{\dis}$$
which is left adjoint to $i_{\dis}: \ncM^{\dis}\into {}^t\M^{\eff}$ and
$\pi_0i_{\dis}=1$.
\end{propose}
\begin{proof} Straightforward.\end{proof}
\begin{remarks} 1) The same results as in Proposition~\ref{condis} above refine
to $\anM^{\eff}$ and $\anM^{\dis}$.\\
2) Note that $\ncM^{\eff}$ has kernels. Let $\phi = (f,g): N\to N'$ be a map in
$\ncM^{\eff}$. Let $g^0: G^0\to G^{'0}$ and $\pi_0(g) : \pi_0(G)\to
\pi_0(G')$ be the induced maps. Then $\ker (\phi) = [\ker (f)
\to \ker (g)]$ as a complex of sheaves; in fact $\ker (g)$ is representable,
$\ker^0(g^0) = (\ker (g))^0$ and $\pi_0 (\ker (g))$ maps to $\ker (\pi_0(g) )$ with
finite kernel. However, it is easy to see that $\ncM^{\eff}$ is not abelian.
\end{remarks}

\begin{propose} \label{nca}
The category $\anM^{\eff}[1/p]$ is abelian.
\end{propose} 
\begin{proof}  Regard $\anM^{\eff}[1/p]$ as a full subcategory of
$C^{[-1,0]}(\Shv(k_{\et}))[1/p]$ of the abelian category of complexes of sheaves
concentrated in degree $-1$ and $0$. For a map 
$\phi = (f,g):N\to N'$, $\ker (\phi) = [\ker (f)
\to \ker (g)]\in \anM^{\eff}$ and $\coker (\phi) = [\coker (f) \to \coker (g)]\in
\anM^{\eff}$. For an extension $0\to N\to N'\to N''\to 0$ in
$C^{[-1,0]}(\Shv(k_{\et}))$ such that 
$N$ and $N''$ belongs to $\anM^{\eff}$ then also $N''\in \anM^{\eff}$.
\end{proof}

\subsection{Homs and Extensions} We will provide a characterisation
of the Yoneda $\Ext$ in the abelian category ${}^t\M[1/p]$. 

\begin{propose}\label{hom} We have
\begin{itemize}
\item[(a)] $\Hom_{{}^t\M} (L[1],L'[1]) = \Hom_k (L,L')$,
\item[(b)] $\Hom_{{}^t\M} (L[1],G') =0$,
\item[(c)] $\Hom_{{}^t\M} (G,G') \subseteq \Hom_k (A,A')\times
\Hom_k(T,T')$ if $G$ (resp. $G'$) is an extension of an abelian
variety $A$ by a torus $T$ (resp. of $A'$ by $T'$),
\item[(d)] $\Hom_{{}^t\M} (G, L'[1])=\Hom_k({}_nG,L'_{\tor})$ if
$nL'_{\tor}=0$.
\end{itemize}
In particular, the group $\Hom_{{}^t\M}(M, M')$ is finitely generated
for all $1$-motives
$M, M'\in {}^t\M[1/p]$. 
\end{propose}

\begin{sloppypar}
\begin{proof} Since there are no \qi to $L[1]$, we have
$\Hom_{{}^t\M}(L[1],L'[1])=\Hom_{\eff}(L[1],L'[1])$ and the latter is
clearly isomorphic to $\Hom_k(L,L')$. By Proposition \ref{free}, we
have
$\Hom_{{}^t\M} (L[1],G') = \Hom_{\eff} (L[1],G')$ and 
$\Hom_{{}^t\M} (G, G') =  \Hom_{\eff}  (G, G')$. The former is clearly
$0$ while $\Hom_{\eff}(G,G')=\Hom_k(G,G')\subseteq
\Hom_k (A,A')\times
\Hom_k (T,T')$ since $\Hom_k(T,A')=\Hom_k(A,T')=0$. For (d),
let $[F \to \tilde G] \to [0 \to G]$ be a \qi and $[F \to \tilde G]
\to [L'\to 0]$ be an effective map providing an element of $\Hom
(G, L'[1])$. If
$L'$ is free then it yields the zero map, as $F$ is torsion. Thus
$\Hom_{{}^t\M}(G, L'[1]) = \Hom_{{}^t\M}(G, L'_{\tor}[1])$. For $n\in \N$
consider the short exact sequence \eqref{ngexseq} in ${}^t\M$. If $n$
is such that
$nL'_{\tor} =0$ taking $\Hom_{{}^t\M}(-,L'_{\tor}[1])$ we further
obtain $\Hom_{{}^t\M}(G, L'[1]) = \Hom_k ({}_nG, L'_{\tor})$.

The last statement follows from these computations and
an easy d\'evissage from \eqref{stexseq}.
\end{proof}
\end{sloppypar}

\begin{remark} If we want to get rid of the integer $n$ in (d), we
may equally write
\[\Hom_{{}^t\M}(G,
L'[1])=\Hom_{\text{cont}}(\hat{T}(G),L'_{\tor})=
\Hom_{\text{cont}}(\hat{T}(G),L')\]
where $\hat{T}(G)=\prod_\ell T_\ell(G)$ is the complete Tate module of
$G$.
\end{remark}

\begin{propose}\label{ext} We have isomorphisms (for $\Ext$ in ${}^t\M[1/p]$):
\begin{itemize}
\item[(a)] $\Ext^1_k (L,L')\iso\Ext^1_{{}^t\M} (L[1],L'[1])$,
\item[(b)] $\Hom_k (L, G')\iso\Ext^1_{{}^t\M} (L[1],G')$,
\item[(c)] $ \Ext^1_k (G,G')\iso\Ext^1_{{}^t\M} (G,G')$ and 
\item[(d)] $\displaystyle\limdir{n}
\Ext^1_k ({}_nG,L')\iso\Ext^1_{{}^t\M} (G, L'[1])$; these two groups
are $0$ if $L'$ is torsion. 
\end{itemize}
\end{propose}

\begin{proof} By Corollary \ref{corexseq}, any short exact sequence
of $1$-motives can be represented up to isomorphism by a short exact
sequence of complexes in which each term is an effective $1$-motive.

For (a), just observe that there are no nontrivial \qi of
$1$-motives  with zero semiabelian part. For (b), note that an
extension of $L[1]$ by $G'$ is given by a diagram 
\begin{equation}\label{ex1} \begin{CD}
0@>>> F' @>{}>> L''@>>> L@>>>  0\\
&&@V{}VV@V{v}VV@V{}VV  \\
0@>>> \tilde G' @>{}>> G''@>>> 0@>>>  0
\end{CD}
\end{equation}
where $\tilde M' = [F'\to \tilde G']$ is \qi to $[0\to G']$. When
$F'=0$, this diagram is equivalent to the datum of $v$: this
provides a linear map $\Hom_k (L, G')\to \Ext^1_{{}^t\M} (L[1],G')$.
This map is surjective since we may always mod out by $F'$ in
\eqref{ex1} and get an quasi-isomorphic exact sequence with $F'=0$.
It is also injective: if \eqref{ex1} (with $F'=0$) splits in ${}^t\M[1/p]$,
it already splits in ${}^t\M^{\eff}[1/p]$ and then $v =0$. 

For (c) we see that an extension of $G$ by $G'$ in ${}^t\M[1/p]$ can be
represented by a diagram 
\[ \begin{CD}
0@>>> F' @>{}>> L''@>>> F@>>>  0\\
&&@V{}VV@V{}VV@V{}VV  \\
0@>>> \tilde G' @>{}>> G''@>>> \tilde G@>>>  0
\end{CD}\]
with $\tilde M'$ as in (b) and $\tilde M = [F\to \tilde
G]$ \qi to $[0\to G]$. Since the top line is exact, $L''$ is
finite. For $F=F'=0$ we just get a group scheme extension of $G$ by
$G'$, hence a homomorphism $\Ext^1_k(G,G')\to \Ext^1_{{}^t\M}(G,G')$.
This homomorphism is surjective:  dividing by
$F'$ we get a quasi-isomorphic exact sequence
\[ \begin{CD}
0@>>> 0 @>{}>> L''/F'@>{\sim}>> F@>>>  0\\
&&@V{}VV@V{}VV@V{}VV  \\
0@>>> G' @>{}>> G''/F'@>>> \tilde G@>>>  0
\end{CD}\]
and further dividing by $F'$ we then obtain 
\[ \begin{CD}
0@>>> 0 @>{}>> 0@>>> 0@>>>  0\\
&&@V{}VV@V{}VV@V{}VV  \\
0@>>> G' @>{}>> G''/L''@>>> G@>>>  0.
\end{CD}\]

Injectivity is seen as in (b).

For (d) we first construct a map $\Phi_n:\Ext^1_k({}_nG,L') \to
\Ext^1_{{}^t\M}(G,L'[1])$ for all $n$.  Let $[L'']\in \Ext_k
({}_nG,L')$ and consider the following diagram
\begin{equation}\label{ex2} 
\begin{CD}
0@>>> L' @>{}>> L''@>>> {}_nG@>>>  0\\
&&@V{}VV@V{}VV @V{}VV  \\
0@>>> 0 @>{}>> G@= G@>>>  0.
\end{CD}\end{equation}

Since $[{}_nG \to G]$ is \qi to $[0\to G]$, this provides an
extension of $G$ by $L'[1]$ in ${}^t\M[1/p]$. For $n$ variable $\{\Ext^1_k
({}_nG,L')\}_n$ is a direct system and one checks easily that the
maps $\Phi_n$ are compatible (by pull-back), yielding a
well-defined linear map
\[\Phi:\varinjlim \Ext^1_k ({}_nG,L')\to\Ext^1_{{}^t\M} (G,
L'[1]).\] 

This map is surjective since any extension of $G$ by $L'[1]$ can be
represented by a diagram \eqref{ex2} for some $n$ (as
multiplication by $n$ is cofinal in the direct sytem of isogenies).
We now show that $\Phi$ is also injective. 

Let $n\mid m$, \eg  $rn =m$, so that the following sequence is exact
$$0\to {}_rG\to {}_mG\longby{r} {}_nG\to 0$$
and yields a long exact sequence
$$\Hom_k ({}_mG,L')\to \Hom_k ({}_rG,L')\to \Ext^1_k
({}_nG,L')\longby{r}
\Ext^1_k ({}_mG,L').$$

If $L'$ is torsion, we have $rL'=0$ for some $r$, hence  
$\varinjlim \Ext^1_k ({}_nG,L')=0$. This shows in particular that
$\Ext^1_{{}^t\M}(G,L'[1])=0$ in this case. 

Suppose now that $L'$ is free. Then we have
$\Hom_k ({}_rG,L')=0$, hence the transition maps are injective.
Therefore, to check that
$\Phi$ is injective it suffices to check that $\Phi_n$ is injective
for all $n$. 

Let $\sigma:G\to [L''\to G]$ be a section of \eqref{ex2} in ${}^t\M[1/p]$.
Then $\sigma$ can be represented by a diagram of effective maps
\[ 
\begin{CD}
&&&&&&{}_{rn}G\\
&&&&&\swarrow&\Big\downarrow &\scriptstyle
r\displaystyle\searrow\\ 
0@>>> L' @>{}>> L''&@>>> &{}_nG\\
&&@V{}VV@V{}VV G&&@V{}VV\\
&&&&&\swarrow&&\scriptstyle r\displaystyle\searrow\\
0@>>> 0 @>{}>> G&@= &G
\end{CD}\]
for some $r$, where the southwest map is a \qi To say that $\sigma$
is a section is to say that this diagram commutes. Hence the image
of
${}_{rn}G$ in $L''$ surjects onto ${}_nG$, and it also injects
since $L'$ is torsion-free. This means that the projection $L''\to
{}_nG$ has a section, hence $[L'']=0$ in $\Ext^1_k({}_nG,L'')$.

For a general $L'$, we reduce to these two special cases through an
easy diagram chase.
\end{proof}

\subsection{Projective objects in ${}^t\M[1/p]$}
We show that there are not enough projective objects in ${}^t\M[1/p]$, at least when $k$ is
algebraically closed:

\begin{propose}\label{proj} Suppose that $k=\bar k$.  Then the only projective object of
${}^t\M[1/p]$  is $0$.
\end{propose}

\begin{proof} Suppose that $M =[L\to G]\in {}^t\M[1/p]$ is such that $\Ext (M, N)\allowbreak
=0$ for any $N\in {}^t\M[1/p]$. From \eqref{stexseq}
we then get a long exact sequence
\begin{multline*}
\Hom (G,\G_m)\to\Ext (L[1],\G_m)\to\Ext (M,\G_m)\\
\to\Ext (G,\G_m)\to
\Ext^2 (L[1],\G_m)
\end{multline*}
where $\Ext (M,\G_m)=0$, thus {\it i)}\,  $\Ext (G,\G_m)$ is finite, and 
{\it ii)}\,  $\Hom (L,\G_m)$ is finitely generated.
We also have an exact sequence
$$\Hom (T,\G_m)\to\Ext (A,\G_m)\to\Ext (G,\G_m)$$
where $\Hom (T,\G_m)$ is the character group of the torus $T$ and $\Ext
(A,\G_m)$ is the group of $k$-points of the dual abelian variety $A$, the
abelian quotient of $G$. From {\it i)}\, we get $A=0$. Since  $\Hom
(L,\G_m)$ is an extension of a finite group by a divisible group, from {\it
ii)}\, we get that $L$ is a finite group.
Now consider the exact sequence, for $l\ne p$
\begin{multline*}
0\to\Hom (L[1],\Z/l[1])\to\Hom (M,\Z/l[1])\\
\to\Hom (T,\Z/l[1])\to 
\Ext (L[1],\Z/l[1])\to 0
\end{multline*}
where the right-end vanishing is $\Ext (M,\Z/l[1])=0$ by assumption. Now 
$\Hom (T,\Z/l[1])= \Ext (T,\Z/l)$ and any extension of the torus $T$ is
lifted to an extension of $M$ by $\Z/l$, therefore to an element of 
$\Hom (M,\Z/l[1])\allowbreak = \Ext (M,\Z/l)$. This yields $\Ext
(L,\Z/l)=0$ for any prime $l\ne p$, thus we see that $L=0$.

Finally, $[0\to \G_m]$ is not projective since, for $n>1$, the epimorphism
\[[0\to \G_m]\by{n}[0\to \G_m]\]
is not split.
\end{proof}


\subsection{Weights}
If $M = [L\by{u}G]\in {}^t\M[1/p]$ is free then Deligne \cite{D} equipped $M = M_{\fr}$ with an
increasing filtration by sub-$1$-motives as follows:
$$W_{-2}(M) \df [0\to T]\subseteq W_{-1}(M)\df [0\to G] \subseteq W_{0}(M) \df M$$

If $M$ is torsion-free we then pull-back the weight filtration along the 
effective map $M \to M_{\fr}$ as follows:
$$W_i(M) \df \left \{\begin {array}{cl} M & i\geq 0\\{} 
[L_{\tor}\into G]& i= -1\\{}
[L_{\tor}\cap T\into T] & i= -2\\ 0 & i \leq -3 \end{array} \right. $$
Note that $W_i(M) $ is \qi to $W_i(M_{\fr})$.

If $M$ has torsion we then further pull-back the weight filtration along the 
effective map $M \to M_{\tf}$. 
\begin{defn}
{\rm Let $M = [L\by{u}G]$ be an effective $1$-motive. Let $u_A : L \to A$
denote the induced map where $A = G/T$. Define
$$W_i(M) \df \left \{\begin {array}{cl} M & i\geq 0\\{} 
[L_{\tor}\to G]& i= -1\\{}
[L_{\tor}\cap \ker (u_A)\to T] & i= -2\\{}
M_{\tor} = L_{\tor}\cap \ker (u)[1] & i = -3\\
0 & i \leq -4 \end{array} \right.$$}
\end{defn}

\begin{remark}It is easy to see that $M\mapsto W_i(M)$ yields a functor
from ${}^t\M[1/p]$ to ${}^t\M[1/p]$. 
\end{remark}

\section{Homotopy invariance for \'etale sheaves with transfers}

One of the main results of Voevodsky concerning presheaves with transfers is that, over a
perfect field
$k$, a Nisnevich sheaf with transfers $F$ is homotopy invariant (that is, $F(X)\iso F(X\times
\Aff^1)$ for any smooth
$X$) if and only if it is strongly homotopy invariant, that is, $H^i_\Nis(X,F)\iso
H^i_\Nis(X\times
\Aff^1,F)$ for any smooth $X$ and any $i\ge 0$. This allows him to define the \emph{homotopy
$t$-structure} on $\DM_-^\eff$.

These results remain ``as true as can be" in the \'etale topology, at least if $k$ has finite
\'etale cohomological dimension. According to an established tradition, this result is probably
well-known to experts but we haven't been able to find it in the literature: it could have been
formulated and proven for example in \cite{VL}. The aim of this appendix is to provide proofs,
for which our main source of results will be \cite{VL}.

\subsection{Homotopy invariance and strict homotopy invariance}

\begin{defn} We denote as in \cite[Def. 2.1]{VL} by $\PST(k) =\PST$ the category of presheaves
with transfers on smooth $k$-varieties. We also denote by $\EST(k)$, or simply $\EST$, the
cat\'egory of \'etale sheaves with transfers over $k$.
\end{defn}

According to \cite[Def. 2.15 and 9.22]{VL}:

\begin{defn}\label{dD.1} a) An object $F$ of $\PST$ or $\EST$ is \emph{homotopy invariant} if
$F(X)\iso F(X\times \Aff^1)$ for any smooth $k$-variety $X$.\\
b) Let $F\in \EST$. Then $F$ is \emph{strictly homotopy invariant} if $H^i_\et(X,F)\iso
H^i_\et(X\times \Aff^1,F)$ for any smooth $k$-variety $X$ and any $i\ge 0$.\\
We denote by $\HI_\et(k)=\HI_\et$ the full subcategory of $\EST$ consisting of homotopy
invariant sheaves, and by $\HI_\et^s(k)=\HI_\et^s$ the full subcategory of $\HI_\et$ consisting
of strictly homotopy invariant sheaves.\index{$\HI_\et^s$}
\end{defn}

(Strict homotopy invariance for $F$ simply means that $F$ is $\Aff^1$-local in $D^{-}(\EST)$,
see \cite[Lemma 9.24]{VL}.)

Note that $\HI_\et$ is a thick abelian subcategory of $\EST$: if $0\to F'\to F\to F''\to 0$ is
an exact sequence in $\EST$, then $F\in \HI_\et$ if and only if $F',F''\in \HI_\et$. We shall
see below that the same is true for $\HI_\et^s$.

The main example of a sheaf $F$ which is in $\HI_\et$ but not in $\HI_\et^s$ is $F=\Z/p$ in
characteristic $p$: because of the Artin-Schreier exact sequence we have
\[k[t]/\cP(k[t])\iso H^1_\et(\Aff^1_k,\Z/p)\]
where $\cP(x) = x^p-x$.

We are going to show that this captures entirely the obstruction for a sheaf in $\HI_\et$ not
to be in $\HI_\et^s$.

The following is an \'etale analogue of \cite[Th. 13.8]{VL}:

\begin{lemma}\label{lD.1.3} Let $F$ be a homotopy invariant presheaf with transfers. Suppose
moreover that $F$ is a presheaf of $\Z[1/p]$-modules, where $p$ is the exponential
characteristic of $k$. Then the associated \'etale sheaf with transfers \cite[Th. 6.17]{VL}
$F_\et$ is strictly homotopy invariant.
\end{lemma}

\begin{proof} The following method is classical: let $0\to F'\to F\to F''\to 0$ be an exact
sequence of homotopy invariant presheaves with transfers, and consider the corresponding exact
sequence $0\to F'_\et\to F_\et\to F''_\et\to 0$. If, among $F'_\et,F_\et$ and $F''_\et$, two
are in $\HI_\et^s$, then clearly so is the third. Using the exact sequence
\[0\to F_{tors}\to F\to F\otimes\Q\to F\otimes\Q/\Z\to 0\]
for the sheaf $F$ of Lemma \ref{lD.1.3}, this reduces us to the following cases:
\begin{itemize}
\item $F$ is a presheaf of $\Q$-vector spaces. Then the result is true by \cite[Lemma
14.25]{VL} (reduction to \cite[Th. 13.8]{VL} by the comparison theorem \cite[Prop. 14.23]{VL}).
\item $F$ is a presheaf of torsion abelian groups. Since, by assumption, this torsion is prime
to $p$, $F_\et$ is locally constant by Suslin-Voevodsky rigidity \cite[Th. 7.20]{VL}. Then the
result follows from \cite[XV 2.2]{sga4} (compare \cite[Lemma 9.23]{VL}).
\end{itemize}
\end{proof}

\begin{propose}\label{pD.1.4} The inclusion $\HI_\et^s\to \HI_\et$ has an exact
left adjoint/left inverse given by $F\mapsto F\otimes_\Z \Z[1/p]$. In particular,
$\HI_\et^s=
\HI_\et[1/p]$.
\end{propose}

\begin{proof}  The fact that $\HI_\et^s$ is $\Z[1/p]$-linear follows from the Artin-Schreier
exact sequence plus the contractibility of $\G_a$ (compare \cite[Prop. 3.3.3 2)]{V}).
Conversely, Lemma \ref{lD.1.3} implies that any $\Z[1/p]$-linear sheaf of $\HI_\et$ belongs
to $\HI_\et^s$. The rest of the proposition follows.
\end{proof}

As a complement, let us mention the following proposition, which extends Proposition
\ref{pD.1.4}:

\begin{propose}\label{pD.2} Let $F\in \HI_\et$.
Then the complex
$C_*(F)$ is canonically isomorphic to $F[1/p]$ in $\DM_{-,\et}^\eff$.
\end{propose}

\begin{proof} The map $F\to F[1/p]$ induces a map $C_*(F)\to C_*(F[1/p])$. The
latter complex is tautologically equal to $C_*(F)[1/p]$. Since $\DM_{-,\et}^\eff$ is
$\Z[1/p]$-linear, the map $C_*(F)\to C_*(F)[1/p]$ is a quasi-isomorphism. Finally, since
$F[1/p]\in \HI_\et^s$ (Prop. \ref{pD.1.4}), the augmentation $F[1/p]\to C_*(F[1/p])$ is a
quasi-isomorphism by \cite[Lemma 9.15]{VL}.
\end{proof}

\begin{cor}\label{cD.1} Let $F$ be a homotopy invariant Nisnevich sheaf with transfers. Then,
the natural functor $\alpha^*:\DM_-^\eff\to \DM_{-,\et}^\eff$ sends $F$ to
$F_\et[1/p]$.
\end{cor}

\begin{proof} According to \cite[Remark 14.3]{VL}, $\alpha^*$ may be described as the
composition
\[\DM_-^\eff\into
D^-(\Shv_\Nis(Sm(k)))\longby{\alpha^*}D^-(\Shv_\et(Sm(k)))\longby{RC}\DM_{-,\et}^\eff\]
where the middle functor is induced by the inverse image functor (change of topology) on sheaves
and $RC$ is induced by $K\mapsto C_*(K)$. The result then follows from Proposition \ref{pD.2}.
\end{proof}

\subsection{Friendly complexes}

\begin{defn}\label{dfr} A object $C\in D^{-}(\EST)$ is \emph{friendly} if there exists an
integer
$N=N(C)$ such that, for any prime number $l\ne p$, $H_q(C/l)=0$ for $q>N$ (in other terms,
$C/l$ is uniformly bounded below). We denote by $D^{-}_\fr(\EST)$ the full subcategory of
friendly objects and by $\DM_{\fr,\et}^\eff$ the intersection $D^-_\fr(\EST)\cap
\DM_{-,\et}^\eff$.
\end{defn}

\begin{thm}\label{tfr} $\DM_{\gm,\et}^\eff\subset \DM_{\fr,\et}^\eff$.
\end{thm}

\begin{proof} If is clear that $D^{-}_\fr(\EST)$ is a thick triangulated subcategory of
$D^-(\EST)$; hence it suffices to prove that $C_*(L_\et(X))$ is friendly for any smooth scheme
$X$. By \cite[Lemmas 6.23 and 9.15]{VL}, we have for any smooth $U$, any prime $l\ne p$ and any
$q\in \Z$
\[\Hom(C_*(L_\et(X\times U)),\Z/l[q])\simeq H^q_\et(X\times U,\Z/l)\]
hence
\[\sext^q(C_*(L_\et(X)),\Z/l)\simeq R^q_\et \pi_*\Z/l\]
where $\pi:X\to \Spec k$ is the structural morphism. By the cohomological dimension results
of \cite[Exp. X]{sga4} and the finiteness results of \cite[Th. finitude]{sga4 1/2}, this shows
that
$\ihom_\et(C_*(L_\et(X)),\Z/l)$ is a bounded complex of constructible $\Z/l$-sheaves. It
follows that the biduality morphism
\[C_*(L_\et(X))/l\to \ihom_\et(\ihom_\et(C_*(L_\et(X)),\Z/l),\Z/l)\]
is an isomorphism of bounded complexes of constructible $\Z/l$-sheaves. Moreover, the lower
bound is at most $2\dim X$, hence is independent of
$l$.
\end{proof}

\subsection{The \'etale homotopy $t$-structure}

The following is an \'etale analogue of \cite[Prop. 14.8]{VL}:

\begin{propose} Let $K\in
D^{-}(\EST)$ be a bounded above complex of \'etale sheaves with transfers. Suppose either
that the \'etale cohomological dimension of $k$ is finite, or that $K$ is friendly. Then $K$ is
$\Aff^1$-local if and only if all its cohomology sheaves are strictly homotopy invariant.
\end{propose}

\begin{proof} In the finite cohomological dimension case, ``if" is trivial (\cf \cite[Prop.
9.30]{VL}). For ``only if", same proof as that of \cite[Prop. 14.8]{VL}, by replacing the
reference to \cite[Th. 13.8]{VL} in \loccit by a reference to Lemma \ref{lD.1.3} (note that if
$K$ is $\Aff^1$-local, then it is $\Z[1/p]$-linear by \cite[Prop. 3.3.3 2)]{V} and thus so are
its cohomology sheaves).

In the friendly case, note that the two conditions
\begin{itemize}
\item $\Aff^1$-local
\item having strictly homotopy invariant cohomology sheaves
\end{itemize}
are stable under triangles: for the first it is obvious and for the second it is because
$\HI_\et^s$ is thick in $\EST$ by Proposition \ref{pD.1.4}. Considering the exact triangle
\[K\to K\otimes \Q\to K\otimes (\Q/\Z)'\by{+1}\]
we are reduced to show the statement separately for $K\otimes \Q$ and for $K\otimes (\Q/\Z)'$.
In the first case this works by reduction to Nisnevich cohomology, while in the second case the
spectral sequence of \cite[Prop. 9.30]{VL} also converges, this time because $K\otimes (\Q/\Z)'$
is bounded below.
\end{proof}

\begin{remark} The finite cohomological dimension hypothesis appears in the spectral sequence
arguments of the proofs of \cite[Prop. 9.30 and 14.8]{VL}. We don't know
if it is really necessary. Nevertheless, Jo\"el Riou pointed out that this argument trivially
extends to fields of virtually finite cohomological dimension: the only issue is for the ``if"
part, but if we know that an object $K$ is $\Aff^1$-local \'etale-locally, then it is clearly
$\Aff^1$-local. (For example, this covers all fields of arithmetic origin.) Therefore:
\end{remark}

\begin{cor}\label{cD.2} If the virtual \'etale cohomological dimension of $k$ is finite, then
$\DM_{-,\et}^\eff$ has a homotopy $t$-structure, with heart $\HI_\et^s$, and the functor
$\alpha^*:\DM_-^\eff\to \DM_{-,\et}^\eff$ is $t$-exact. Without any cohomological
dimension assumption, $\DM_{\fr,\et}$ has a homotopy $t$-structure.\qed
\end{cor}

\printindex

\end{document}